\documentclass[reqno]{amsart}

\usepackage{amssymb}
\usepackage{geometry, enumitem}

\usepackage{mathtools}
\usepackage{mathrsfs}
\usepackage{hyperref}
\usepackage{cleveref}

\newtheorem{thm}{Theorem}[section]
\newtheorem{prop}[thm]{Proposition}
\newtheorem{lem}[thm]{Lemma}
\newtheorem{cor}[thm]{Corollary}
\crefname{thm}{Theorem}{Theorems}
\crefname{prop}{Proposition}{Propositions}
\crefname{lem}{Lemma}{Lemmas}
\crefname{cor}{Corollary}{Corollaries}

\newtheorem{prob}{Problem}
\crefname{claim}{Claim}{Claims}
\crefname{step}{Step}{Steps}

\theoremstyle{definition}
\newtheorem{definition}[thm]{Definition}
\newtheorem{defn}[thm]{{Definition}}
\newtheorem{example}[thm]{Example}
\crefname{definition}{Definition}{Definitions}
\crefname{example}{Example}{Examples}

\theoremstyle{remark}
\newtheorem{remark}[thm]{{\bf {\em Remark}}}
\crefname{remark}{Remark}{Remarks}

\numberwithin{equation}{section}

\newcommand{\pair}[2]{\ensuremath\langle#1,#2\rangle}

\newcommand{\compl}{^\mathrm{c}}
\newcommand{\defeq}{\vcentcolon=}
\newcommand{\trinorm}[1]{\ensuremath|\!|\!|#1|\!|\!|}

\renewcommand{\leq}{\leqslant}
\renewcommand{\geq}{\geqslant}
\newcommand{\ssubset}{%
\Subset}

\renewcommand{\phi}{\varphi}

\renewcommand{\epsilon}{\varepsilon}

\newcommand{\tildee}[1]{\widetilde{#1}}

\DeclareMathOperator{\LSC}{LSC}
\DeclareMathOperator{\USC}{USC}

\DeclareMathOperator{\intr}{int}
\DeclareMathOperator{\Aff}{Aff}

\newcommand{\gf}{\mathfrak{g}}

\newcommand{\N}{\mathbb{N}}

\newcommand{\de}{\partial}

\newcommand{\R}{{\mathbb{R}}}

\newcommand{\call}[1]{\ensuremath\mathcal{#1}}
\newcommand{\frk}[1]{\ensuremath\mathfrak{#1}}
\newcommand{\scr}[1]{\ensuremath\mathscr{#1}}

\renewcommand{\bar}[1]{\ensuremath\overline{#1}}
\newcommand{\dto}{\ensuremath{\searrow}}
\newcommand{\uto}{\ensuremath{\nearrow}}

\newcommand{\veps}{\varepsilon}
\newcommand{\SAff}{\mathrm{SA}}
\newcommand{\SA}{\mathrm{S}\cA}
\newcommand{\cA}{\mathcal{A}}
\newcommand{\cJ}{\mathcal{J}}

\newcommand{\scK}{\mathscr{K}}

\newcommand{\cF}{\mathcal{F}}

\newcommand{\cG}{\mathcal{G}}
\newcommand{\cH}{\mathcal{H}}
\newcommand{\scH}{\mathscr{H}}
\newcommand{\cN}{\mathcal{N}}
\newcommand{\cM}{\mathcal{M}}
\newcommand{\cP}{\mathcal{P}}
\newcommand{\cQ}{\mathcal{Q}}
\newcommand{\cD}{\mathcal{D}}

\newcommand{\wt}{\widetilde}

\newcommand{\cS}{\mathcal{S}}

\makeatother

\newcommand{\changelocaltocdepth}[1]{%
	\addtocontents{toc}{\protect\setcounter{tocdepth}{#1}}%
	\setcounter{tocdepth}{#1}%
}

\newcommand{\noticina}[1]{\marginpar{\tiny\emph{#1}}}

\renewcommand{\marginpar}[2][]{}

\begin{document}

\date{\today} \linespread{1.2}

\title[Comparison principles by fiberegularity and monotonicity]{Comparison principles for nonlinear potential theories and PDEs with fiberegularity and sufficient monotonicity}
\author{Marco Cirant}
\address{Dipartimento di Matematica ``T.\ Levi-Civita''\\ Università degli Studi di Padova\\ \newline Via Trieste 63\\ 35121--Padova, Italy}
\email{cirant@math.unipd.it (Marco Cirant)}
\author{Kevin R.\ Payne}
\address{Dipartimento di Matematica ``F.\ Enriques''\\ Università degli Studi di Milano\\ \newline Via~C.~Saldini~50\\ 20133--Milano, Italy}
\email{kevin.payne@unimi.it (Kevin R. Payne)} 
\author{Davide F.\ Redaelli}
\address{Dipartimento di Matematica ``T.\ Levi-Civita''\\ Università degli Studi di Padova\\ \newline Via Trieste 63\\ 35121--Padova, Italy}
\email{redaelli@math.unipd.it (Davide F.\ Redaelli)}


\subjclass[2010]{}

\makeatletter
\def\l@subsection{\@tocline{2}{0pt}{2.5pc}{5pc}{}}
\makeatother

\begin{abstract}
	We present some recent advances in the productive and symbiotic interplay between general potential theories (subharmonic functions associated to closed subsets $\mathcal{F} \subset \mathcal{J}^2(X)$ of the 2-jets on $X \subset \mathbb{R}^n$ open) and subsolutions of degenerate elliptic and parabolic PDEs of the form $F(x,u,Du,D^2u) = 0$. We will implement the {\em monotonicity-duality} method begun by Harvey and Lawson \cite{HL09} (in the pure second order constant coefficient case) for proving comparison principles for potential theories where $\mathcal{F}$ has {\em sufficient monotonicity} and {\em fiberegularity} (in variable coefficient settings) and which carry over to all differential operators $F$ which are {\em compatible} with $\mathcal{F}$ in a precise sense for which the {\em correspondence principle} holds. 
	
	We will consider both elliptic and parabolic versions of the comparison principle in which the effect of boundary data is seen on the entire boundary or merely on a proper subset of the boundary.
	
	Particular attention will be given to {\em gradient dependent} examples with the requisite sufficient monotonicity of {\em proper ellipticity} and {\em directionality} in the gradient. Examples operators we will discuss include the degenerate elliptic operators of {\em optimal transport} in which the target density is strictly increasing in some directions as well as operators which are weakly parabolic in the sense of Krylov. Further examples, modeled on {\em hyperbolic polynomials} in the sense of G\aa rding give a rich class of examples with directionality in the gradient. Moreover we present a model example in which the comparison principle holds, but standard viscosity structural conditions fail to hold. 
\end{abstract}

\maketitle

\setcounter{tocdepth}{1}
\tableofcontents

\changelocaltocdepth{2}

\makeatletter
\def\l@subsection{\@tocline{2}{0pt}{2.5pc}{5pc}{}}
\makeatother

\setcounter{tocdepth}{1}
\tableofcontents

\section{Introduction}\label{sec:intro}

In this work, we continue an investigation into the validity of the {\em comparison principle} 
\begin{equation}\label{CP:intro}
	u \leq w \ \text{on} \ \partial \Omega \ \ \Rightarrow \ \ 	u \leq w \ \text{in} \ \Omega 
\end{equation}
on bounded domains $\Omega$ in Euclidean spaces $\R^n$. We will operate in the two seemingly distinct frameworks of general second order {\em nonlinear potential theories} and of general second order {\em fully nonlinear PDEs}, where the formulations of {\em comparison} \eqref{CP:intro} in the two frameworks will soon be made precise. Comparison is of interest in both frameworks since, as is well known, it implies uniqueness of solutions to the natural Dirichlet problem in both frameworks (in the presence of some mild form of ellipticity). Moreover, comparison \eqref{CP:intro} together with suitable strict boundary convexity (that ensures the existence of needed barriers) leads to existence of solutions to the Dirichlet problem by Perron's method. 

In both frameworks we will also treat the following variant of the comparison principle
\begin{equation}\label{CPP:intro}
u \leq w \ \text{on} \ \partial ^-\Omega \ \ \Rightarrow \ \ 	u \leq w \ \text{in} \ \Omega, 
\end{equation}
where $\partial ^-\Omega \subsetneq \partial \Omega$ is a proper subset of the boundary. The version \eqref{CP:intro}  ``sees'' the entire boundary and will hold under weak ellipticity assumptions and hence we will refer to it as the {\em elliptic version} of comparison. On the other hand, the version \eqref{CPP:intro} which sees only a  ``reduced boundary'' will be refered to as the {\em parabolic version} of comparison since it holds (for example) under weak parabolicity assumptions. 

While both versions of comparison are seemingly different in the two frameworks, we will connect the two frameworks for both versions by something called the {\em correspondence principle} which gives precise conditions of {\em compatibility} for which comparison in the two frameworks is equivalent. This is important for many reasons. A given second order potential theory on an open subset $X \subset \R^n$ is determined by a {\em constraint set} $\cF$  which is closed subset of $\cJ^2(X)$ (the space of $2$-jets on $X$) and which identifies a class of {\em $\cF$-subharmonic functions} on $X$, while a PDE on $X$ is an equation of the form $F(x,J) = 0$ determined by an operator $F$ acting on $(x,J) \in \cJ^2(X)$. There may be many differential operators $F$ organized about $\cF$ which are compatible with the constraint set in the sense
$$
	\cF = \{(x,J) \in \cJ^2(X): F(x,J) \geq 0\} \quad \text{and} \quad  \cF = \{(x,J) \in \cJ^2(X): F(x,J) = 0\}.
$$
Hence potential-theoretic comparison for $\cF$ gives operator-theoretic comparison for \underline{every} operator $F$ compatible with $\cF$. This is just one instance of the productive interplay between potential theory and operator theory. See the survey paper \cite{HP22} and the preface of the monograph \cite{CHLP22} for a more complete discussion of this interplay. We will have more to say on the origins and development of this program below, after presenting the two formulations of comparison and the main results obtained here, which will allow us to clearly underline what is new in this paper.

We now describe comparison in the first (potential theoretic) framework. Here one asks when does comparison \eqref{CP:intro} hold on $\overline{\Omega}$ for each pair $u \in \USC(\overline{\Omega})$ and $w \in \LSC(\overline{\Omega})$ which are respectively $\cF-$subharmonic and $\cF-$superharmonic functions on $\Omega \ssubset X$ where 
\begin{equation*}\label{Subeq:intro}
\cF \subset \cJ^2(X) := X \times \cJ^2 = X \times \R \times \R^n \times \cS(n), \ \ X \subset \R^n \ \text{open},
\end{equation*}
is a {\em subequation (constraint set) on $X$} in the space of $2$-jets. The precise definitions of subequations $\cF$ and their associated subharmonics are given in Definitions \ref{defn:subeq} and \ref{defn:Fsub}, respectively. We note here only that $\cF$ is closed and satisfies certain natural (monotonicity and topological) axioms which ensure that the potential theory determined by $\cF$ (the associated $\cF$subharmonics) is meaningful and rich where the subharmonics on $\Omega$ satisfy the differential inclusion 
\begin{equation*}\label{Fsub:intro}
		J^{2,+}_x u \subset \cF_x := \{ J \in \cJ^2: (x,J) \in \cF\}, \ \ \forall \, x \in \Omega
\end{equation*}
in the viscosity sense where $J^{2,+}_x u$ is the set of {\em upper test jets for $u$ in $x$}. 

A general comparison principle is presented in Theorem \ref{thm:GCT} in this potential theoretic setting and the proof makes use of the {\em monotonicity-duality method} that was initiated in the constant coefficient pure second order case in Harvey-Lawson \cite{HL09}. To use the method, we require three additional assumptions. First, the constraint set $\cF$ much be sufficiently {\em monotone} in the sense that there exists a constant coefficient subequation $\cM \subsetneq \cJ^2$ which is a {\em monotonicity cone subequation} for $\cF$; that is, in addition to being a subequation it is also a convex cone with vertex at the origin such one has the monotonicity property
\begin{equation*}\label{Mmono:intro}
\cF_x + \cM \subset \cF_, \ \ \forall \, x \in X.
\end{equation*}
We also require that the constraint set $\cF$ is {\em fiberegular} in the sense that the fiber map
\begin{equation}\label{Freg:intro}
\Theta: X \to \scK(\cJ^2) \quad \text{defined by} \quad \Theta(x):= \cF_x, \ x \in X
\end{equation}
is continuous with respect to the Euclidian metric on $X$ and the Hausdorff metric on the closed subsets $\scK(\cJ^2)$ of $\cJ^2$. This notion was introduced in \cite{CP17} in the variable coefficient pure second order case and then was extended to the gradient-free case in \cite{CP21}. Finally, for the elliptic version of comparison \eqref{CP:intro} on $\Omega$, we require that the monotonicity cone $\cM$ admits a function $\psi \in C^2(\Omega) \cap \USC(\overline{\Omega})$  which is strictly $\cM$-subharmonic on $\Omega$. For the the parabolic version of comparison \eqref{CPP:intro} on $\Omega$,
we also require that $\psi$ blows up on the complement of the reduced boundary in the sense that
\begin{equation}\label{SA_parab}
\psi \equiv - \infty \quad \text{on}\  \partial \Omega \setminus \partial^-\Omega.
\end{equation}

The utility of the General Comparison Theorem \ref{thm:GCT} is greatly enhanced by exploiting  by the detailed study of monotonicity cone subequations in \cite{CHLP22}, which we briefly review. There is a three parameter {\em fundamental family} of monotonicity cone subequations (see Definition 5.2 and Remark 5.9 of \cite{CHLP22}) consisting of
\begin{equation*}\label{fundamental_family1:intro}
\cM(\gamma, \cD, R):= \left\{ (r,p,A) \in \cJ^2: \ r \leq - \gamma |p|, \ p \in \cD, \ A \geq \frac{|p|}{R}I \right\}
\end{equation*}
where
\begin{equation*}\label{fundamental_family2:intro}
\gamma \in [0, + \infty), R \in (0, +\infty] \ \text{and} \ \cD \subseteq \R^n \ \text{is a {\em directional cone}},
\end{equation*}
in the sense of Definition \ref{defn:property_D} below. The family is fundamental in the sense that for any monotonicity cone subequation, there exists an element $\cM(\gamma, \cD, R)$ of the familly with $\cM(\gamma, \cD, R) \subset \cM$ (see Theorem 5.10 of \cite{CHLP22}). Hence if $\cF$ is an $\cM$-monotone subequation, then it is $\cM(\gamma, \cD, R)$-monotone for some triple $(\gamma, \cD, R)$. Moreover, from Theorem 6.3 of \cite{CHLP22}, given any element $\cM = \cM(\gamma, \cD, R)$ of the fundamental family, one knows for which domains $\Omega \ssubset \R^n$ there is a $C^2$-strict $\cM$-subharmonic and hence for which domains $\Omega$ one has the comparison principle \eqref{CP:intro} in any potential theory determined by a fiberegular and $\cM$-monotone subequation $\cF$. This leads to the Fundamental Family Comparison Theorem \ref{ffctcs}. There is a simple dichotomy. If $R = + \infty$, then arbitrary bounded domains $\Omega$ may be used, while in the case of $R$ finite, any $\Omega$ which is contained in a translate of the truncated cone $\cD_R := \cD \cap B_R(0)$. 

Next, we describe comparison in the second (operator theoretic) framework.  Here one asks when does comparison \eqref{CP:intro} hold on $\overline{\Omega}$ for each pair $u \in \USC(\overline{\Omega}), w \in \LSC(\overline{\Omega})$ which are {\em $\cG$-admissible viscosity subsolutions, supersolutions} in the sense of Definition \ref{defn:AVS} of a {\em proper elliptic}  equation
\begin{equation}\label{FNE:intro}
	F(J^2u):= F(x,u(x),Du(x),D^2u(x)) = 0, \ \ \forall \, x \in \Omega,
\end{equation}
where $F \in C(\cG)$ with either $\cG = \cJ^2(X)$ (the {\em unconstrained case}) or $\cG \subsetneq \cJ^2(X)$ is a subequation (the {\em constrained case}). The proper ellipticity means the following monotonicity property: for each $x \in X$ and each $(r,p,A) \in \cG_x$ one has
\begin{equation}\label{PE:intro}
F(x,r,p,A) \leq F(x,r + s, p, A + P) \ \quad \forall \, s \leq 0 \ \text{in} \ \R \ \text{and} \  \forall \, P \geq 0 \ \text{in} \ \cS(n).
\end{equation}
Notice that one of the subequation axioms on $\cG$ is the monotonicity property that for each $x \in X$ one has
\begin{equation*}\label{mono:intro}
\cG_x + \cM_0 \subset \cG_x \quad \text{where} \quad \cM_0:= \{(r,p,A): \ s \leq 0 \ \text{in} \ \R \ \text{and} \  \forall \, P \geq 0 \ \text{in} \ \cS(n) \},
\end{equation*}
which is needed in \eqref{PE:intro}. We will refer to $(F, \cG)$ as a {\em proper elliptic pair} (see Definition \ref{defn:PEO}). Notice also that proper ellipticity is the familiar (opposing) monotonicity in solution variable $r$ and the Hessian variable $A$, which we do \underline{not} necessarily assume globally on all of $\cJ^2(X)$ (but do in the unconstrained case). In general, a given operator $F$ must be retricted to a subequation in order to have the minimal monotonicity needed \eqref{PE:intro}. For example, the Monge-Amp\`ere operator $F(x,r,p,A) := {\rm det}(A)$ must be restricted to the constraint $\cG := \{ (r,p,A) \in \cJ^2: \ A \in \cS(n): \ A \geq 0\}$  in order to be increasing in $A$. In this case, the $\cG$-admissible subsolutions are convex and one uses only convex  lower test functions in the definition of $\cG$-admissible subsolutions. 

We will deduce a general operator theoretic comparison Theorem \ref{thm:CP_MME} from the potential theoretic comparison Theorem \ref{thm:GCT} by way of the aforementioned Correspondence Principle of Theorem \ref{thm:corresp_gen}. This correspondence consists of the two equivalences: for every $u \in \USC(X)$
\begin{equation*}\label{Corr1:intro}
u \ \text{is $\cF$-subharmonic on $X$} \ \Leftrightarrow u \ \text{is a $\cG$-admissible subsolution of} \ F(J^2u) = 0 \ \text{on $X$} 
\end{equation*}
and
\begin{equation*}\label{Corr2:intro}
u \ \text{is $\cF$-superharmonic on $X$}  \ \Leftrightarrow  \ u \ \text{is a $\cG$-admissible supersolution of} \ F(J^2u) = 0,
\end{equation*}
where $(F, \cG)$ is a proper elliptic pair and $\cF$ is the constraint set defined by the {\em correspondence relation}
\begin{equation}\label{relation:intro}
\cF = \{ (x,J) \in \cG: \ F(x,J) \geq 0 \}.
\end{equation}
The correspondence principle holds provided that $\cF$ is itself a subequation and provided that one has {\em compatibility}
\begin{equation*}\label{compatibility1:intro}
\intr \cF = \{ (x,J) \in \cG: \ F(x,J) > 0\},
\end{equation*} 
which for subequations $\cF$ defined by \eqref{relation:intro} is equivalent to
\begin{equation*}\label{compatibility2:intro}
\partial  \cF = \{ (x,J) \in \cG: \ F(x,J) = 0\}.
\end{equation*}
Now, $\cF$ is indeed a subequation by Theorem \ref{thm:CMM} if the following three hypotheses are satisfied: (i) $\cG$ is fiberegular in the sense \eqref{Freg:intro} (with $\cF$ replaced by $\cG$), (ii) $(F, \cG)$ is $\cM$-monotone for some monotonicity cone subequation and (iii) $(F, \cG)$ satisfies the following regularity condition: for some fixed $J_0 \in \intr \call M$, given $\Omega \ssubset X$ and  $\eta > 0$, there exists $\delta= \delta(\eta, \Omega) > 0$ such that 
\begin{equation*}\label{UCF:intro}
F(y, J + \eta J_0) \geq F(x, J) \quad \forall J \in \cG_x,\ \forall x,y \in \Omega \ \text{with}\ |x - y| < \delta.
\end{equation*}
Not only is $\cF$ a subequation (and hence the correspondence principle holds), $\cF$ is also fiberegular and $\cM$-monotone. Hence one will have both operator theoretic and potential theoretic comparison in both versions \eqref{CP:intro} and \eqref{CPP:intro} on any domain $\overline{\Omega}$ which admits a $C^2$ strictly $\cM$-subharmonic function (which also satisfies \eqref{SA_parab} in the parabolic version).

Having described the main general comparison theorems in both frameworks, we now place them in context to indicate what is new in the paper. In the important special case of constant coefficient subequations $\cF \subset \cJ^2$ and constant coefficient operators $F \in C(\cG)$ with constant coefficient admissibility constraint $\cG \subseteq \cJ^2$, the entire program (and much more) has been developed in the research monograph \cite{CHLP22}. In particular, it is there in Chapter 11 that the important bridge of the Correspondence Principle was refined into its present form. In the current paper, variable coefficients require the fiberegularity conditions like \eqref{Freg:intro} in order to overcome the essential difficulty that the {\em sup-convolutions} used to approximate upper semicontinuous $\cF$-subharmonics with {\em quasi-convex}\footnote{We have adopted the term quasi-convex which is consistent with the use of {\em quasi-plurisubharmonic} function in several complex variables. Quasi-convex functions are often referred to as {\em semiconvex} functions, although this term is a bit misleading. They are functions whose Hessian (in the viscosity sense) is locally bounded from below.} functions do \underline{not} preserve the property of being $\cF$-subharmonic in the variable coefficient case. Here fiberegularity ensures what we call the {\em uniform translation property} for subharmonics which roughly states that (see Theorem \ref{thm:UTP}): {\em if $u \in \cF(\Omega)$, then there are small $C^2$ strictly $\cF$-subharmonic perturbations of \underline{all small translates} of  $u$ which belong to  $\cF(\Omega_{\delta})$}, where  $\Omega_{\delta}:= \{ x \in \Omega: d(x, \partial \Omega) > \delta \}$.

The formulation of the fiberegularity condition on $\cF$ and the proof that it implies the uniform translation property was done first in the variable coefficient pure second order case where
$$
	\cF \subset X \times \cS(n) \quad \text{and} \quad F = F(x,A) \in C(\cG) \ \text{with} \ \cG \subset X \times \cS(n)
$$
in \cite{CP17} and then extended to the variable coefficient gradient-free second order case where
$$
\cF \subset X \times \R \times \cS(n) \quad \text{and} \quad F = F(x,r,A) \in C(\cG) \ \text{with} \ \cG \subset X \times \R \times \cS(n)
$$
in \cite{CP21}. However, in these papers, the term fiberegularity was not used. The term fiberegularity was coined in the production of the survey paper \cite{HP22}. The terminology of \cite{CP17} and \cite{CP21} borrowed much from the fundamental paper of Krylov \cite{Kv95} on the general notion of ellipticity. More importantly, the form of the correspondence principle in \cite{CP17} and \cite{CP21} was more rudimentary than the form described above.
The present paper adds additional results and refinements to the variable coefficient pure second order and gradient-free situations of \cite{CP17} and \cite{CP21}, which heavily benefit from the investigation of the constant coefficient case in \cite{CHLP22}.

This brings us to the main issue of this paper, which is establishing comparison in both elliptic and parabolic versions \eqref{CP:intro} and \eqref{CPP:intro} for variable coefficient potential theories and PDEs with {\em directionality} in the gradient variables.  
\begin{defn}\label{defn:property_D} A closed convex cone $\cD \subseteq \R^n$ (possibly all of $\R^n$) with vertex at the origin and $\intr \cD \neq \emptyset$ will be called a {\em directional cone}. We say that a subequation $\cF \subset \cJ^2(X)$ on an open subset $X$ satisfies the {\em directionality condition (with respect to $\cD$)} if 
	\begin{equation*}\label{DC1}
	{\rm (D)} \ \ \ \ (r,p,A) \in \cF_x \ \ \text{and} \ \  q \in \cD \ \ \Rightarrow\ \ (r,p + q,A)\in \cF_x, \ \ \forall \, x \in X.
	\end{equation*}
\end{defn} 
Notice that when $\cD = \R^n$, the $\cD$-directionality just means that $\cF$ is gradient-free. Hence we will be particularly interested in  situations where there is $\cD$-directionality of the subequation $\cF$ with a directional cone $\cD \subsetneq \R^n$,  in order to extend what is known from the gradient-free case in \cite{CP21}. Similarly, for a given proper elliptic pair $(F, \cG)$ we will be particularly interested in situations in which $\cG$ satisfies (D) with $\cD \subsetneq \R^n$ and the natural property of directionality in the gradient variables
	\begin{equation*}\label{DC2}
 F(x,r,p + q,A) \geq F(x,r,p,A) \quad \text{for each } \ (r,p,A)\in \cG_x, q \in \cD, x \in X.
\end{equation*}

Some interesting directional cones $\cD \subsetneq \R^n$ are given in Example 12.33 of \cite{CHLP22} and we recall two of them here:
\begin{equation}\label{parab_cone}
\cD  = \{p = (p',p_n): \ p_n \geq 0\} \quad \text{(a half-space)};
\end{equation}
\begin{equation}\label{pos_cone}
\cD  = \{p = (p_1, \ldots, p_n): \ p_j \geq 0, \ \ j = 1, \ldots, n\} \quad \text{(the positive cone)}.
\end{equation}

We now describe two example applications of the general comparison theorems to interesting fully nonlinear PDEs with directionality in the gradient from Section \ref{sec:examples}; an elliptic example and a parabolic example. The elliptic example concerns equations that arise the theory of {\em optimal transport} and is the following example.
\begin{example}[Example \ref{exe:OTE} (Optimal transport)]\label{exe:OTE:intro}
	The equation 
	\begin{equation}\label{e:ot:intro}
g(Du) \det(D^2 u) = f(x), \quad x \in \Omega \ssubset \R^n
\end{equation}
describes the optimal transport plan from a source density $f$ to a target density $g$. In Proposition \ref{prop:OTE} we will prove the elliptic version of comparison under the hypotheses that
\begin{equation}\label{fg_ot:intro}
f, g \in C(\overline{\Omega}) \ \ \text{and are nonnegative}
\end{equation}
and that $g$ is $\cD$-directional with respect to some directional cone $\cD \subsetneq \R^n$; that is,
	\begin{equation}\label{g_ot1:intro}
g(p + q) \ge g(p), \quad \text{for each} \ p,q \in \cD.
\end{equation}
We also require some measure of strict directionality in the sense that there exists $\bar q \in \intr \cD$ and a modulus of continuity $\omega : (0,\infty) \to (0,\infty)$ (satisfying $\omega(0^+) = 0$) such that   
	\begin{equation}\label{g_ot2:intro}
	g(p + \eta \bar q) \ge g(p) + \omega(\eta), \quad \text{for each} \ p,q \in \cD \ \text{and each} \ \eta > 0.
	\end{equation}
The natural operator $F$ associated to \eqref{e:ot:intro} is defined $F(x,r,p,A):= g(p) {\rm det}(A) - f(x)$ and is proper elliptic when restricted to $A \geq 0$ in $\cS(n)$. The compatible subequation $\cF$ with fibers
\begin{equation*}\label{SE:OTE}
		\cF_x := \{(r,p,A) \in \cJ^2: \ p \in \cD,  A \geq 0 \ \text{in} \ \cS(n) \geq 0 \ \text{and} \ F(x,r,p,A) \geq 0\}
\end{equation*}
is fiberegular and $\cM$-monotone for
\begin{equation*}\label{MC:OTE}
\cM = \cM(\cD, \cP) := \{(r,p,A) \in \cJ^2: \ p \in \cD \ \text{and} \ A \geq 0 \ \text{in} \ \cS(n)\}.
\end{equation*}
As shown in \cite{CHLP22}, these cones admit $C^2$ strictly $\cM$ subharmonics on all bounded domains so one has potential theoretic comparison for $\cF$ as well as operator theoretic comparison for $\cG$-admissible subsolutions, supersolutions of \eqref{e:ot:intro} with $\cG = \cM(\cD, \cP)$.
	\end{example}

The parabolic example that we describe is a prototype of a fully nonlinear PDE which is weakly parabolic in the sense of Krylov and also indicates the utility of half-space cones (in the gradient variable) \eqref{parab_cone} in this parabolic context. 

\begin{example}[Example \ref{exe:KPO} (Krylov's parabolic Monge-Amp\`ere operator)]\label{exe:KPO:intro} In \cite{Krylov76}, the following nonlinear parabolic equation is considered
	\begin{equation}\label{e:kry:intro}
	-\partial_t u \det (D_x^2 u) = f(x,t), \quad   (x,t) \in X :=\Omega \times (0,T) \subset \R^{n+1},
	\end{equation}
	where $\Omega \ssubset \R^n$ is open and $T > 0$. The reduced boundary of the parabolic cylinder $X$ is
	\begin{equation*}\label{par_bdy:intro}
		\partial^-X := \left(\Omega \times \{0\}\right) \cup \left(\partial \Omega \times (0, T)\right),
	\end{equation*}
	which is the usual parabolic boundary of $X$. In Proposition \ref{prop:KPO}, for arbitrary bounded parabolic cylinders $X$ and $f \in C(\overline{X})$ nonnegative, we prove the parabolic version of comparison
	\begin{equation}\label{par_CP:intro}
		u \leq v \ \text{on} \ \partial^- X \ \ \Rightarrow \ \ 	u \leq v \ \text{in} \  X,
\end{equation}
	for $\cG$-admissible subsolutions, supersolutions $u,v$ of the equation \eqref{e:kry:intro}, where the admissibility constraint is the natural constant coefficient subequation with constant fibers
	\begin{equation*}\label{KPO_constraint:intro}
	\cG := \cM(\cD_n,\cP_n) :=  \{(r, p, A) \in \R \times \R^{n+1} \times \cS(n+1): p_{n +1} \leq 0 \ \text{and} \ A_n \geq 0 \},
	\end{equation*}
where $A_n \in \cS(n)$ is the upper-left $n \times n$  submatrix of $A$. This is because the compatible subequation $\cF$ with fibers
$$
\cF_{(x,t)} := \{ (r,p,A) \in \cM(\cD_n, \cP_n): \ F((x,t),r,p,A) := -p_{n+1} {\rm det}(A) - f(x,t) \geq 0 \}
$$
is fiberegular and $\cM(\cD_n, \cP_n)$-monotone, where every parabolic cylinder $X$ admits a strictly $C^2$ and $\cM(\cD_n, \cP_n)$-subharmonic function $\psi$ which satisfies
$$
\psi \equiv - \infty \quad \text{on}\  \partial X \setminus \partial^-X,
$$
and hence the  parabolic version of Theorem \ref{thm:CP_MME} applies.  
\end{example}

Many additional examples of fully nonlinear operators with directionality in the gradient variables can be constructed from {\em Dirichlet-G\aa rding polynomials} on the vector space $V = \R^n$, which are homogeneous polynomials $\gf$ of degree $m$ are {\em hyperbolic} with respect to the direction $q \in \R^n$ in the sense that the one-variable polynomial 
\begin{equation}\label{hyp_poly:intro}
t \mapsto \gf(tq + p) \ \text{has exaclty $m$ real roots for each $p \in \R^n$.}
\end{equation}
See Definition \ref{defn:hyp_poly} and the brief discussion which follows on concerning elements of G\aa rding's theory of hyperbolic polynomials. The key point is that one represent the first order operator determined by $\gf$ as a {\em generalized Monge-Amp\`ere operator} in the sense that
\begin{equation*}\label{GMAO:intro}
	\gf(p) = \lambda_1^{\gf}(p) \cdots \lambda_m^{\gf}(p).
\end{equation*}
For $k = 1, \ldots , m$, the factor $\lambda_k^{\gf}(p)$ is the {\em $k$-th G\aa rding $q$-eigenvalue of $\gf$}, which is just the negative of the $k$-th root in \eqref{hyp_poly:intro} (reordered so that $\lambda_1^{\gf}(p) \leq \lambda_2^{\gf}(p) \cdots \leq \lambda_m^{\gf}(p)$).
There is always a natural monotonicty cone $\overline{\Gamma}$ (the {\em (closed) G\aa rding cone}) for a hyperbolic polynomial $\gf$, which in the case $V = \R^n$ is a directional cone $\cD$. 

In order to illustrate the construction above, in Example \ref{exe:hyp_poly} we discuss the polynomial defined for $p = (p_1, p_2) \in \R^2$ by
\begin{equation*}\label{exe:hyp_poly:intro}
\gf(p_1,p_2) := p_1^2 - p_2^2 = \lambda_1^{\gf}(p) \lambda_2^{\gf}(p) = (p_1 - |p_2|)(p_1 + |p_2|),
\end{equation*}
which determines the pure first order fully nonlinear equation 
\begin{equation}\label{PFO:intro}
	u_x^2 - u_y^2 = 0, \quad (x,y) \in \R^2.
\end{equation}
The associated directional cone is
\begin{equation}\label{DC:intro}
	\cD_{\gf} = \overline{\Gamma} = \big\{ p \in \R^2: \ \lambda_1^{\gf}(p) := p_1 - |p_2| \geq 0\}.
\end{equation}
In Proposition \ref{prop:DGO} we prove a parabolic version of comparison on rectangular domains $\Omega \subset \R^2$ for $\cG$-admissible subsolutions, supersolutions of \eqref{PFO:intro} with respect to the natural admissibility constraint
\begin{equation*}\label{DGO_constraint:intro}
\cG= \cM_{\gf} := \R \times \cD_{\gf} \times \cS(2),
\end{equation*}
which is also the monotonicity cone subequation for comparison.

As a final example of a fully nonlinear operator with directionality in the gradient, we will discuss the following interesting operator.

\begin{example}[Example \ref{exe:PMAD} Perturbed Monge-Amp\`ere operators with directionality]\label{exe:PMAD:intro}
	On a bounded domain $\Omega \subset \R^n$, consider the operator defined by
\begin{equation}\label{PMO:intro}
	F(x,r,p,A) = F(x,p,A) \defeq \det\!\big(A + M(x,p)\big) - f(x), \ \ (x,r,p,A) \in \Omega \times \cJ^2
\end{equation}
with $f \in UC(\Omega; [0,+\infty))$ and with $M \in UC(\Omega \times \R^n; \call S(n))$ of the form
\begin{equation*} \label{Mex1:intro}
M(x,p) \defeq \pair{b(x)}{p} P(x)
\end{equation*}
with $P \in UC(\Omega; \cP)$ and $b \in UC(\Omega; \R^n)$ such that
\begin{gather*}
\label{condexcone:intro}
\text{there exists a unit vector $\nu \in \R^n$ such that $\pair{b(x)}{\nu} \geq 0$ for each $x \in \Omega$.}
\end{gather*}
Such operators with $M = M(x)$ have been proposed by Krylov as interesting test cases for probabilistic and analytic methods. Our interest in this example is two fold. On the one hand, we can prove the elliptic version of comparison using our methods with a natural directional cone 
\[
\call D \defeq \bigcap_{x \in \Omega} H^+_{b(x)} \quad \text{where} \quad H_{b(x)}^+ \defeq \big\{ q \in \R^n :\ \pair{b(x)}{q} \geq 0 \big\}.
\]
See the discussion in Example \ref{exe:PMAD}. On the other hand, we show in Proposition \ref{prop:fail} standard viscosity structural conditions fail for $F$, and hence this example shows that our methods can provide comparison in non standard cases with directionality in the gradient (as was already known in the pure second order and gradient-free settings from \cite{CP17} and \cite{CP21}).

\end{example}

\noticina{slight lingustic change in text} While our main focus here is on the (strict) directionality with respect to first order terms, we stress that our theory allows us to treat operators that are parabolic in a broad sense. For instance, pure linear second-order operators ${\rm tr}(B D^2 u)$ are classically considered to be parabolic provided that $B \ge 0$ and $\det B = 0$. Hence, there is a natural nontrivial monotonicity cone associated to the (restricted) subequation $\cF = \{A : {\rm tr}(B A) \ge 0\}$ which is $\cM = \cF$. Therefore, any operator $F$ that can be paired with a subequation with such monotonicity cone $\cM$ can be regarded as parabolic, and specific comparison principles on suitable restricted boundaries can be easily deduced.

\smallskip

Fundamental to the entire project we discuss here is the groundbreaking paper of Harvey and Lawson \cite{HL09} which examined the potential theoretic comparison as well as existence via Perron's method in the constant coefficient case pure second order case for potential theories
$$
	\cF \subset \cS(n).
	$$
No correspondence principle is found there, as the focus was on the potential theory side. This is because in the geometric situations of interest to them, often there are no natural operators associated to geometric potential theories of interest. It was in \cite{HL09} that Krylov's fundamental insight to associate constraint sets which encode the natural notion of ellipticity for differential operators takes shape and is encoded by them in the language of general (nonlinear) potential theories. Moreover, in \cite{HL09} the natural notion of duality (the {\em Dirichlet dual}) is formalized. It is implicit in \cite{Kv95}, but made explicit in \cite{HL09} and used elegantly to clarify the notion of supersolutions in constrained cases and as a crucial ingredient of their {\em monotonicity-duality method} for proving comparison. This method is presented in Section \ref{sec:SAaC}, which for the first time incorporates the parabolic version into the crucial {\em Zero Maximum Principle} (ZMP) in Theorem \ref{thm:zmp} (note that this is new also for pure second order and gradient-free settings).

\noticina{Good idea, Marco! Amplified discussion a bit with reference also to HL - InHom paper} 
Two remarks on the crucial fiberegularity condition (the continuity of the fiber map $\Theta$ of \eqref{Freg:intro}) are in order. First, the recent interesting work of Brustad \cite{B23} in the pure second order setting (operators without first and zero-th order terms), introduces a regularity property for the fiber map $\Theta$ which is weaker than the fiberegularity used here. A concise discussion this weaker condition is given in the introduction of \cite{B23} which aims to incorporate the best features of fiberegulairty and standard viscosity structual conditions in this case. Second, the recent important paper of Harvey and Lawson \cite{HLInHom}, which studies the Dirichlet problem for {\em inhomogeneous equations} on manifolds $X$
	$$
		F(J^2u) = \psi(x), \ \ x \in X,
	$$
under the assumptions that $(F, \cG)$ is an $\cM$-monotone compatible operator-subequation pair for which the operator is {\em tame}. In the constant coefficient case on $\R^n$ this condition requires that for every $s, \lambda > 0$ there exists $c(s,\lambda) > 0$ such that
\begin{equation}\label{tame}
	F(J + (r,0,P)) - F(J) \geq c(s,\lambda), \quad \forall \, J \in \cG \ \text{and} \ P \geq \lambda I \ \text{in} \  \cS(n).
\end{equation}
This property, which not comparable to the fiberegularity of $\cF:= \{ J \in \cG: \ F(J) \geq 0\}$, plays the same role as fiberegularity in this inhomogeneous setting.

We conclude this introduction with a brief description of the contents. Part 1 of the paper (Sections \ref{sec:NPT} - \ref{sec:Char}) concerns the potential theoretic setting, including the elliptic and parabolic versions of comparison by the monotonicity-duality method in the presence of fiberegularity. Section \ref{sec:Char} also gives some new characterizations of {\em dual cone subharmonics} that play a crucial role in comparison by way of the (ZMP). Part 2 of the paper is Section \ref{sec:correspondence} which builds the bridge between the potential theoretic framework and the operator theoretic framework by way of the correspondence principle. Part 3 of the paper treats comparison in the operator theoretic framework and is highlighted by the examples mentioned above. 

In addition there are three appendices. Appendix \ref{AppA} contains many new auxilliary technical results needed to complete the proof of of Theorem \ref{thm:MMS} which proves that given an $\cM$-monotone pair $(F, \cG)$ the natural constraint set $\cF$ defined by the correspondence relation
\begin{equation*}\label{relation1:intro}
\cF:= \{ (x,J) \in \cG: \ F(x,J) \geq 0 \},
\end{equation*} 
is a fiberegular $\cM$-monotone subequation if $\cG$ is fiberegular. This theorem plays an important role in the general PDE comparison principle Theorem \ref{thm:CP_MME}. Appendix \ref{AppB} collects some known results which are fundamental for the potential theoretic methods and is included for the convenience of the reader. Appendix \ref{AppC} recalls some elementary facts about the Hausdorff distance which are used in the discussion of fiberegularity in Section \ref{sec:FR}.

\section{Background notions from nonlinear potential theory}\label{sec:NPT}

In this section, we give a brief review of some key notions and fundamental results in the theory of $\cF$-subharmonic functions defined by a subequation constraint set $\cF$.

\subsection{Subequation constraint sets and their subharmonics}\label{subsec:subeqs}

Suppose that $X$ is an open subset of $\R^n$ with $2$-jet space denoted by  $\cJ^2(X) = X \times \call J^2 = X \times (\R \times \R^n \times \cS(n))$. A good definition of a constraint set with a robust potential theory was given in \cite{HL11} (also for manifolds).
\begin{defn}[Subequations]\label{defn:subeq} A set $\cF \subset \cJ^2(X)$ is called a {\em subequation (constraint set)} if
		\begin{itemize}
		\item[(P)] $\cF$ satisfies the {\em positivity condition}
		(fiberwise); that is, for each $x \in X$
		$$
		(r,p,A) \in \cF_x \ \ \Rightarrow \ \ (r,p,A + P) \in \cF_x, \ \ \forall \, P \geq 0 \ \text{in} \ \cS(n).
		$$
		\item[(T)]  $\cF$ satisfies three conditions of {\em topological stability}\footnote{Here and below $\intr \cF$ denotes the interior of a subset $\cF$ of a topological space.} : 
		\begin{gather}\tag{T1} 
		\cF = \overline{\intr \cF}; \\
		\tag{T2} 
		\cF_x = \overline{\intr \left( \cF_x \right)}, \ \ \forall \, x \in X; \\
		\tag{T3} \left( \intr \cF \right)_x = \intr \left( \cF_x \right), \ \ \forall \, x \in X.
		\end{gather}
		\item[(N)]  $\cF$ satisfies the {\em negativity condition}
		(fiberwise); that is, for each $x \in X$
		$$
		(r,p,A) \in \cF_x \ \ \Rightarrow \ \ (r + s,p,A) \in \cF_x, \ \ \forall \, s \leq 0 \ \text{in} \ \R.
		$$
	\end{itemize}
\end{defn}
Notice that by property (T3) we can write without ambiguity $\intr \cF_x$ for the subset of $\cJ^2$, which can be calculated in two ways. The conditions (P), (T) and (N) have various (important) implications for the potential theory determined by $\cF$. Some of these will be mentioned below (see the brief discussion following Definition \ref{defn:Fsub}). In addition, the conditions (P) and (N) are {\bf {\em monotonicity}} properties; monotonicity plays a central and unifying role as will be discussed in Subsection \ref{sec:monotonicity}. The role of property (T) is clarified by the notion of {\bf {\em duality}}; another fundamental concept that will be discussed in Subsection \ref{sec:duality}.  For now, notice that  by property (T1), $\cF$ is closed in $\cJ^2(X)$ and each fiber $\cF_x$ is closed in $\cJ^2$ by (T2). In addition, the interesting case is when each fiber $\cF_x$ is not all of $\cJ^2$, which we almost always assume. Also notice that in the constant coefficient pure second order case where the (reduced) subequation \footnote{See Subsection \ref{sec:monotonicity} for a discussion on reduced subequations.} can be identified with a subset $\cF  \subset \cS(n)$, property (N) is automatic and property (T) reduces to (T1)  $\cF = \overline{\intr \cF}$, which is implied by (P) for $\cF$ closed. Hence in this case, subequations $\cF \subset \cS(n)$ are closed sets simply satisfying (P). Additional considerations on property (T) will be discussed in Appendix \ref{AppA}.

Next we recall the notion of {\bf {\em $\cF$-subharmonicity}} for a given subequation $\cF \subset \cJ^2(X)$. There are two different natural formulations for differing degrees of regularity.
The first is the classical formulation.

\begin{defn}[Classical or $C^2$ subharmonics]\label{defn:CSH} A function $u \in C^2(X)$ is said to be {\em $\cF$-subharmonic on $X$} if
	\begin{equation}\label{VsubClass}
	J^2_x u := (u(x), Du(x), D^2u(x)) \in \cF_x, \ \ \forall \, x \in X
	\end{equation}
	with the accompanying notion of being {\em strictly $\cF$-subharmonic} if 
	\begin{equation}\label{VsubCS}
	J^2_x u \in \intr (\cF_x) = (\intr \cF)_x, \forall \, x \in X.
	\end{equation}
\end{defn}

For merely upper semicontinuous functions $u \in \USC(X)$ with values in $[-\infty, + \infty)$, one replaces the $2$-jet $J^2_x u$ with the set of {\em $C^2$ upper test jets}
\begin{equation}\label{UCJ}
J^{2,+}_{x} u := \{ J^2_{x} \varphi:  \varphi \ \text{is} \ C^2 \ \text{near} \ x, \  u \leq \varphi \ \text{near} \  x \ \text{with equality at} \ x \},
\end{equation}
thus arriving at the following viscosity formulation.

\begin{defn}[Semicontinuous subharmonics]\label{defn:Fsub}  A function $u \in \USC(X)$ is said to be {\em $\cF$-subharmonic on $X$} if
	\begin{equation}\label{Vsub}
	J^{2,+}_x u \subset \cF_x, \ \ \forall \, x \in X.
	\end{equation}
	We denote by $\cF(X)$ the set of all $\cF$-subharmonics on $X$.
\end{defn} 

We now recall some of the implications that properties (P), (T) and (N) have on an $\cF$-potential theory. Property (P) ensures that Definition \ref{defn:Fsub} is meaningful since for each $u \in \USC(X)$ and for each $x_0 \in X$ one has property (P) for the upper test jets
\begin{equation}\label{JIP}
	(r,p,A) \in J^{2,+}_{x_0} u\ \ \Rightarrow \ \ (r,p,A + P) \in J^{2,+}_{x_0} u, \ \ \text{for each $P \geq 0$ in $\cS(n)$}. 
\end{equation}
Indeed, given an upper test jet $J^2_{x_0} \varphi =(r,p,A)$ with $\varphi$ a $C^2$ function near $x_0$ and satisfying $ u \leq \varphi$ near   $x_0$ with equality at $x_0$ then, for each $P \geq 0$, the quadratic perturbation $\wt{\varphi}(\cdot):= \varphi(\cdot) + \frac{1}{2} \langle P(\cdot - x_0), (\cdot - x_0) \rangle$ determines an upper test jet $J^2_{x_0} \wt{\varphi} =(r,p,A + P)$. Property (P) is also crucial for {\bf {\em $C^2$-coherence}}, meaning classical $\cF$-subharmonics are $\cF$-subharmonics in the sense \eqref{Vsub}, since for $u$ which is $C^2$ near $x$, one has
$$
J^{2,+}_x u = J_x^2u + (\{0\} \times \{0\} \times \cP )\ \ \text{where} \ \ \cP = \{ P \in \cS(n): \ P \geq 0 \}.
$$
Next note that property (T) insures the local existence of strict classical $\cF$-subharmonics at points $x \in X$ for which $\cF_x$ is non-empty. One simply takes the quadratic polynomial whose $2$-jet at $x$ is $J \in \intr (\cF_x)$. Finally, property
(N) eliminates obvious counterexamples to comparison. The simplest counterexample is provided by the constraint set $\cF \subset \cJ^2(\R)$ in dimension one associated to the equation $u^{\prime \prime} - u = 0$, which is defined by $\cF := \{(r,p,A) \in \R^3: A - r \geq 0 \}$.

\subsection{Duality and superharmonics}\label{sec:duality}  The next fundamental concept  is {\bf {\em duality}}, a notion first introduced in the pure second order coefficient case in \cite{HL09}.

\begin{defn}[Duality for constraint sets]\label{defn:dual} For a given subequation $\cF \subset \cJ^2(X)$ the 
	{\em Dirichlet dual} of $\cF$ is the set $\wt{\cF} \subset \cJ^2(X)$ given by \footnote{Here and below, $c$ denotes the set theoretic complement of a subset.}
	\begin{equation}\label{dual}
	\wt{\cF} :=  (- \intr \cF)^c = - (\intr \cF)^c \ \ \text{(relative to 
		$\cJ^2(X)$)}.
	\end{equation}
\end{defn}
With the help of property (T), the dual can be calculated fiberwise
\begin{equation}\label{dual_fiber}
\wt{\cF}_x :=  (- \intr \cF_x)^c = - (\intr \cF_x)^c \ \ \text{(relative to \  $\cJ^2$)}, \ \ \forall \, x \in X.
\end{equation}
This is a true duality in the sense that one can show the following two facts: 
\begin{equation}\label{true_dual}
\wt{\wt{\cF}} = \cF \quad \quad \quad \text{and} \quad \quad \quad \text{$\cF$ is a subequation \ \ $\Rightarrow$ \ \  $\wt{\cF}$ is a subequation}. 
\end{equation}
Additional (and useful) properties of the dual can be found in Propositions 3.2 and 3.4 of \cite{CHLP22}. These properties include the behavior of the dual with respect to inclusions, intersections and  fiberwise sums:
\begin{equation}\label{dual_inclusion}
\cF \subset \cG \ \ \Rightarrow \ \ \wt{\cG} \subset \wt{\cF};
\end{equation}
\begin{equation}\label{dual_intersections}
\wt{\cF \cap \cG} = \wt{\cF} \cup \wt{\cG};
\end{equation}
\begin{equation}\label{dual_sums}
\cF_x + J \subset \cF_x \ \ \Rightarrow \ \ \wt{\cF}_x + J \subset \wt{\cF}_x \quad \text{for each} \ x \in X \ \text{and} \  J = (r,p,A)  \in \cJ^2.
\end{equation}
This last formula, when combined with the monotonicity discussed below, will lead to the fundamental formula \eqref{mono_dual} for the monotonicity-duality method.

Another importance of duality is that it can be used to reformulate the notion of {\em $\cF$-superharmonic functions} in terms of {\em dual subharmonic functions}. This will have important implications for the correct definition of {\em supersolutions} to a degenerate elliptic PDE $F(J^2u) = 0$ in the presence of {\bf {\em admissibility constraints}}. See Subsection \ref{sec:AVS} for this discussion.

 The natural notion of $w \in \LSC(X)$ being {\em $\cF$-superharmonic} using {\em lower test jets} is  
\begin{equation}\label{Vsuper1}
J^{2,-}_x w \subset \left(\intr (\cF_x)\right)^c, \ \ \forall \, x \in X,
\end{equation}
which by duality and property (T) is equivalent to $-w \in \USC(X)$ satisfying
\begin{equation}\label{Vsuper2}
J^{2,+}_x (-w) \subset \wt{\cF}_x, \ \ \forall \, x \in X.
\end{equation}
That is,  
\begin{equation}\label{Vsuper3}
\text{$w$ is $\cF$-superharmonic \ \ $\Leftrightarrow$ \ \ 	$-w$ is $\wt{\cF}$-subharmonic}.
\end{equation}

\subsection{Monotonicity}\label{sec:monotonicity}
This fundamental notion appears in various guises. It is a useful and unifying concept. One says that a subequation $\cF$ is {\em $\cM$-monotone} for some subset $\cM \subset \cJ^2(X)$ if 
\begin{equation}\label{monotone}
\cF_x + \cM_x \subset \cF_x \ \ \text{for each} \  x \in X.
\end{equation}
For simplicity, we will restrict attention to (constant coefficient) {\em monotonicity cones}; that is, monotonicity sets $\cM$ for $\cF$ which have constant fibers which are closed cones with vertex at the origin. 

First and foremost, the properties (P) and (N) are monotonicity properties. Property (P) for subequations $\cF$ corresponds to {\em degenerate elliptic} operators $F$ and properties (P) and (N) together correspond to {\em proper elliptic} operators. Note that (P) plus (N) can be expressed as the single monotonicity property 
\begin{equation}\label{MM}
\cF_x + \cM_0 \subset \cF_x   \ \ \text{for each} \ x \in X
\end{equation}
where
\begin{equation}\label{MMC}
\cM_0 := \cN \times \{0\} \times \cP \subset \cJ^2 = \R \times \R^n \times \cS(n)
\end{equation}
with
\begin{equation}\label{NP}
\cN := \{ r \in \R : \ r \leq 0 \} \quad \text{and} \quad  \cP := \{ P \in \cS(n) : \ P \geq 0 \} .
\end{equation}
Hence $\cM_0$ will be referred to as the {\em minimal monotonicity cone} in $\cJ^2$. However,  it is important to remember that $\cM_0 \subset \cJ^2$ is {\bf not} a subequation since it has empty interior so that property (T) fails. A monotonicity cone which is also a subequation will be called a {\em monotonicity cone subequation}.

Combined with duality and {\bf {\em fiberegularity}} (defined in Section \ref{sec:FR}), one has a very general, flexible and elegant geometrical approach to comparison when a subequation $\cF$ admits a constant coefficient monotonicity cone subequation $\cM$. We call this approach the {\bf {\em monotonicity-duality method}} and it will be discussed in Section \ref{sec:SAaC}.
One key point in the method is the following {\em monotonicity-duality formula} that combines monotonicity \eqref{monotone} and the duality formula on fiberwise sums \eqref{dual_sums}: 
\begin{equation}\label{mono_dual}
\cF_x + \cM \subset \cF_x \ \ \Rightarrow \ \ \cF_x + \wt{\cF}_x \subset \wt{\cM} \quad \text{for eaxh} \ x \in X.
\end{equation}
It is interesting to note that if a subequation $\cF$ has a constant coefficient monotonicity cone subequation $\cM$ then the fiberwise sum of $\cF$ and its dual $\wt{\cF}$ yields a constant coefficient subequation $\wt{\cM}$ which is also a cone (dual to the monotonicity cone for $\cF$ and $\wt{\cF}$). A detailed study of monotonicity cone subequations can be found in Chapters 4 and 5 of \cite{CHLP22}, including the construction of a {\em fundamental family} of monotonicity cones that is recalled below in \eqref{fundamental_family1}-\eqref{fundamental_family2}.

Monotonicity is also used to formulate {\bf {\em reductions}} when certain jet variables are ``silent'' in the subequation constraint $ \cF$. For example, one has
$$
\text{(pure second order)} \ \ \ \cF_x + \cM(\cP) \subset \cF_x: \ \ \cM(\cP):= \R \times \R^n \times \cP
$$
$$
\text{(gradient free)} \ \ \ \cF_x + \cM(\cN,\cP) \subset \cF_x: \ \ \cM(\cN,\cP):= \cN \times \R^n \times \cP
$$
$\cM(\cP)$ and $\cM(\cN, \cP)$ are fundamental {\em constant coefficient (cone) subequations} which can be identified
with $\cP \subset \cS(n)$ and $\cQ := \cN \times \cP \subset \R \times \cS(n)$. One can identify $\cF$ with subsets of the {\em reduced jet bundles} $X \times \cS(n)$ and $X \times (\R \times \cS(n))$, respectively, ``forgetting about'' the silent jet variables (see Chapter 10 of \cite{CHLP22}). For a more extensive review of the monotonicity, see subsection 2.2 of \cite{HP22}.  

Three important ``reduced'' examples are worth drawing special attention to. They are all monotonicity cone subequations and play a fundamental role in our method. We focus on characterizing their subharmonics and their dual subharmonics. 

\begin{example}[The convexity subequation]\label{exe:CSE} The {\em convexity subequation} is $\cF = X \times \cM(\cP)$ and reduces to $X \times \cP$ which has constant coefficients (each fiber is $\cP$) and for $u \in \USC(X)$ 
	$$
	u \ \text{is $\cP$-subharmonic} \ \ \Leftrightarrow \ \ u \ \text{is locally convex} 
	$$
	(away from any connected components where $u \equiv - \infty$).
	
	The convexity subequation has a so-called {\em canonical operator} 
	 $F  \in C(\cS(n), \R)$ defined by the minimal eigenvalue $F(A) := \lambda_{\rm min}(A)$, for which 
	\begin{equation}\label{PDual}
	\cP = \{ A \in \cS(n): \ \lambda_{\rm min}(A) \geq 0 \}.
	\end{equation} 
	The dual subequation $\wt{\cF}$ has constant fibers given by 
	\begin{equation}\label{Pdual}
	\wt{\cP} = \{ A \in \cS(n) : \lambda_{\rm max}(A) \geq 0 \}
	\end{equation}
	which is the {\em subaffine subequation}. The set $\wt{\cP}(X)$ of dual subharmonics agrees with $\SAff(X)$ the set of {\em subaffine functions} defined as those functions $u \in \USC(X)$ which satisfy the {\em subaffine property} (comparison with affine functions): for every $ \Omega  \ssubset X$ one has
	\begin{equation}\label{SAP}
	u \leq a \ \ \text{on} \ \partial \Omega \ \ \Rightarrow \ \ u \leq a \ \ \text{on} \ \Omega, \ \ \text{for every $a$ affine}. 
	\end{equation}
	The fact that $\wt{\cP}(X) = \SAff(X)$ is shown in \cite{HL09}. The subaffine property for $u$ is stronger than the maximum principle for $u$ since constants are affine functions. It has the advantage over the maximum principle of being a local condition on $u$. This leads to the comparison principle for all pure second order constant coefficient subequations \cite{HL09} and extends to variable coefficient subequations \cite{CP17} using the notion of fiberegularity noted above.
	
\end{example}

\begin{example}[The convexity-negativity subequation]\label{exe:CNSE} The constant coefficient gradient-free subequation $\cF= X \times \cM(\cN, \cP)$ reduces to $X \times \cQ \subset X \times (\R \times \cS(n))$ 
	whose (constant) fibers are
	\begin{equation}\label{Q}
	\cQ = \cN \times \cP = \{ (r,A) \in \R \times \cS(n): r \leq 0 \ \ \text{and} \ \  A \geq 0 \}.
	\end{equation}
	The (reduced) dual subequation has (constant) fibers
	\begin{equation}\label{Qdual}
	\wt{\cQ} = \{ (r,A) \in \R \times \cS(n): r \leq 0 \ \ \text{or} \ \  A \in \wt{\cP} \}.
	\end{equation}
	The set $\wt{\cQ}(X)$ of dual subharmonics agrees with  $\SAff^+(X)$, the set of {\em subaffine plus functions} defined as those functions $u \in \USC(X)$  which satisfy the {\em subaffine plus property}: for every $\Omega \ssubset X$ one has
	\begin{equation}\label{SAPP}
	u \leq a \ \ \text{on} \ \partial \Omega \ \ \Rightarrow \ \ u \leq a \ \ \text{on} \ \Omega, \ \ \text{for every $a$ affine with} \ a_{|\overline{\Omega}} \geq 0. 
	\end{equation}
	from which 
	the Zero Maximum Principle (ZMP) of Theorem \ref{thm:zmp} for $\wt{\cQ}$ subharmonics follows immediately. 
	The fact that $\wt{\cQ}(X) = \SAff^+(X)$ is shown in \cite{CHLP22} together with the additional equivalence 
	\begin{equation}\label{SAPP2}
	\SAff^+(X):= \{ u \in \USC(X): \ u^+ := \max \{u, 0 \} \in \SAff(X) = \wt{\cP}(X) \},
	\end{equation}
	This leads to the comparison principle by the monotonicity-duality method for all gradient free subequations with constant coefficients in \cite{CHLP22} and extends to variable coefficient gradient-free subequations in \cite{CP21}, using the notion of fiberegularity.
\end{example}

The third example is many respects the focus of the present work, as it treats a sufficient monotonicity in the gradient variables for the monotonicity-duality method when the gradient variables are present. In this section, we will limit ourselves to characterizing the subharmonics, which is interesting in its own right. The characterization of the dual subharmonics will be done in section \ref{sec:Char} in the general context of characterizing dual cone subharmonics.

\begin{example}[The directionality subequation]\label{exe:Dmon} Consider a {\em directional cone} $\cD \subset \R^n$ as defined in Definition \ref{defn:property_D}; that is, a closed convex cone with vertex at the origin and non empty interior $\intr \cD$. 
	The {\em directionality subequation} is the constant coefficient pure first order subequation $\cF = X \times \cM(\cD) = X \times (\R \times \cD \times \cS(n))$ reduces to $X \times \cD \subset X \times \R^n$ whose (constant) fibers are the directional cone $\cD$. The (reduced) dual subequation has (constant) fibers given by the Dirichlet dual 
\begin{equation}\label{dual_cone}
	\wt{\cD} := - (\intr \cD)^{\circ} \subset \R^n,
\end{equation}
which is also a directional cone. Two examples
 of directional cones were recalled in \eqref{parab_cone} and \eqref{pos_cone}.
	\end{example}  

The following characterization of $\cM(\cD)$ subharmonics is new. 	
	
\begin{prop}[Directionality subharmonics are increasing in polar directions] Suppose that $\cD \subset \R^n$ is a directional cone with polar cone\footnote{We follow the convention of \cite{CHLP22} in calling $\cD^{\circ}$ defined by \eqref{polar_cone} the polar cone determined by the set $\cD$. Some call this set the {\em dual cone} and denote it by $\cD^*$ and then define the polar cone as $-\cD^*$. Our choice avoids confusion with the (Dirichlet) dual cone \eqref{dual_cone}.}
\begin{equation}\label{polar_cone}
\cD^{\circ} = \{q \in \R^n : q \cdot p \ge 0 \ \ \forall p \in \cD \}.
\end{equation}
The set of $\cM(\cD)$-subharmonics can be characterized as follows: $u \in \USC(X)$ is $\cM(\cD)$-subharmonic on $X$ if and only if 
\begin{equation}\label{Dmon}
u(x) - u(x_0) \ge 0 \quad \text{for every $x, x_0 \in X$ such that $[x_0, x] \subset X,\ x-x_0 \in \cD^\circ$}. 
\end{equation}
\end{prop}

\begin{proof}

Indeed, assume first that \eqref{Dmon} holds. By Definition \ref{defn:Fsub}, we need to show that for each $x_0 \in X$ and for each upper test jet $(\varphi(x_0), D\varphi(x_0), D^2 \varphi(x_0)) \in J^{2,+}_{x_0} u$ we have $D\varphi(x_0) \in \cD$, which is equivalent to 
\begin{equation}\label{direct1}
\langle D\varphi(x_0),  q \rangle \ge 0 \qquad \forall \, q \in \cD^\circ,\ |q| = 1.
\end{equation}
For any unit vector $q$ in $\cD^{\circ}$ to be used in \eqref{direct1}, consider $x = x_0 + r q$ with $r > 0$. Since $\varphi$ is an upper test function for $u$ in $x_0$ we have 
$u(x) - u(x_0) \le \varphi(x) - \varphi(x_0)$ for each $r > 0$ sufficiently small. In addition, since $x_0 \in X$ with $X$ open and $\cD^{\circ}$ is a cone, we have $[x_0,x] \subset X$ and $x - x_0 = rq \in \cD^{\circ}$. Therefore, by a Taylor expansion of $\varphi$ and \eqref{Dmon},
\[
0 \le u(x) - u(x_0) \le \varphi(x) - \varphi(x_0)  = r \left( \langle D\phi(x_0), q \rangle + o(1) \right), \ \ r \to 0^+.
\]
Dividing by $r >0$ and taking the limit $r \to 0^+$ yields $\langle D\phi(x_0) , q \rangle \ge 0$, which is the desired inequality \eqref{direct1}.

To show the other implication, we need some machinery from nonsmooth and convex analysis. First, for an $\cM(\cD)$-subharmonic function $u$ we consider the sequence of sup-convolutions $u^\veps$ (see \eqref{sup_conv} below). We have that $u^{\veps}$ is $\frac{1}{\veps}$-quasi-convex and decreases pointwise to $u$ as $\veps \dto 0$. Moreover, for any $\Omega \subset\subset X$, since $\cM(\cD)$ has constant coefficients, $u^{\veps}$ is $\cM(\cD)$-subharmonic on $\Omega$ for $\veps$ small enough and, by Alexandroff's theorem, $u^{\veps}$ is almost everywhere twice differentiable, so
\[
D u^{\veps}(x) \in \cD \qquad \text{for a.e. $x \in \Omega$.}
\]
Note that for any such point, $D u^{\veps}(x)$ represents the generalized subgradient $\partial u^{\veps}(x)$ (see for example \cite[Section 2]{Clarke}). In fact, for every $x \in \Omega$, $\partial u^{\veps}(x)$ is given by the convex hull of limit points of (converging) sequences $D u^{\veps}(x_n)$, where $x_n \to x$ (see \cite[Theorem 2.5.1]{Clarke}). Since we can choose $x_n$ such that $D u^{\veps}(x_n) \in \cD$ and $\cD$ is a closed convex cone, we get that $\partial u^{\veps}(x) \subseteq \cD$ for every $x \in \Omega$, and therefore $\langle \partial u^{\veps}(x) , q \rangle$ is a subset of nonnegative reals for every $q \in \cD^\circ$ and $x \in \Omega$. Finally, if $[x_0, x] \subset \Omega$ and $x-x_0 \in \cD^\circ$, one applies Lebourg's Mean Value Theorem \cite[Theorem 2.3.7]{Clarke} to obtain for some $\xi \in (x_0, x)$
\[
u^{\veps}(x) - u^{\veps}(x_0) \in \langle \partial u^{\veps}(\xi) , (x-x_0) \rangle,
\]
so that $u^{\veps}(x) \ge u^{\veps}(x_0)$. Passing to the limit $\veps \to 0^+$ yields the defired conclusion \eqref{Dmon}.
\end{proof}

\section{Fiberegularity}\label{sec:FR}

In this section we discuss a fundamental notion which is crucial in the passage from constant coefficient subequations (and operators) to ones with variable coefficients. We begin with the definition.

\begin{defn}\label{defn:fibereg} A subequation $\cF \subset \cJ^2(X)$ is {\em fiberegular} if the fiber map $\Theta$ of $\cF$ is {\em (Hausdorff) continuous}; that is, if the set-valued map
	$$
	\Theta: X \to \scK(\cJ^2) \ \ \text{defined by} \ \  \Theta(x):= \cF_x,\ \ x \in X
	$$
	is continuous when the closed subsets $\scK(\cJ^2)$ of $\cJ^2$ are equipped with the {\em Hausdorff metric}
	$$
	d_{\scH}(\Phi, \Psi) := {\rm max} \left\{ \sup_{J \in \Phi} \inf_{J' \in \Psi} \trinorm{J - J'} , \sup_{J' \in \Psi} \inf_{J \in \Phi} \trinorm{J - J'} \right\}
	$$
	where 
	$$ 
	\trinorm{J}  = \trinorm{(r,p,A)} := \max \left\{ |r|, |p|, \max_{1 \leq k \leq n} |\lambda_k(A)| \right\}
	$$ 
	is taken as the norm on $\cJ^2$ where $\lambda_1(A) \leq \cdots \lambda_n(A)$ are the (ordered) eigenvalues of $A \in \cS(n)$. We will also make use of some standard facts concerning the Hausdorff distance in the proof of Proposition \ref{unifcontequiv} below;   these facts will be recalled in Appendix \ref{AppC} for the reader's convenience.
\end{defn}
 
This notion was first introduced in \cite{CP17} in the special case $\cF \subset X \times \cS(n)$. We will also refer to $\Theta$ as a {\em continuous proper ellipitc map} since it takes values in the closed (non-empty and proper) subsets of $\cJ^2$ satisfying properties (P) and (N). 

Note that by the Heine--Cantor Theorem, fiberegularty is equivalent to the local uniform continuity of the fiber map $\Theta$. Moreover, if $\cF$ is $\cM$-monotone for some (constant coefficient) monotonicity cone subequation, fiberegularity has more useful equivalent formulations. 

\begin{prop}[Fiberegularity of $\cM$-monotone subequations] \label{unifcontequiv}
	Let $\cF$ be an $\cM$-monotone subequation on $X$ with fiber map $\Theta:X \to \scK(\call J^2)$. Then the following are equivalent:
	\begin{enumerate}[label=(\alph*)]
		\item $\Theta$ is locally uniformly continuous, that is for each $\Omega\ssubset X$ and every $\eta > 0$ there exists $\delta = \delta(\eta, \Omega) > 0$ such that
		\[
		x,y\in \Omega,\ |x-y|<\delta  \implies  d_{\scr H}(\Theta(x), \Theta(y)) < \eta;
		\]
		\item $\Theta$ is locally uniformly upper semicontinuous (in the sense of multivalued maps), that is for each $\Omega\ssubset X$ and every $\eta > 0$ there exists $\delta = \delta(\eta, \Omega) > 0$ such that 
		\[
		\Theta(B_\delta(x)) \subset N_\eta(\Theta(x)) \quad \forall x \in \Omega,
		\]
		where $N_\epsilon(S) \defeq \{ J \in \call J^2 :\  \inf_{J'\in S}\trinorm{J-J'} < \epsilon \}$ is the $\epsilon$-enlargement of the set $S$;
		\item there exists $J_0 \in \intr \call M$ such that for each fixed $\Omega\ssubset X$ and $\eta > 0$ there exists $\delta = \delta(\eta, \Omega) > 0$ such that
		\begin{equation}\label{FR_M}
		x,y\in \Omega,\ |x-y|<\delta  \implies  \Theta(x) + {\eta J_0} \subset \Theta(y).
		\end{equation}
		Moreover, the validity of this property for one fixed $J_0 \in \intr \call M$ implies the validity of the property for each $J_0 \in \intr \call M$.
	\end{enumerate}
\end{prop}

Formulation (c) is the most useful definition of fiberegularity for $\cM$-monotone subequations. In the pure second order and gradient-free cases there is a ``canonical'' reduced jet $J_0 = I \in \cS(n)$ and $J_0 = (-1,I) \in \R \times \cS(n)$, respectively.

\begin{proof}[Proof of Proposition \ref{unifcontequiv}]
		What follows is an adaptation of the proofs of \cite[Propositions 4.2 and 4.4]{CP17}.
	
		\medskip
		\noindent \underline{{(a)} implies {(c)} for every $J_0 \in \intr\call M$.} \ By definition \eqref{defhaus} we have, for $\call I, \call K \subset \call J^2$,
		\begin{equation} \label{hausdistform}
		d_{\scr H}(\call I, \call K) = \max\left\{\sup_{J\in \call I}\inf_{J' \in \call K} \trinorm{J-J'},\ \sup_{J'\in \call K}\inf_{J \in \call I} \trinorm{J-J'} \right\}.
		\end{equation}
		Fix now $J_0 \in \intr{\call M}$ and $\eta > 0$; if $\Theta$ is uniformly continuous on $\Omega$, then, for $\eta' > 0$ to be determined, there exists $\delta = \delta(\eta', \Omega)$ such that 
		\[
		x,y \in \Omega,\ |x-y| < \delta \implies \inf_{J' \in \Theta(y)} \trinorm{J-J'} < \eta' \quad \forall J \in \Theta(x).
		\]
		Hence for $x, y \in \Omega$ with $|x-y| < \delta$ one has
		\[
		\forall J \in \Theta(x) \ \; \exists J' \in \Theta(y) \ \text{such that $K \defeq J-J'$ satisfies $\trinorm{K} < \eta'$};
		\]
		that is
		\begin{equation} \label{5d4.19}
		\forall J \in \Theta(x) :\ \text{ $J=J'+K$ with $J'\in \Theta(y)$ and $\trinorm{K}< \eta'$}.
		\end{equation}
		We want to show that for each $J \in \Theta(x)$, one has
		\begin{equation} \label{5d4.20}
		J + \eta J_0 \in \Theta(y).
		\end{equation}
		Using the decomposition (\ref{5d4.19}),
		\[
		J + \eta J_0 = J' + (K + \eta J_0) \quad \text{where $J' \in \Theta(y)$}
		\]
		so that, by the $\call M$-monotonicity of $\Theta(y)$, one has (\ref{5d4.20}) provided that 
		\begin{equation} \label{5d4.21}
		K + \eta J_0 \in \call M.
		\end{equation}
		Since $J_0 \in \intr\call M$ and $\call M$ is a cone, there exists $\rho = \rho(\eta,J_0) > 0$ such that $\call B_\rho(J_0) \subset \call M$, where we denoted by $\call B$ the ball in $\call J^2$. Therefore (\ref{5d4.21}) holds for $\eta' < \rho$.
		
		\medskip
		\noindent\underline{{(c)} for any fixed $J_0 \in \intr \call M$ implies {(b)}.} \ Fix $\eta > 0$ and choose any $J_0 \in \intr{\call M}$; let $\delta = \delta(\eta', \Omega, J_0)$ as in \emph{(c)}, with $\eta'<\eta/\trinorm{J_0}$. For each $x \in \Omega$ and $y \in B_\delta(x) \subset \Omega$ we have
		\[
		\Theta(y) + \eta' J_0 \subset \Theta(x),
		\]
		hence 
		\[
		\Theta(y) \subset \Theta(x) - \eta' J_0 \subset N_\eta(\Theta(x)). 
		\]
		
		\medskip
		\noindent\underline{{(b)} implies {(a)}} \  This is a standard proof which does not require any monotonicity assumption. For $\eta > 0$ fixed, let $\delta = \delta(\eta', \Omega)$ be as in {\it(b)}, with $\eta'< \eta$. For $x,y \in \Omega$ such that $|x-y| < \delta$ one has
		\begin{equation} \label{symmetricxy}
		\Theta(x) \subset N_{\eta'}(\Theta(y)) \quad \text{and} \quad \Theta(y) \subset N_{\eta'}(\Theta(x)),
		\end{equation}
		hence, thanks to the first inclusion, for every $J \in \Theta(x)$ there exists $J' \in \Theta(y)$ such that $J = J' + K$ for some $K$ with $\trinorm{K} < \eta'$. Therefore
		\[
		\inf_{J' \in \Theta(y)} \trinorm{J-J'} < \eta' \quad \forall J \in \Theta(x),\ \forall y \in B_\delta(x),
		\]
		which yields
		\[
		\sup_{J \in \Theta(x)}\inf_{J' \in \Theta(y)} \trinorm{J-J'} \leq \eta' \quad \text{whenever $|x-y| < \delta$}.
		\]
		By the second inclusion in (\ref{symmetricxy}), one also has
		\[
		\sup_{J' \in \Theta(y)}\inf_{J \in \Theta(x)} \trinorm{J-J'} \leq \eta' \quad \text{whenever $|x-y| < \delta$},
		\]
		and thus by (\ref{hausdistform})
		\[
		d_{\scr H}(\Theta(x), \Theta(y)) \leq \eta' < \eta. \qedhere
		\]
	\end{proof}

Fiberegularity is crucial since it implies the {\bf {\em uniform translation property}} for subharmonics. This property is the content of the following result, which roughly speaking states that: {\em if $u \in \cF(\Omega)$, then there are small $C^2$ strictly $\cF$-subharmonic perturbations of \underline{all small translates} of  $u$ which belong to  $\cF(\Omega_{\delta})$}, where  $\Omega_{\delta}:= \{ x \in \Omega: d(x, \partial \Omega) > \delta \}$.

\begin{thm}[Uniform translation property for subharmonics]\label{thm:UTP}
	Suppose that a subequation $\cF$ is fiberegular and $\cM$-monotone on $\Omega \ssubset \R^n$ for some monotonicity cone subequation $\cM$. Suppose that $\cM$ admits a strict approximator \footnote{The term strict approximator for $\psi$ refers to the fact that this function generates an approximation from above of the $\wt{\cM}$-subharmonic function which is identically zero. This is explained in the proof of Theorem 6.2 of \cite{CHLP22}.}; that is, there exists $\psi \in \USC(\overline{\Omega}) \cap C^2(\Omega)$ which is strictly  $\cM$-subharmonic   on  $\Omega$. Given $u \in \cF(\Omega)$, for each $\theta > 0 $ there exist $\eta = \eta(\psi, \theta) > 0$ and $\delta = \delta(\psi, \theta) > 0$ such that
	\begin{equation}\label{uythetadef}
	u_{y,\theta} = \tau_yu + \theta \psi \ \ \text{belongs to} \ \cF(\Omega_{\delta}), \ \ \forall \, y \in B_{\delta}(0),
	\end{equation}
	where $\tau_y u(\, \cdot \, ):= u(\, \cdot - y)$.
\end{thm}

\begin{proof}
	We are going to use the Definitional Comparison \Cref{defcompa} in order to adapt the method used in the proofs of the pure second-order and the gradient-free counterparts of this uniform translation property (see~\cite[Proposition~3.7(5)]{CP17} and~\cite[Proposition~3.7(4)]{CP21}).
	
	Fix $J_0 \in \intr{\call M}$ and let $\delta = \delta(\eta, \Omega)$ be as in \Cref{unifcontequiv}{(c)}, with $\eta > 0$ to be determined. Consider $\Omega' \ssubset \Omega_\delta$ and $v \in C^2(\Omega') \cap \USC(\bar{\Omega'})$, strictly $\tildee{\call F}$-subharmonic on $\Omega'$. In order to prove the subharmonicity of $u_{y;\theta}$ (defined as in~(\ref{uythetadef})) via the definitional comparison (cf.~\Cref{appldefcompa}), it suffices to show that, for a suitable $\eta$,
	\begin{equation} \label{utp:defcompaimpl}
	\exists\, x_0 \in \Omega' :\ (u_{y;\theta} + v)(x_0) > 0 \quad \implies \quad \exists\, y_0 \in \de\Omega' :\ (u_{y;\theta} + v)(y_0) > 0.
	\end{equation}
	Fix $\theta > 0$ and for $y \in B_\delta$ consider the function
	\[
	\hat v_{y;\theta} \defeq \tau_{-y} v + \theta \tau_{-y}\psi,
	\]
	defined on $\Omega' + y$, which satisfies
	\begin{equation} \label{twolinej2}
	\begin{split}
 J^2_{x-y} \hat v_{y;\theta} &=  J^2_x v + \theta  J^2_x \psi =  J^2_x v + \eta J_0 + \theta\Big(  J^2_x \psi - \frac\eta\theta J_0 \Big) \qquad \forall x \in \Omega'.
	\end{split}
	\end{equation}
	By \Cref{unifcontequiv}{(c)},\footnote{It is easy to see that in the proof of {(a) $\implies$ (c)} one can choose $J' \in \intr\Theta(y)$. Then, if $J \in \intr\Theta(x)$,  one uses the elementary fact that $\intr{\call F}_x + \call M \subset \intr{\call F}_x$.}
	\[
	J^2_{x} v + \eta J_0 \in \intr{\tildee{\call F}}_{x-y} \qquad \forall x \in \Omega',
	\]
	therefore, by the $\call M$-monotonicity of $\tildee{\call F}$, and using~(\ref{twolinej2}),
	\begin{equation} \label{jxyinf}
	 J^2_{x - y} \hat v_{y;\theta} \in \intr{\tildee{\call F}}_{x-y} \qquad \forall x \in \Omega'
	\end{equation}
	provided that
	\begin{equation} \label{utp:inclinM}
	 J^2_x \psi - \frac\eta\theta J_0 \in \call M \qquad \forall x \in \Omega'.
	\end{equation}
	Since $\psi$ is a strict approximator for $\call M$ on $\bar\Omega$, we know that there exists $\rho(x) > 0$ such that $\call B_{\rho(x)}(J^2_x \psi) \subset \call M$; also, since $\psi \in C^2(\bar\Omega)$, we know that $\rho_0 \defeq \inf_{\Omega} \rho > 0$. Therefore it suffices to choose
	\begin{equation} \label{utp:uppboundeta}
	\eta < \frac{\theta\rho_0}{\trinorm{J_0}},
	\end{equation}
	in order for (\ref{utp:inclinM}) to hold for any $\Omega' \ssubset \Omega_\delta$. It is worth noting that the bound (\ref{utp:uppboundeta}) is independent of $\delta$, which does depend on $\eta$, and hence an $\eta$ satisfying~(\ref{utp:uppboundeta}) can be chosen.
	
	We have proved that $\hat v_{y;\theta} \in C^2(\Omega'-y) \cap \USC(\bar{\Omega' - y})$ is strictly $\tildee{\call F}$-subharmonic on $\Omega' - y \ssubset \Omega_\delta - y \subset \Omega$ for each $y \in B_\delta$, and we know that $u$ is $\call F$-subharmonic on $\Omega_\delta-y \subset \Omega$ by hypothesis. By our initial assumption, there exists $x_0 \in \Omega'$ such that
	\[
	(u + \hat v_{y;\theta})(x_0 - y) = (\tau_y u + \tau_y\hat v_{y;\theta})(x_0) = (u_{y;\theta} + v)(x_0) > 0 ;
	\]
	hence by the definitional comparison applied to $u$ and $\hat v_{y;\theta}$ on $\Omega'-y$, there exists $\tilde y_0 = y_0 - y \in \de (\Omega' - y) = \de\Omega' - y$ such that 
	\[
	0 < (u + \hat v_{y;\theta})(\tilde y_0) = (u_{y;\theta} + v)(y_0),
	\]
	thus proving implication (\ref{utp:defcompaimpl}).
\end{proof}

The uniform translation property of Theorem \ref{thm:UTP} will play a key role in the treatment of the variable coefficient setting, where one does not have translation invariance. In particular, it will be used to show that given a semicontinuous $\cF$-subharmonic function $u$ there are {\bf {\em quasi-convex approximations}} of $u$ which remain $\cF$-subharmonic provided that $\cF$ is fiberegular and $\cM$-monotone (see Theorem \ref{prop:approx2.0}).

\begin{remark}\label{rem:UTP}
Concerning the additional hypothesis that the monotonicity cone subequation $\cM$ admits a $C^2$ strict subharmonic, we note that in the pure second order and gradient-free cases ($\cF \subset \Omega \times \cS(n)$ and $\cF \subset \Omega \times (\R \times \cS(n))$, one always has a quadratic strict approximator $\psi$. Thus Theorem \ref{thm:UTP} holds for all continuous coefficient $\cF$ which are minimally monotone (with  $\cM = \cP \subset \cS(n)$ and $\cM = \cQ = \cN \times \cP \subset \R \times \cS(n)$ respectively). In the general $\cM$-monotone and fiberegular case this additional hypothesis will be essential in the proof of the so-called {\bf{\em Zero Maximum Principle}} (ZMP) of Theorem \ref{thm:zmp} 
 for the dual monotonicity cone $\wt{\cM}$. The (ZMP) is a key ingredient in the monotonicty-duality method for proving comparison, as wil be discussed below in Section \ref{sec:SAaC}. Moreover, the (constant coefficient) monotonicity cone subequations which admit strict approxiamtors are well understood by the study made in \cite{CHLP22} and will be recalled below in Theorem \ref{ffctcs} and the discussion which precedes the theorem.
\end{remark}

Fiberegularity of an $\cM$-monotone subequation has two additional consequences which are of use in treating existence by Perron's method. While we will not pursue existence here, we record the result for future use. A general property of uniformly continuous maps on some open subset $\Omega$ with boundary $\partial \Omega$ is the possibility to extend them to the boundary. One can prove that the $\cM$-monotonicity is preserved as well. Also, one can define in a natural way the \emph{dual fiber map} of $\Theta$ by
\[
\tildee\Theta(x) \defeq \tildee{\Theta(x)} \quad \forall x \in X;
\]
note that this is a pointwise (or fiberwise) definition, and that by a straightforward extension to variable coefficient subequations of the elementary properties of the Dirichlet dual collected in~\cite[Section~4]{HL09}, \cite[Section~3]{HL11} or~\cite[Proposition~3.2]{CHLP22}, it is clear that $\tildee\Theta$ is still $\call M$-monotone. Furthermore, it is uniformly Hausdorff-continuous if $\Theta$ is.

The following proposition, which extends \cite[Proposition 3.6]{CP21}, collects these two properties.

\begin{prop}[Extension and duality] \label{unifcontext} \label{todual}
	Let $\Theta$ be a uniformly continuous $\call M$-monotone map on $\Omega$. Then 
	\begin{enumerate}[label=(\alph*)]
		\item $\Theta$ extends to a uniformly continuous $\call M$-monotone map on $\bar\Omega$;
		\item $\tildee\Theta$ is uniformly continuous and $\call M$-monotone on $\Omega$.
	\end{enumerate}
\end{prop}

\begin{proof}

	\underline{\it(a)} \ We essentially reproduce the proof of \cite[Proposition 3.5]{CP17}. One extends $\Theta$ to $\bar x \in \de \Omega$ as a limit
	\begin{equation} \label{limitucext}
	\Theta(\bar x) = \lim_{k\to \infty} \Theta(x_k)
	\end{equation}
	where $\{x_k\}_{k\in \N} \subset \Omega$ is a sequence such that $\lim_{k\to\infty} x_k = \bar x$ and the limit in (\ref{limitucext}) is to be understood in the complete metric space $(\frk K(\call J^2), d_{\scr H})$. This limit exists since $\{x_k\}$ is Cauchy sequence and hence so is $\{\Theta(x_k)\}$ by the uniform continuity of $\Theta$. Moreover, this limit is independent of the choice of $\{x_k\}$, and we have the extension of $\Theta$ to $\de\Omega$ by performing this construction for each $\bar x \in \de\Omega$. The resulting extension is uniformly continuous and each $\Theta(\bar x)$ is closed by construction.
	It remains to show that the extension takes values in the set of $\call M$-monotone sets.
	First of all, each $\Theta(\bar x)$ is non-empty because $d_{\scr H}(\Theta(x), \emptyset) = +\infty$ for all $x \in \Omega$ (by property \Cref{hausinf}) and hence $\Theta(x_k) \not\to \emptyset$.
	As for the $\cM$-monotonicity of the limit set $\Theta(\bar x)$, note that by (\ref{hausdistform}) it is easy to show that $\Theta(\bar x)$ is the set of limits of all converging sequences $\{J_k\}$ in $\call J^2$ such that $J_k \in \Theta(x_k)$  for all $k$ (cf.\ \cite[Exercise 7.3.4.1]{BBI01}), hence given $J \in \Theta(\bar x)$ we have
	\[
	J = \lim_{k \to \infty} J_k  \quad \text{with $J_k\in \Theta(x_k)$} 
	\]
	and thus for each $\hat J \in \call M$
	\[
	J + \hat J = \lim_{k \to \infty} \big( J_k + \hat J \big) \quad \text{with $J_k + \hat J \in \Theta(x_k)$} 
	\]
	by the $\call M$-monotonicity of each $\Theta(x_k)$; hence $J + \hat J \in \Theta(\bar x)$ for each $J \in \Theta(\bar x)$ and each $\hat J \in \call M$.
	Finally, to prove that $\Theta(\bar x)$ is a proper subset of $\call J^2$, it suffices to invoke \Cref{lem:hausinfmon}; indeed, arguing as above, this guarantees that $\Theta(x_k) \not\to \call J^2$.
	
	\medskip
	\noindent\underline{\it(b)} \ We proceed as in the proof of \cite[Proposition 3.5]{CP17}. As we already noted, by the elementary properties of the Dirichlet dual (namely~\cite[Proposition~3.2, properties (2) and (6)]{CHLP22}), one knows that $\Theta$ is $\call M$-monotone if and only if $\tildee\Theta$ is (note that this is a fiberwise property); then we only need to show that $\tildee\Theta$ is uniformly continuous on $\Omega$. Since $\Theta$ is uniformly continuous on $\Omega$, by \Cref{unifcontequiv}{(c)}, for $\eta> 0$, $J_0 \in \intr{\call M}$, and a suitable $\delta = \delta(\eta, \Omega, J_0)$ one has
	\[
	\Theta(x) + {\eta J_0} \subset \Theta(y)
	\]
	whenever $x,y\in \Omega$ are such that $|x-y|<\delta$.
	Hence, by the above-mentioned elementary properties of the Dirichlet dual, one obtains
	\[
	\tildee\Theta(y)  \subset \tildee\Theta(x) - \eta J_0;
	\]
	that is,
	\[
	\tildee\Theta(y) + {\eta J_0} \subset \tildee\Theta(x),
	\]
	thus proving the uniform continuity of $\tildee\Theta$ by exploiting the equivalent formulation of \Cref{unifcontequiv}{(c)} again. 
\end{proof}

\begin{remark} \label{deltathetarel}
	It is worth noting that this proof shows that the relation between $\eta$ and $\delta$ is the same for both $\Theta$ and $\tildee\Theta$. 
\end{remark}

\section{Quasi-convex approximation and the Subharmonic Addition Theorem}\label{sec:QCA}

In this section, we present a final ingredient for the duality-monotonicity method for potential theoretic comparison; namely, the so-called {\bf {\em Subharmonic Addition Theorem}}. Roughly, it states that if one has a {\em jet addition formula}
\begin{equation}\label{JAF}
	\cG_x + \cF_x \subset \cH_x, \ \ \forall \, x \in X
\end{equation}
for subequations $\cG, \cF$ and $\cH$ in $\cJ^2(X)$, then one has a {\em subharmonic addition relation}
\begin{equation}\label{SAR}
\cG(X) + \cF(X) \subset \cH(X)
\end{equation}
for the associated spaces of subharmonics. This implication will take on considerable importance when combined with the fundamental monotonicity-duality formula of jet addition noted in \eqref{mono_dual} \begin{equation}\label{jet_add}
\cF_x + \cM_x \subset \cF_x \ \ \Longrightarrow \ \ \cF_x + \wt{\cF}_x \subset\wt{\cM}_x , \ \ \text{for each} \ x \in X. 
\end{equation}

Many results about $\cF$-subharmonic functions $u$, including the implication \eqref{JAF} $\Rightarrow$ \eqref{SAR}, are more easily proved 
if one assumes that $u$ is also locally quasi-convex. Then, one can make use of {\em quasi-convex approximation} by way of {\em sup-convolutions} to extend the result to semicontinuous $u$. In general, when the subequations have variable coefficients, the quasi-convex approximation will be $C^2$-perturbation (with small norm) of the sup-convolution. The quasi-convex approximation and subharmonic addition theorems in this section were essentially given in \cite{R20}, and extend those known for for fiberegular subequations independent of the gradient ~\cite{CP17, CP21}. 

\subsection{Quasi-convex approximations}\label{sec:PQCA}

We begin by recalling some basic notions. 

\begin{defn}
	A function $u: C \to \R$ is {\em $\lambda$-quasi-convex} on a convex set  $C \subset \R^n$ if there exists $\lambda \in \R_{+}$ such that $u + \frac{\lambda}{2} | \cdot |^2$ is convex on $C$. A function $u: X \to \R$ is {\em locally quasi-convex} on an open set  $X \subset \R^n$ if for every $x \in X$, $u$ is $\lambda$-quasi-convex on some ball about $x$ for some $\lambda \in \R_{+}$.  
\end{defn}
Such functions are twice differentiable for almost every\footnote{The relevant measure is Lebesgue measure on $\R^n$.}  $x \in X$ by a very easy generalization of Alexandroff's theorem for convex functions (the addition of a smooth function has no effect on differentiability). This is one of the many properties that quasi-convex functions inherit from convex functions. See \cite{PR22} for an extensive treatment of quasi-convex functions. Quasi-convex functions are used to approximate $u \in \USC(X)$ (bounded from above) by way of the {\em sup-convolution}, which for each $\veps > 0$ is defined by
\begin{equation}\label{sup_conv}
u^{\veps}(x) := \sup_{y\in X} \left( u(y) -\frac{1}{2 \veps} |y - x|^2 \right), \ \ x \in X.
\end{equation}
One has that $u^{\veps}$ is $\frac{1}{\veps}$-quasi-convex and decreases pointwise to $u$ as $\veps \searrow 0$. 

Now, making use of the the uniform translation property of Theorem \ref{thm:UTP}, we will prove the quasi-convex approximation result which is needed for the proof of the Subharmonic Addition Theorem in the case of fiberegular $\cM$-monotone subequations. This approximation result substitutes the constant coefficient result of~\cite[Theorem~8.2]{HL09}.  

\begin{thm}[Quasi-convex approximation] \label{prop:approx2.0}
	Suppose that a subequation $\cF$ is fiberegular and $\cM$-monotone on $\Omega \ssubset \R^n$ for some monotonicity cone subequation $\cM$ and suppose that $\cM$ admits a strict approximator $\psi$. Suppose that $u \in \call F(\Omega)$ is bounded, with $|u|\leq M$ on $\Omega$. For every $\theta > 0$, let $\eta, \delta > 0$ be as in \eqref{uythetadef}.  Then there exists $\epsilon_* = \epsilon_*(\delta, M) > 0$ such that 
	\begin{equation} \label{scpert}
	u^\epsilon_\theta \defeq u^\epsilon + \theta \psi \in \call F(\Omega_\delta) \ \ \forall \, \epsilon \in (0, \epsilon_*), 
	\end{equation}
	where $\Omega_\delta \defeq \{ x \in \Omega:\ d(x, \de\Omega) > \delta \}$ and $u^{\veps}$ is the sup-convolution \eqref{sup_conv} of $u$.
\end{thm}

\begin{remark}
	The approximating function $u^\epsilon_\theta$ is quasi-convex, since it is the sum of a quasi-convex term, namely $u^\epsilon$, and a smooth term, namely $\theta \psi$, with Hessian bounded from below.
\end{remark}

\begin{proof}[Proof of Theorem \ref{prop:approx2.0}]
	By the uniform translation property (Theorem \ref{thm:UTP}), we know that 
	\[
	\scr F \defeq \big\{ u_{z;\theta}:\ |z| < \delta \big\} \subset \call F(\Omega_\delta).
	\]
	By the \emph{sliding property} (see~\Cref{elemprop}\textit{(iv)}) we also have
	\[
	\scr F_\epsilon \defeq \big\{ u_{z;\theta} - \tfrac1{2\epsilon}|z|^2:\ |z| < \delta \big\} \subset \call F(\Omega_\delta),
	\]
	and this family is locally bounded above. Therefore, by the \emph{families-locally-bounded-above property} of (see~\Cref{elemprop}\textit{(vii)}), the upper semicontinuous envelope $v_\epsilon^*$ of its upper envelope $v_\epsilon \defeq \sup_{w \in \scr F_\epsilon} w$ belongs to $\call F(\Omega_\delta)$. Now, a basic property of the sup-convolution is that it can also be represented as (for example, see~\cite[Section~8]{HL09}):
	\begin{equation} \label{supconvball}
	u^\epsilon = \sup_{z \in B_\delta} \Big( u(\cdot - z) - \frac1{2\epsilon}{|z|^2} \Big), \qquad \delta = 2\sqrt{\epsilon M}.
	\end{equation}
Hence, using the bound $|u| \leq M$, by choosing
	\begin{equation} \label{epsleqd2/4M}
	\epsilon \leq \frac{\delta^2}{4M},
	\end{equation}
	one has
	\[
	\sup_{w \in \scr G_\epsilon} w = u^\epsilon, \qquad \scr G_\epsilon \defeq \big\{u(\,\cdot - z)- \tfrac1{2\epsilon}|z|^2:\ |z|<\delta\big\},
	\]
	and thus $u_\epsilon^* \defeq (\sup_{w \in \scr G_\epsilon} w)^* = u^\epsilon$ since $u^\epsilon$ is upper semicontinuous. The desired conclusion now follows by noting that $v^*_\epsilon = u^*_\epsilon +  \theta \psi$. 
\end{proof}

\subsection{Subharmonic addition for fiberegular $\cM$-monotone subequations}

We will now make use of the quasi-convex approximation result of Theorem \ref{prop:approx2.0} to prove subharmonic addition.  Given the local nature of the definition of subharmonicity, we are going to use the following local argument: in order to prove that $\call F(X) + \call G(X) \subset \call H(X)$ it suffices to prove that $\call F(B) + \call G(B) \subset \call H(B)$ for one small open ball $B$ about each point of $X$. Therefore we will be in the situation where $\Omega = B \subset X$ can be chosen in such a way that $\call M$ indeed admits a strict approximator on $\bar\Omega$: for every $x \in X$, it suffices to consider a quadratic strict $\call M$-subharmonic on $B_r(x)$, for some $r>0$, (which we know there exists thanks to the topological property (T)) and then set $B = B_{r/2}(x)$.

In order to better understand the role of the assumptions of fiberegularity and $\cM$-monotonicity on the subequations, perhaps it is useful to review the argument in the constant coefficient case, as given in~\cite[Theorem~7.1]{CHLP22}. Suppose that $u \in \call F(X)$ and $v \in \call G(X)$ for a pair of subequations $\call F$ and $\call G$ and suppose that there exists a third subequation $\call H$ with $\call F + \call G \subseteq \call H$. As noted above, since the definition of $\call H$-subharmonic is local, in order to show that $u+v \in \call H(X)$ it is enough to show that  $u+v \in \call H(U_x)$ for some open neighborhood $U_x$ of each $x \in X$. At this point, it is known~\cite[Remark~2.13]{CHLP22} that if one chooses the $U_x$'s to be small enough, then property (T) ensures the existence of smooth (actually, quadratic) subharmonics $\phi_x \in \call F(U_x)$ and $\psi_x \in \call G(U_x)$. This is useful in order to apply another elementary property of the family of $\call F$-subharmonics (or $\call G$-subharmonics), namely the \emph{maximum property}~\cite[Proposition~D.1(B)]{CHLP22} (see also~\Cref{elemprop}\textit{(ii)}). This property says that: $u, v \in \call F(X) \ \Rightarrow \ \max\{u, v\} \in \call F(X)$. Applying the maximum property to the pairs $(u, \phi_x - m)$, $(v, \psi_x -m)$, for $m \in \N$, where $\phi_x - m$ and $\psi_x - m$ are subharmonic by the negativity property, one obtains two approximating truncated sequences of subharmonics  $u_m \in \call F(U_x)$, $v_m \in \call G(U_x)$, which are bounded on $U_x$ and decrease to the limits $u$, $v$, respectively as $m \to \infty$. The boundedness on $U_x$ allows one to apply~\cite[Theorem~8.2]{HL09} in order to produce, via the sup-convolution, two sequences of approximating quasi-convex subharmonics $u_m^\epsilon, v_m^\epsilon$, which are decreasing with pointwise limits $u_m, v_m$, respectively. Finally, one can now apply the Subharmonic Addition Theorem for quasi-convex functions~\cite[Theorem~5.1]{HL16} and the \emph{decreasing sequence property}~\cite[Section~4, property~(5)]{HL09} (or~\cite[Proposition~D.1(E)]{CHLP22}, or~\Cref{elemprop}\textit{(v)}) in order to conclude the proof.

The only obstruction to generalizing this constant coefficient proof to the case of variable coefficients is the need for a variable coefficient version of the constant coefficient quasi-convex approximation result of~\cite[Theorem~8.2]{HL09}. All of the other steps are known to be valid also in the variable coefficient case: the local existence of smooth subharmonics easily follows  (see~\cite[Remark~4.6]{R20} or \cite[Remark~2.1.6]{PR22} )  from the triad of topological conditions (T) which one requires a subequation to satisfy (cf.~\cite[Section~3]{HL11}); the maximum property is straightforward and the decreasing sequence property can be proven essentially as in~\cite{HL09}, by using the Definitional Comparison \Cref{defcompa} (see~\Cref{elemprop}).
Therefore if one uses \Cref{prop:approx2.0} instead of~\cite[Theorem~8.2]{HL09}, one has all the ingredients in order to carry out essentially the same proof. 

Actually, it is worth noting one final thing: the parameters $\theta,\epsilon,\delta$ in~(\ref{scpert}), which are to be sent to $0$, are linked in such a way that
\[
\text{\emph{a priori}, and in general one may suppose so, $\delta \dto 0\ \text{as}\ \theta \dto 0$ (cf.~(\ref{utp:uppboundeta}) and the def.~of $\delta$),}
\]
\[
\text{it is possible to let $\epsilon \dto 0$ with $\theta,\delta$ fixed (cf.~(\ref{epsleqd2/4M})),}
\]
\[
\text{letting $\theta \dto 0$ would force $\epsilon \dto 0$ as well (cf.~the relationships recalled above).}
\]
This suggests that one should first let $\epsilon \dto 0$ and then $\theta \dto 0$ (and thus $\delta \dto 0$).
Also, we have no \emph{a priori} information on the sign of the perturbing strict approximator in~(\ref{scpert}), namely $\theta\psi$; hence one cannot use the decreasing sequence property in order to deal with the limit $\theta \dto 0$.
Luckily enough, again thanks to the Definitional Comparison Lemma, another elementary property can be easily extended to variable coefficients, namely the \emph{uniform limits property}~\cite[Section~4, property~(5')]{CHLP22} (see~\Cref{elemprop}\textit{(vi)}); and the reader shall notice that, after computing the (decreasing) limit of $u^\epsilon_\theta$ as $\epsilon \dto 0$, one gets $u + \theta\psi$, which uniformly converges to $u$ as $\theta \dto 0$.

The theorem that we are going to state has a gradient-free analogue~\cite[Theorem~5.2]{CP21}, which has been proven by applying the same procedure, where~\cite[Lemma~5.6]{CP21} substitutes the quasi-convex approximation result~\cite[Theorem~8.2]{HL09}.

\begin{thm}[Subharmonic Addition for fiberegular $\cM$-monotone subequations] \label{sacd}
	Let $X \subset \R^n$ be open. Let $\call M$ be a constant coefficient monotonicity cone subequation and let $\call F, \call G \subset \call J^2(X)$ be fiberegular $\call M$-monotone subequations on $X$. For any subequation $\call H \subset \call J^2(X)$,
	\begin{equation*}\tag{Jet Addition}
	\cG_x + \cF_x \subset \cH_x, \ \ \forall \, x \in X
	\end{equation*}
implies
	\begin{equation*} \tag{Subharmonic Addition}
	\cG(X) + \cF(X) \subset \cH(X).
	\end{equation*}
	\end{thm}

\begin{proof}
	We have already outlined how a proof can be performed. For the sake of completeness, we give a brief sketch. Without loss of generality, suppose that $u \in \call F(X)$ and $v \in \call G(X)$ are bounded; indeed, if not, it suffices to proceed as follows:
	\begin{itemize}[leftmargin=*]
		\item for each $x \in X$, consider some ball $B \defeq B_\rho(x)$ and two quadratic subharmonics $\phi \in \call F(B)$ and $\psi \in \call G(B)$;
		\item for all $m \in\N$, define $u_m \defeq \max\{u, \phi-m\}$ and $v_m \defeq \max\{v, \psi-m\}$;
		\item prove the theorem for $u_m$ and $v_m$;
		\item apply the decreasing sequence property as $m \to \infty$.
	\end{itemize}
	Without loss of generality, also suppose that the fiber maps
	\[
	\Theta_{\call F}(x) \defeq \call F_x \quad\text{and}\quad \Theta_{\call G}(x) \defeq \call G_x
	\]
	are in fact uniformly continuous on $X$; indeed, again, if not, by the local nature of the definition of subharmonicity on $X$, it suffices to show that
	\[
	\call F(\Omega) + \call G(\Omega) \subseteq \call H(\Omega) \qquad \forall\Omega \ssubset X.
	\]
	Finally, as noted at the beginning of this subsection, after possibly choosing a smaller ball $B$, property (T) assures that $\call M$ admits a (quadratic) strict approximator on $\bar B$, so that we may assume without loss of generality that $\call M$ admits a strict approximator $\psi$ on $\bar X$.
	
	Thanks to these reductions, we are in the situation where all the hypotheses of Theorems~\ref{thm:UTP} and~\ref{prop:approx2.0} hold. Therefore we know that there exist two nets of quasi-convex functions
	\[
	u^\epsilon_\theta \in \call F(X_\delta), \quad v^\epsilon_\theta \in \call G(X_\delta)
	\]
	where the parameter $\delta$ is chosen as $\delta \defeq \min\{\delta_{\call F}, \delta_{\call G}\}$, where $\delta_{\call F}$ and $\delta_{\call G}$ are those coming from Theorem \ref{thm:UTP}, associated to the subequations $\call F$ and $\call G$, respectively. By the Subharmonic Addition Theorem for quasi-convex functions~\cite[Theorem~5.1]{HL16}, one has
	\[
	u^\epsilon_\theta + v^\epsilon_\theta \in \call H(X_\delta).
	\]
	Therefore, since we know that
	\[
	u^\epsilon_\theta + v^\epsilon_\theta \dto u + v + 2\theta\psi \quad \text{as}\ \epsilon \dto 0,
	\]
	by letting $\epsilon \dto 0$ the decreasing sequence property yields 
	\[
	u + v + 2\theta\psi \in \call H(X_\delta).
	\]
	Letting $\theta \dto 0$, by the uniform limit property and the fact that $X_\delta \uto X$ as $\delta \dto 0$,
	\[
	u+v \in \call H(X_{\delta^*}) \quad\text{for each $\delta^* > 0$ small}.
	\]
	This is equivalent to $u+v \in \call H(X)$, which is the desired conclusion.
\end{proof}

\section{Potential theoretic comparison by the monotonicity-duality method}\label{sec:SAaC}

In this section, we present a flexible method for proving {\bf {\em comparison}} (the comparison principle) in a fiberegular $\cM$-monotone nonlinear potential theory. The method works with sufficient monotonicity; that is, when the (constant coefficient) monotonicity cone subequation $\cM$ admits a {\bf {\em strict approximator}} $\psi$ on a given domain $\Omega \ssubset \R^n$, which we recall is a function $\psi \in \USC(\overline{\Omega}) \cap C^2(\Omega)$ that is strictly $\cM$-subharmonic on $\Omega$. Using monotonicity  and duality, comparison is a consequence of the following constant coefficient Zero Maximum Principle (ZMP). We give two versions. The first is the  ``elliptic'' version in Theorem 6.2 of \cite{CHLP22} which uses a boundary condition on the entire boundary. The second is a  ``parabolic'' version  which uses a boundary condition on a proper subset of the boundary and generalizes Theorem 12.37 of \cite{CHLP22}.

\begin{thm}[ZMP for dual constant coefficient monotonicity cone subequations] \label{thm:zmp}
Suppose that $\call M$ is a constant coefficient monotonicity cone subequation that admits a strict approximator on a domain $\Omega \ssubset \R^n$. Then the \emph{zero maximum principle} holds for $\tildee{\call M}$ on $\bar\Omega$; that is,
\[\tag{ZMP}
z \leq 0 \ \text{on $\de\Omega$} \quad \implies \quad z \leq 0 \ \text{on $\Omega$}
\]
for all $z \in \USC(\bar\Omega) \cap \tildee{\call M}(\Omega)$.

If, in addition, the strict approximator $\psi$ satisfies
\begin{equation}\label{SA_par}
\psi \equiv-\infty \text { on } \partial \Omega \setminus \partial^{-} \Omega
\end{equation}
for some $\partial^{-} \Omega \subset \partial \Omega$, then the zero maximum principle holds in the following form:
\[ \tag{ZMP\textsuperscript{--}}
z \leq 0 \ \text{on $\de^{-}\Omega$} \quad \implies \quad z \leq 0 \ \text{on $\Omega$}
\]
for all $z \in \USC(\bar\Omega) \cap \tildee{\call M}(\Omega)$.

\end{thm}

\begin{proof} The first statement has been shown in \cite[Theorem~6.2]{CHLP22}. To get its version on the ``reduced'' boundary $\de^{-}\Omega$, one may argue as follows. Since  $\intr\mathcal{M}$ has property (N) and since $\cM$ is a cone, one has
\begin{equation}\label{ZMP1}
\mbox{ $\varepsilon \psi-m$ is strictly $\mathcal{M}$-subharmonic on $\Omega$ \ \  for each $m > 0$ and each $\veps > 0$.}
\end{equation}
Moreover, since $\psi \in \USC(\overline{\Omega})$, there exists $M$ such that $\psi \leq M$ on $\bar{\Omega}$ and hence
\begin{equation}\label{ZMP2}
\mbox{$\varepsilon \psi-m \leq 0$ on $\bar{\Omega}$ \ \ for each $m > 0$ and each $\veps \in \left( 0, \frac{m}{M} \right)$.}
\end{equation}

On the one hand, $z \leq 0$ on $\partial^{-} \Omega$ by hypothesis and hence, by \eqref{ZMP2}, one has
\[
\mbox{ $z+\varepsilon \psi-m \leq 0 $ on $\partial^{-} \Omega$ \ \   for each $m > 0$ and each $\veps \in \left( 0, \frac{m}{M} \right)$.}
\]
On the other hand,
\[
z+\varepsilon \psi-m \leq 0 \text { on } \partial \Omega \setminus \partial^{-} \Omega,
\]
because $\psi\left(\partial \Omega \setminus \partial^{-} \Omega\right)=\{-\infty\}$, and $z$ is bounded from above on $\bar{\Omega}$.

Therefore $z+\varepsilon \psi-m \leq 0$ on $\partial \Omega$ where $z$ is $\wt{\cM}$-subharmonic on $\Omega$ by hypothesis and  $\veps \psi - m$ is $C^2$ and strictly $\cM$-subharmonic on $\Omega$ by \eqref{ZMP1}. Hence  
\begin{equation}\label{ZMP3}
z+\varepsilon \psi-m \leq 0 \text { on } \Omega
\end{equation}
by the Definitional Comparison (Lemma \ref{defcompa}) with $\cF=\widetilde{\cM}$ and $\widetilde{\cF}=\widetilde{\widetilde{\cM}}=\cM$. Taking the limit in \eqref{ZMP3} as $m, \epsilon \searrow 0$ gives $z \leq 0$ on $\Omega$.
\end{proof}

The following is a general result for fiberegular $\cM$-monotone nonlinear potential theories.

\begin{thm}[A General Comparison Theorem] \label{thm:GCT}
Let $\Omega \ssubset \R^n$ be a bounded domain. Suppose that a subequation $\call F \subset \call J^2(\Omega)$ is fiberegular and $\call M$-monotone on $\Omega$ for some monotonicity cone subequation $\call M$. If $\call M$ admits a strict approximator on $\Omega$, then \emph{comparison} holds for $\call F$ on $\bar\Omega$; that is,
\begin{equation}\tag{CP}\label{cp}
u \leq w \ \text{on $\de\Omega$} \quad \implies \quad u \leq w \ \text{on $\Omega$}
\end{equation}
for all $u \in \USC(\bar\Omega)$, $\call F$-subharmonic on $\Omega$, and $w \in \LSC(\bar\Omega)$, $\call F$-superharmonic on $\Omega$.

 If, in addition, the strict approximator is $-\infty$ on $\partial \Omega \setminus \partial^{-} \Omega$, for some $\partial^{-} \Omega \subset \partial \Omega$, then
\begin{equation}\tag{CP\textsuperscript{--}}\label{cpm}
u \leq w \ \text{on $\de^{-}\Omega$} \quad \implies \quad u \leq w \ \text{on $\Omega$}
\end{equation}
for all $u \in \USC(\bar\Omega)$, $\call F$-subharmonic on $\Omega$, and $w \in \LSC(\bar\Omega)$, $\call F$-superharmonic on $\Omega$.
\end{thm}

\begin{proof}
As noted in \eqref{Vsuper3}, by duality, $w$ is $\call F$-superharmonic on $\Omega$ if and only the function $v \defeq -w$ is $\tildee{\call F}$-subharmonic in $\Omega$. Hence the the comparison principle (CP) is equivalent to
\begin{equation}\tag{CP$'$} \label{CP'}
u + v \leq 0 \ \text{on $\partial \Omega$} \quad \implies \quad u + v \leq 0 \ \text{on $\Omega$}
\end{equation}
for all $u \in \USC(\overline{\Omega}) \cap \cF(\Omega)$ and $v \in \USC(\overline{\Omega}) \cap \wt{\cF}(\Omega)$. Obviously, \eqref{CP'} is equivalent to the zero maximum principle (ZMP) for $z \defeq u + v$, sums of $\cF$ and $\wt{\cF}$-subharmonics.

By elementary properties of the Dirichlet dual~\cite{HL11, CHLP22, R20} one knows that monotonicity and duality gives the jet addition formula
\[
\call F_x + \tildee{\call F}_x \subset \tildee{\call M}, \ \ \forall \, x \in \Omega,
\]
as noted in \eqref{mono_dual} and recalled in \eqref{jet_add}. Then the Subharmonic Addition Theorem~\ref{sacd} yields the subharmoniic addition relation
\[
\call F(\Omega) + \tildee{\call F}(\Omega) \subset \tildee{\call M}(\Omega).
\]
Therefore $z \in \tildee{\call M}(\Omega)$ and the desired conclusion follows from \Cref{thm:zmp}. The proof of \eqref{cpm} is completely analogous.
\end{proof}

As discussed in the introduction, 
the utility of the General Comparison Theorem \ref{thm:GCT} is greatly facilitated by the detailed study of monotonicity cone subequations in \cite{CHLP22}. For the convenience of the reader, we redroduce that discussion here. There is a three parameter {\em fundamental family} of monotonicity cone subequations (see Definition 5.2 and Remark 5.9 of \cite{CHLP22}) consisting of
\begin{equation}\label{fundamental_family1}
\cM(\gamma, \cD, R):= \left\{ (r,p,A) \in \cJ^2: \ r \leq - \gamma |p|, \ p \in \cD, \ A \geq \frac{|p|}{R}I \right\}
\end{equation}
where
\begin{equation}\label{fundamental_family2}
\gamma \in [0, + \infty), R \in (0, +\infty] \ \text{and} \ \cD \subseteq \R^n,
\end{equation}
where $\cD$ is a {\em directional cone}; that is, a closed convex cone with vertex at the origin and non-empty interior . The family is fundamental in the sense that for any monotonicity cone subequation, there exists an element $\cM(\gamma, \cD, R)$ of the familly with $\cM(\gamma, \cD, R) \subset \cM$ (see Theorem 5.10 of \cite{CHLP22}). Hence if $\cF$ is an $\cM$-monotone subequation, then it is $\cM(\gamma, \cD, R)$-monotone for some triple $(\gamma, \cD, R)$. Moreover, from Theorem 6.3 of \cite{CHLP22}, given any element $\cM = \cM(\gamma, \cD, R)$ of the fundamental family, one knows for which domains $\Omega \ssubset \R^n$ there is a $C^2$-strict $\cM$-subharmonic and hence for which domains $\Omega$ one has the (ZMP) for $\wt{\cM}$-subharmonics according to Theorem \ref{thm:zmp}. There is a simple dichotomy. If $R = + \infty$, then arbitrary bounded domains $\Omega$ may be used, while in the case of $R$ finite, any $\Omega$ which is contained in a translate of the truncated cone $\cD_R := \cD \cap B_R(0)$.

As a corollary, one has \Cref{ffctcs} below, which is the generalization to fiberegular subequations of the Fundamental Family Comparison Theorem~\cite[Theorem~7.6]{CHLP22}. The proof, which we omit, essentially amounts to showing that any fundamental cone defined in~\cite[Section~5]{CHLP22} (and recalled in \eqref{fundamental_family1}-\eqref{fundamental_family2}) admits a strict approximator on suitable domains (as showed in~\cite[proof of Theorem~6.3]{CHLP22}), in order to apply Theorem \ref{thm:GCT}.

\begin{thm}[The Fundamental Family Comparison Theorem] \label{ffctcs}
Let $\call F \subset \call J^2(\Omega)$ be a fiberegular $\call M$-monotone subequation on a bounded domain $\Omega \ssubset \R^n$, for some constant coefficient monotonicity cone subequation $\call M \subset \call J^2$. Suppose that  
\begin{enumerate}[label=\it(\roman*), leftmargin=*, parsep=3pt]
\item either $\call M \supset \call M(\gamma, \call D, R)$, for some $\gamma,R \in (0,+\infty)$ and some directional cone $\call D$, and $\Omega$ is contained in a translate of the truncated cone $\cD_R:= \call D \cap B_R(0)$\footnote{That is, there exists $y \in \R^n$ such that $\Omega - y \subset \call D_R$.}
\item or $\call M \supset \call M(\gamma, \call D, \call P)$ (that is, $\call M  \supset \call M(\gamma, \call D, R)$ with $R = +\infty$). 
\end{enumerate}
Then the comparison principle (\ref{cp}) holds on $\Omega$.
\end{thm}

\section{Characterizations of dual cone subharmonics}\label{sec:Char}

In this section, we will present characterizations of the subharmonics $\wt{\cM}(X)$ determined by the dual of a monotonicity cone subequation $\cM$. Before presenting the characterizations, a few remarks are in order. 

First, interest in such characterizations comes from the fact that the space of dual subharmonics $\wt{\cM}(X)$ on an open subset $X \subset \R^n$ associated to a constant coefficient monotonicity cone subequation $\cM \subset \cJ^2$ plays a key role in the monotonicity-duality method for proving comparison through the subharmonic addition theorem
\begin{equation}\label{recall_SAT}
\cF(X) + \wt{\cF}(X) \subset \wt{\cM}(X)
\end{equation}
if $\cF$ (and hence $\wt{\cF}$) is a fiberegular $\cM$-monotone subequation. This reduces comparison on a domain $\Omega \ssubset X$ to the zero maximum principle (ZMP) for $\wt{\cM}$-subharmonics, which is in turn implied by the existence of a strict approximator $\psi \in C^2(\Omega) \cap C(\overline{\Omega})$ (a strict $\wt{\cM}$-subharmonic on $\Omega$). Moroever, by \eqref{recall_SAT}, $\wt{\cM}(X)$ contains the differences of all $\cF$-subharmonics and $\cF$-superharmonics and $\cM$ has constant coefficients, even if $\cF$ does not.

Second, since
\begin{equation}\label{mono_mono}
\cM_1 \subset \cM_2 \ \ \Rightarrow \ \ \wt{\cM}_2 \subset \wt{\cM}_1,
\end{equation}
if one enlarges the monotonicity cone $\cM$, the chances of finding a strict approximator improve, while the space $\cM(X)$ reduces, yielding a weaker (ZMP). This ``monotonicity'' in the family of monotonicity cones \eqref{mono_mono} will be used in the characterizations we present.

Third, since
\begin{equation}\label{int_union}
\cM = \cM_1 \cap \cM_2 \ \ \Rightarrow \ \ \wt{\cM} = \wt{\cM}_1 \cup \wt{\cM}_2,
\end{equation}
and since the fundamental family $\cM(\gamma, \cD, R)$ is constructed from the intersection of eight elementary cones (see Definition 5.2 and Remark 5.9 of \cite{CHLP22}), 
one can use this fact in the proof of characterizations for cones in the fundamental family.

For a given monotonicity cone subequation $\wt{\cM} \subset \cJ^2$ and an open set $X \subset \R^n$, we will seek characterizations of 
\begin{equation}\label{DCSX}
\wt{\cM}(X) = \{ u \in \USC(X): \ u \ \text{is $\cM$-subharmonic on $X$} \}
\end{equation}
as well as
\begin{equation}\label{DCS_Omega}
\wt{\cM}(\overline{\Omega}) = \{ u \in \USC(\overline{\Omega}): \ u \ \text{is $\cM$-subharmonic on $\Omega$} \}, \ \ \Omega \Subset X
\end{equation}
in terms of {\em sub-$\cA$ functions} in the sense of the following definition.

\begin{defn} \label{def:subA}
	Given $X \subset \R^n$ and a collection of functions $\cA = \bigsqcup_{\Omega \Subset X} \cA(\overline{\Omega})$ where $\emptyset \neq \cA(\overline{\Omega}) \subset \LSC(\overline{\Omega})$ for each $\Omega$, a function $u \in \USC(X)$ is said to be {\em sub-$\cA$ on $X$} if $u$ satisfies the following comparison principle: for each $\Omega \Subset X$
	\begin{equation}\label{sub_A}
	u \leq a \  \ \text{on} \ \partial \Omega \ \ \Rightarrow \ \ u \leq a \  \ \text{on} \ \Omega, \quad \text{for each} \ a \in \cA(\overline{\Omega}).
	\end{equation}
	In this case we will write $u \in \SA(X)$. With $\Omega \Subset X$ fixed, we will also denote by 
	\begin{equation}\label{subA_Omega}
	\SA(\overline{\Omega}) = \{ u \in \USC(\overline{\Omega}): \ \eqref{sub_A} \ \text{holds for each} \ a \in \cA(\overline{\Omega}) \}.
	\end{equation}
\end{defn}

With respect to these definitions we will address two problems.

\begin{prob}\label{prob_X}
	Given a monotonicity cone subequation $\cM \subset \cJ^2$ and given an open set $X \subset \R^n$, determine a collection of functions $\cA = \bigsqcup_{\Omega \Subset X} \cA(\overline{\Omega})$ on $X$ such that
	\begin{equation}\label{Char1}
	\wt{\cM}(X) + \SA(X)
	\end{equation}
	where $\SA(X)$ is defined as in the first part of Definition \ref{def:subA}.
\end{prob}

\begin{prob}\label{prob_Omega}
	Given a monotonicity cone subequation $\cM \subset \cJ^2$ and given an open set $ \Omega \Subset \R^n$, determine a class of functions $ \cA(\overline{\Omega})$ on $\overline{\Omega}$ such that
	\begin{equation}\label{Char2}
	\wt{\cM}(\overline{\Omega}) = \SA(\overline{\Omega})
	\end{equation}
	where $\SA(\overline{\Omega})$ is defined as in the second part of Definition \ref{def:subA}.
\end{prob}

Before presenting some motivating examples and the general results, a few remarks are in order.

\begin{remark}\label{rem:problem_versions} A solution to Problem \ref{prob_X} will automatically solve Problem \ref{prob_Omega} for each $\Omega \Subset X$. We will see that a key role is played by domains $\Omega$ such that
	\begin{equation}\label{strict_M_sub}
	\text{ there exists a $C^2$-strictly $\cM$-subharmonic function on $\Omega$}.
	\end{equation}
	The property \eqref{strict_M_sub} holds for arbitrary $\Omega$ for many subequation cones $\cM$, but not all. Moreover, as noted at the beginning of the section, we are interested in the validity of the (ZMP) for $\cM$ on $\overline{\Omega}$, so Problem \ref{prob_Omega} is interesting in its own right.
\end{remark}

\begin{remark}\label{rem:monotonicity} In both versions, there is an obvious ``monotonicity property''
	\begin{equation}\label{monotonicity_A}
	\cA_1(\overline{\Omega}) \subset \cA_2(\overline{\Omega}) \ \ \Rightarrow \ \ \SA_2(\overline{\Omega}) \subset  \SA_1(\overline{\Omega}),	
	\end{equation}
	since increasing the test functions $a$ makes the sub-property \eqref{sub_A} more restrictive. Hence the inclusion
	\begin{equation}\label{Char_Sub}
	\wt{\cM}(\overline{\Omega}) \subset \SA(\overline{\Omega})
	\end{equation}
	is made easier for ``smaller'' classes $\cA$, while enlarging $\cA$ will sharpen \eqref{Char_Sub} and help in the reverse inclusion
	\begin{equation}\label{Char_Super}
	\wt{\cM}(\overline{\Omega}) \supset \SA(\overline{\Omega})
	\end{equation}
\end{remark}

We now begin to discuss some motivating examples. As noted in Examples 2.5 and 2.6, a characterization of the form \eqref{Char1} of Problem \ref{prob_X} is already known for two of the elementary monotonicity cone subequations in the fundamental family, which we recall in the following two examples.

\begin{example}[Subaffine functions]\label{exe:SA} If $\call M = \call M(\call P) := \R \times \R^n \times \call P$ is the {\em convexity (cone) subequation}, then the dual cone is
	$$
	\tildee{\call M} = \{ (r,p,A) \in \call J^2: \ A \in \tildee{\call P}\} = \{ (r,p,A) \in \call J^2 :\ \lambda_n(A) \geq 0 \}
	$$
	and $\tildee{\call M}(X) = \mathrm{S}\call A(X)$ where $\cA = \{ \cA(\overline{\Omega})\}_{\Omega \Subset X}$ with
	\begin{equation}\label{SA_char}
	\call A(\bar \Omega) = {\rm Aff}(\bar\Omega) \defeq \{ a|_{\bar\Omega} :\ a \text{ affine on } \R^n \}, \ \ \Omega \ssubset X.
	\end{equation}
	$\SA (X)$ with $\cA$ defined by \eqref{SA_char} is the space of {\em subaffine functions}. This example appears in connection with every pure second order subequation $\call F$ and every pure second order (degenerate) elliptic operator $F$.
\end{example}

\begin{example}[Subaffine-plus functions]\label{exe:SAP} If $\call M = \call M(\call N, \call P) := \call N \times \R^n \times \call P $ is the {\em convexity-negativity (cone) subequation}, then the dual cone is
	$$
	\tildee{\call M} = \{ (r,p,A) \in \call J^2: \ \ r \in \call N \ \text{or} \ A \in \tildee{\call P} \} = \{ (r,p,A) \in \call J^2 :\ r \leq 0 \ \text{or} \ \lambda_n(A) \geq 0 \}
	$$
	and $\tildee{\call M}(X) = \mathrm{S}\call A(X)$ where $\cA = \{ \cA(\overline{\Omega})\}_{\Omega \Subset X}$ with
	\begin{equation}\label{SAP_char}
	\call A(\bar\Omega) = {\rm Aff}^+(\bar\Omega) \defeq \{ a \in \Aff(\bar\Omega) :\ a \geq 0 \}, \ \ \Omega \ssubset X.
	\end{equation}
	$\SA (X)$ with $\cA$ defined by \eqref{SAP_char} is the space of {\em subaffine-plus functions}. This example appears in connection with every gradient-free subequation $\call F$ and every gradient-free proper elliptic operator $F$.
\end{example}

The next example of an  elementary monotonicity cone subequation in the fundamental family is, by iteslf, not particularly interesting. However, we record it anyway to make another point about intersections.

\begin{example}[Sub-plus functions]\label{exe:SZ} The {\em negativity (cone) subequation} $\call M = \call M(\call N) := \call N \times \R^n \times \call S(n) $ is self-dual; that is, $\tildee{\call M} = \call M(\call N)$, and $\tildee{\call M}(X) = \mathrm{S}\call A(X)$ where $\cA = \{ \cA(\overline{\Omega})\}_{\Omega \Subset X}$ with
	\begin{equation}\label{SZ_char}
	\call A(\bar\Omega) = {\rm Plus}(\bar\Omega) \defeq \{ a|_{\bar\Omega} :\ a \text{ quadratic},\ a|_{\overline{\Omega}} \geq 0\}, \ \ \Omega \ssubset X.
	\end{equation}
	$\SA (X)$ with $\cA$ defined by \eqref{SZ_char} is the space of {\em sub-plus functions}. 
\end{example}

\begin{remark}[On intersections]\label{rem:intersections}
	In Example \ref{exe:SAP}, the monotonicity cone $\call M(\call N, \call P)= \call M(\call P) \cap \call M(\call N)$ and the dual of $\call M(\call N, \call P)$ is the union of the dual cones of $\call M(\call P)$ and $\call M(\call N)$ in accordance with \eqref{int_union}. Moreover, considering the three examples taken together, if we denote by
	$$
	\call M_1 = \call M(\call P),\ \call M_2 = \call M(\call N), \quad \call M = \call M_1 \cap \call M_2,
	$$ 
	and
	$$
	\call A_1 = {\rm Aff}(X), \quad \call A_2 = {\rm Plus}(X)
	$$
	in addition to \eqref{int_union} we also have
	\begin{equation}
	\tildee{\call M}(X) = \mathrm{S}\call A(X) \quad \text{with} \quad \call A(\bar\Omega) = \call A_1(\bar\Omega) \cap \call A_2(\bar\Omega), \ \ \Omega \ssubset X.
	\end{equation}
	This consideration leads us to ask: {\em under what conditions is it true that}
	\begin{equation}\label{intersection1}
	\wt{\cM_1 \cap \cM_2}(\overline{\Omega}) = \mathrm{S}(\cA_1 \cap \cA_2)(\overline{\Omega})?
	\end{equation}
\end{remark} 

We will give general characterization results which also give sufficient conditions under which \eqref{intersection1} holds. We begin with a lemma on the ``reverse inclusion'' of \eqref{Char_Super} which exploits 
 part (b) of the Definitional Comparison Lemma \ref{defcompa}.

\begin{lem}\label{lem:SAT_reverse} Suppose that $\cM \subset \cJ^2$ is a monotonicity cone subequation. Then its dual subharmonics satisfy 
	\begin{equation}\label{reverse_inclusion}
	\SA(X) \subset \wt{\cM}(X)
	\end{equation}
	where $\cA = \{\cA(\overline{\Omega})\}_{\Omega \Subset X}$ with 
	\begin{equation}\label{define_A}
	\cA(\overline{\Omega}) := \{ a_{|\overline{\Omega}}: \ \text{$a$ is quadratic and $-a$ is $\cM$-subharmonic in $\Omega$} \}.
	\end{equation}
	Moreover, for any pair of monotonicity come subequations $\cM_1$ and $\cM_2$ and with $\cA_1$ and $\cA_2$ defined as in \eqref{define_A}, one has the intersection property
	\begin{equation}\label{intersection_reverse}
	\mathrm{S}(\cA_1 \cap \cA_2)(X) \subset \wt{\cM_1 \cap \cM_2}(X)
	\end{equation}
\end{lem}

\begin{proof}
	For the claim \eqref{reverse_inclusion}, we assume that $u \in  \SA(\overline{X})$ and we show that $u \in \wt{\cM}(X)$ by using part (b) of the Definitional Comparison Lemma with $v = -a$ quadratic. It is enough to show that for every $x_0 \in X$, there exist arbitrary small balls $B_{\rho}(x_0) \Subset X$ such that
	\begin{equation}\label{comparison_a}
	\mbox{ $u - a \leq 0$ on $\partial B_{\rho}(x_0) \Longrightarrow u - a \leq 0$ on $B_{\rho}(x_0)$,}
	\end{equation}
	for each quadratic $a$ such that $-a$ is strictly $\cM$-subharmonic on $B_{\rho}(x_0)$.
	But we have \eqref{comparison_a} on {\bf every} ball for {\bf all} quadratic $a$ such that $-a$ is merely  $\cM$-subharmonic on $B_{\rho}(x_0)$ (by the hypothesis that $u \in \SA(X)$ with $\cA$ defined by \eqref{define_A}.
	
	For the intersection property \eqref{intersection_reverse}, for each $\Omega \Subset X$, consider
	$$
	\cA(\overline{\Omega}) := \{ a_{|\overline{\Omega}}: \ \text{$a$ is quadratic and $-a$ is $\cM_1 \cap \cM_2$-subharmonic in $\Omega$} \} = 	\cA_1(\overline{\Omega}) \cap 	\cA_2(\overline{\Omega}),
	$$
	where the last equality is merely the observation that for quadratic ($C^2$) functions $a$,
	$$
	-a \in (\cM_1 \cap \cM_2)(\Omega) \ \Leftrightarrow \ J^2_x(-a) \in \cM_1 \cap \cM_2, \ \forall x \in \Omega,
	$$
	which is equivalent to $J^2_x(-a) \in \cM_k$ for each $ x \in \Omega$ for $k = 1,2$. By the first part, we conclude that $\mathrm{S}(\cA_1 \cap \cA_2)(X) = \SA(X) \subset \wt{\cM_1 \cap \cM_2}(X)$. 
\end{proof}

Notice that Lemma \ref{lem:SAT_reverse} implies that for each $\Omega \Subset X$ one has also the reverse inclusions 
\begin{equation}\label{reverse_inclusions}
\SA(\overline{\Omega}) \subset \wt{\cM}(\overline{\Omega}) \quad \text{and} \quad 	\mathrm{S}(\cA_1 \cap \cA_2)(\overline{\Omega}) \subset \wt{\cM_1 \cap \cM_2}(\overline{\Omega})
\end{equation}

Next we give a lemma on the ``forward inclusion'' \eqref{Char_Sub} and the forward inclusion in the intersection property \eqref{intersection1} on $\Omega \Subset X$ which satisfy property \eqref{strict_M_sub}.

\begin{lem}\label{lem:SAT} Suppose that $\Omega$ admits a $C^2$ strict $\cM$-subharmonic for some monotonicity cone subequation $\cM \subset \cJ^2$. Then the following hold.
	\begin{itemize}
		\item[(a)] $\wt{\cM}(\overline{\Omega}) \subset  \SA(\overline{\Omega})$ for any class  $\cA(\overline{\Omega})$ such that $- \cA(\overline{\Omega}) \subset \cM(\overline{\Omega})$; that is, if
		\begin{equation}\label{M_sub}
		\cA(\overline{\Omega}) \subset - \cM(\overline{\Omega}) = \{ w \in \LSC(\overline{\Omega}): \ \ - w \ \text{is $\cM$-subharmonic on $\Omega$} \}. 
		\end{equation}
		\item[(b)] In particular, with $\cA$ as defined in \eqref{define_A}; that is, with
		\begin{equation}\label{define_A1}
		\cA(\overline{\Omega}) := \{ a_{|\overline{\Omega}}: \ \text{$a$ is quadratic and $-a$ is $\cM$-subharmonic in $\Omega$} \}, 
		\end{equation}
		one has the forward inclusion $	\wt{\cM}(\overline{\Omega})  \subset \SA(\overline{\Omega})$. Moreover, for any pair $\cM_1$ and $\cM_2$ and with $\cA_1$ and $\cA_2$ defined as in \eqref{define_A1} one has
		\begin{equation}\label{forward_inclusions}
		\wt{\cM_1 \cap \cM_2}(\overline{\Omega}) \subset \mathrm{S}(\cA_1 \cap \cA_2)(\overline{\Omega}),
		\end{equation}
		provided that $\Omega$ admits a $C^2$ strict $(\cM_1 \cap \cM_2)$-subharmonic.
	\end{itemize}
\end{lem}

\begin{proof} For the proof of part (a), given $u \in \wt{\cM}(\overline{\Omega})$ the sub-$\cA$ property \eqref{sub_A} is equivalent to the (ZMP) for all differences $z:= u - a$ with $a \in \cA(\overline{\Omega})$; that is,
	\begin{equation}\label{SA_property_ZMP}
	\mbox{ $u - a \leq 0$ on $\partial \Omega \ \ \Rightarrow \ \ u - a \leq 0$ on $\Omega$, \ \ for each $a \in \cA(\overline{\Omega})$.}
	\end{equation}
	Since $\cM$ admits a $C^2$ strict $\cM$-subharmonic, the (ZMP) holds for each $z \in\wt{\cM}(\overline{\Omega})$. Hence it suffices to have the subharmonic difference formula
	\begin{equation}\label{SDF}
	\wt{\cM}(\Omega) - \cA(\Omega) \subset \wt{\cM}(\Omega),
	\end{equation}
	but this holds under the assumption $ - \cA(\Omega) \subset \cM(\Omega)$. Indeed,  for any monotonicity cone subequation $\cM \subset \cJ^2$, one has $\cM + \cM \subset \cM$ and hence by duality one has the jet addition formula
	$\wt{\cM} + \cM \subset \wt{\cM}$. Hence by the Subharmonic Addition Theorem \ref{sacd} for every open set $X$ one has
	\begin{equation}\label{SAT}
	\wt{\cM}(X) + \cM(X) \subset \wt{\cM}(X).
	\end{equation}	
	
	The first claim in part (b) is immediate from part (a) as the choice of $\cA$  in \eqref{define_A1} is one allowed by \eqref{M_sub}. Finally, assuming that $\Omega$ admits a $C^2$ strict $(\cM_1 \cap \cM_2)$-subharmonic, by part (a) one has 
	$$
	\wt{\cM_1 \cap \cM_2}(\overline{\Omega}) \subset \SA(\overline{\Omega}),
	$$
	for any $\cA(\overline{\Omega})$ such that
	$$
	\cA(\overline{\Omega}) \subset - (\cM_1 \cap \cM_2)(\overline{\Omega}) = \{ w \in \LSC(\overline{\Omega}): \ \ - w \ \text{is $(\cM_1 \cap \cM_2)$-subharmonic on $\Omega$} \},
	$$
	and in particular for 
	$$
	\cA(\overline{\Omega}) := \{ a_{|\overline{\Omega}}: \ \text{$a$ is quadratic and $-a$ is $(\cM_1 \cap \cM_2)$-subharmonic in $\Omega$} \}.
	$$ \end{proof}

Putting together Lemma \ref{lem:SAT_reverse} and Lemma \ref{lem:SAT}, we have the following general result, whose proof is immediate. 

\begin{thm}[Characterizing dual cone subharmonics]\label{thm:DCSC} Suppose that $\cM \subset \cJ^2$ is a monotoncity cone subequation. Then the following hold.
	\begin{itemize}
		\item[(a)] If $\Omega \Subset X$ admits a $C^2$ strict $\cM$-subharmonic, then $
		\wt{\cM}(\overline{\Omega}) = S \cA(\overline{\Omega})$ 	where 
		\begin{equation}\label{define_A_thm}
		\cA(\overline{\Omega}) := \{ a_{|\overline{\Omega}}: \ \text{$a$ is quadratic and $-a$ is $\cM$-subharmonic in $\Omega$} \}. 
		\end{equation}
		Moreover if $\Omega$ also admits a $C^2$ strict $(\cM_1 \cap \cM_2)$-subharmonic,  one has
		\begin{equation}\label{IP1}
		\wt{\cM_1 \cap \cM_2}(\overline{\Omega}) = S(\cA_1 \cap \cA_2)(\overline{\Omega}),
		\end{equation} 
		for  pairs $\cM_1, \cM_2$ and $\cA_1, \cA_2$ as defined in \eqref{define_A_thm}.
		\item[(b)] Consequently, if each $\Omega \Subset X$ admits a $C^2$ strict $\cM$-subharmonic, then 
		$$
		\wt{\cM}(X) = S \cA(X)
		$$ 
		for $\cA = \{ \cA(\overline{\Omega})\}_{\Omega \Subset X}$ with $\cA(\overline{\Omega)}$ as in \eqref{define_A_thm}. Moreover, if each $\Omega \Subset X$ admits a $C^2$ strict $(\cM_1 \cap \cM_2)$-subharmonic, one has  
		\begin{equation}\label{IP2}
		\wt{\cM_1 \cap \cM_2}(X) = S(\cA_1 \cap \cA_2)(X),
		\end{equation}
		for  pairs $\cM_1, \cM_2$ and $\cA_1, \cA_2$ as defined in \eqref{define_A_thm}.
	\end{itemize}
\end{thm}

Before proceeding to examine additional examples, including a discussion of characterizing the $\wt{\cM}$-subahrmonics for $\cM = \cM(\gamma, \cD, R)$ in the fundamental family, we record the following observation.

\begin{remark}\label{rem:optimality} In Theorem \ref{thm:DCSC}, provided that $\cM$ admits a $C^2$ strict $\cM$-subharmonic, we have characterizations  $\wt{\cM}(\overline{\Omega}) = S \cA(\overline{\Omega})$ with $\cA(\overline{\Omega})$ some class of quadratic functions easily determined by $\cM$; those quadratics $a$ such that $-a$ is $\cM$-subharmonic. However, it is not said that the characterization is optimal since it is possible that 
	\begin{equation}\label{non_uniqueness_A}
	S \cA_1(\overline{\Omega}) = S \cA_2(\overline{\Omega}) \ \ \text{even with} \ \cA_1(\overline{\Omega}) \subsetneq \cA_2(\overline{\Omega}).
	\end{equation}
	For example, by applying Theorem \ref{thm:DCSC} to Example \ref{exe:SA} with $\cM = \cM(\cP)$ the theorem gives $\cA_2(\overline{\Omega})$ as those quadratics $a$ such that $-a$ is $\cM(\cP)$-subharmonic; that is $a$ a concave quadratic. On the other hand, we know that the characterization holds for $\cA_1(\overline{\Omega})$  chosen as affine functions. Obviously affine quadratics are also concave and are the ``minimal'' concave quadratics. In this pure second order case, one has the deep study of Harvey-Lawson \cite{HL19} involving {\em edge functions}. Such improvements in the general case would be interesting.
\end{remark}

We now complete the discussion by presenting the characterizations of $\wt{\cM}$-subharmonics for all monotonicity cone subequations $\cM$ that belong to the fundamental family of cones introduced in \cite{CHLP22}. The family was recalled and briefly discussed beginning with the definition in \eqref{fundamental_family1}-\eqref{fundamental_family2}:
\begin{equation}\label{FFM}
\call M = \call M(\gamma, \call D, R) \defeq \bigg\{ (r,p,A) \in \call J^2 :\ r \leq -\gamma|p|, \ p \in \call D, \ A \geq \frac{|p|}R I \bigg\},
\end{equation}
where with $\gamma \in [0,+\infty)$, $\call D \subset \R^n$ a directional cone (a closed convex cone with vertex at the origin and non-empty interior), and $R \in (0,+\infty]$. 

We recall that in the limiting case $R = + \infty$ we interpret the last inequality in \eqref{FFM} as  
$$
A \geq \frac{|p|}R I \ \ \Leftrightarrow \ \ A \geq 0 \ \text{in} \ \cS(n) \ \ \Leftrightarrow \ \ A \in \cP.
$$
We recall also that the family is fundamental in the sense that for each monotonicity cone subequation $\cM \subset \cJ^2$, there exists a member of the fundamental family $\cM(\gamma, \cD, R)$ such that $\cM(\gamma, \cD, R) \subset \cM$ and hence by duality for each $\Omega \Subset X$
$$
\wt{\cM}(\overline{\Omega }) \subset \wt{\cM}(\gamma, \cD, R)(\overline{\Omega }).
$$
Hence the characterizations of all $\wt{\cM}(\gamma, \cD, R)(\overline{\Omega })$ will say something about the general case of $\wt{\cM}(\overline{\Omega })$.

The fundamental family $\wt{\cM}(\gamma, \cD, R)$ is generated by five elementary cones by taking double and triple intersections of the five generators, which  are:
\begin{equation}\label{gen1}
\cM(\cP) := \R \times \R^n \times \cP = \{ (r,p,A) \in \cJ^2: \ A \geq 0 \};
\end{equation}
\begin{equation}\label{gen2}
\cM(\cN) := \cN \times \R^n \times \cS(n) = \{ (r,p,A) \in \cJ^2: \ r \leq 0 \};
\end{equation}
\begin{equation}\label{gen3}
\cM(\cD) := \R \times \cD \times \cS(n) = \{ (r,p,A) \in \cJ^2: \ p \in \cD \},  \quad  \cD \subsetneq \R^n;
\end{equation}
\begin{equation}\label{gen4}
\call M(\gamma) \defeq \{ (r,p,A) \in \cJ^2 :\ r \leq - \gamma |p| \}, \quad \gamma \in (0, +\infty),
\end{equation}
\begin{equation}\label{gen5}
\call M(R) \defeq \left\{ (r,p,A) \in \call J^2 :\ A \geq \frac{|p|}{R}I \right\}, \quad R \in (0, +\infty).
\end{equation}
Examples \ref{exe:SA} and \ref{exe:SZ} characterize $\cM(X)$ for the generators in \eqref{gen1} and \eqref{gen2} respectively, where we note that for these two cones, for each $\Omega \Subset \R^n$ there are $C^2$-strict $\cM$-subharmonics. By exploiting Theorem \ref{thm:DCSC} (including the intersection properties), it suffices to characterize $\wt{\cM}(\overline{\Omega})$ for $\cM$ for the remaining generating comes \eqref{gen3} - \eqref{gen5} and to check when there are $C^2$-strict $\cM$-subharmonics for these generators and all of the intersections of the generators.

The following corollary addresses the remaining generators.

\begin{cor}\label{cor:gen345} Let $X \subset \R^n$ be open.
	\begin{itemize}
		\item[(a)] For any directional cone $\cD \subsetneq \R^n$, the monotonicity cone $\cM(\cD)$ defined in \eqref{gen3} has as dual cone $
		\wt{\cM}(\cD) = \{ (r, p, A) \in \cJ^2: \ p \in \wt{\cD} = - (\intr \cD)^c) \}
		$
		and one has $\wt{\cM}(\cD)(X) = S \cA_{\cD}(X)$ where $\cA_{\cD} = \{ \cA_{\cD}(\overline{\Omega})\}_{\Omega \Subset X} $ with 
		\begin{equation}
		\cA_{\cD}(\overline{\Omega}) = \left\{ a|_{\bar\Omega} :\ a \ \text{quadratic},\  Da \in - \call D \ \text{on $\Omega$} \right\}, \ \ \Omega \Subset X.
		\end{equation}
		\item[(b)] For any $\gamma \in (0, + \infty)$, the monotonicity cone $\cM(\gamma)$ defined in \eqref{gen4} has as dual cone $
		\wt{\cM}(\gamma) = \{ (r, p, A) \in \cJ^2: \  r \leq  \gamma |p| \}
		$
		and one has $\wt{\cM}(\gamma)(X) = S \cA_{\gamma}(X)$ where $\cA_{\gamma} = \{ \cA_{\gamma}(\overline{\Omega})\}_{\Omega \Subset X} $ with 
		\begin{equation}\label{char4}
		\cA_{\gamma}(\overline{\Omega}) = \left\{ a|_{\bar\Omega} :\ a \ \text{quadratic},\ a \geq \gamma |Da| \ \text{on $\Omega$} \right\}, \ \ \Omega \Subset X.
		\end{equation}
		\item[(c)] For any $R \in (0, +\infty)$,  the monotonicity cone $\cM(R)$ defined in \eqref{gen5} has as dual cone $\wt{\cM}(R) = \left\{ (r, p, A) \in \cJ^2: \ A + \frac{|p|}{R}I \in \wt{\cP} \right\}$ and for any $\Omega$ contained in a ball of radius $R$ one has  $\wt{\cM}(R)(\overline{\Omega}) = S \cA_{R}(\overline{\Omega})$ where
		\begin{equation}\label{char5}
		\call A_{R}(\bar\Omega) \defeq \left\{ a|_{\bar\Omega} :\ a \ \text{quadratic},\ D^2a \leq - \frac{|Da|}R I \ \text{on} \ \Omega \right\}.
		\end{equation}
	\end{itemize}
\end{cor}

\begin{proof}
	As shown in Chapter 6 of \cite{CHLP22}, one can find quadratic functions $\psi$ which are strictly $\cM$-subharmonic on all $\Omega \Subset \R^n$ for the cones $\cM(\cD)$ and $\cM(\gamma)$ and on all $\Omega$ contained in a ball of radius $R$ for the cone $\cM(R)$. Moreover, it is also shown that the (ZMP) fails for $\wt{\cM}(R)$-harmonics on balls of radius $R' > R$. Using Theorem \ref{thm:DCSC}, it suffices only to check that for each $\Omega \Subset X$ one has
	$$
	- \cA_{\cD}(\overline{\Omega}) \subset \cM(\cD)(\Omega), \quad - \cA_{\gamma}(\overline{\Omega}) \subset \cM(\gamma)(\Omega) \quad \text{and} \quad - \cA_{R}(\overline{\Omega}) \subset \cM(R)(\Omega),
	$$
	which are easily verified.
\end{proof}

\begin{remark}\label{rem:minimal_A} As noted in \eqref{non_uniqueness_A} of Remark \ref{rem:optimality},  it is not said that the classes of functions $\cA_{\cD}(\overline{\Omega}), \cA_{\gamma}(\overline{\Omega})$ and $\cA_R(\overline{\Omega})$ are the minimal classes for which the characterizations of Corollary \ref{cor:gen345} hold. For example, for each $R \in (0, +\infty)$, one can replace $\cA_R(\overline{\Omega})$ with
	\begin{equation}\label{A_R_minimal}
	\cA_{R, \min}(\overline{\Omega}) := \left\{ a|_{\bar\Omega} :\ a \ \text{quadratic},\ D^2a = - \frac{\max_{\overline{\Omega}} |Da|}{R} I \ \text{on} \ \Omega \right\}.
	\end{equation}
	Clearly $\cA_{R, \min}(\overline{\Omega})  \subsetneq \cA_{R}(\overline{\Omega})$ and the quadratics in $\cA_{R, \min}(\overline{\Omega})$ are ``minimal'' in the sense that they are the most ``concave'' quadratics in $\cA_{R}(\overline{\Omega})$. Also notice that taking the limit $R \to +\infty$ of  $\cA_{R, \min}(\overline{\Omega})$ yields the affine functions used to characterize the dual subharmonics for the limiting cone $\cM(\cP)$ of Example \ref{exe:SA}.
\end{remark}

The remaining twelve (distinct) cones $\cM(\gamma, \cD, R))$ in the fundamental family are formed by taking double and triple intersections of the five generators \eqref{gen1}-\eqref{gen5}. Seven of the intersections
$$
\cM(\cN, \cP), \cM(\cN, \cD), \cM(\cD, \cP), \cM(\gamma, \cP), \cM(\gamma, \cD), \cM(\gamma, \cP), \cM(\cN, \cD, \cP) \  \text{and} \ \cM(\gamma, \cD, \cP)
$$
do not use \eqref{gen5} $\cM(R)$ (with $R$ finite) admit quadratic strictly $\cM$-subharmonics all every $\Omega \Subset X$ and hence the intersection property \eqref{IP2} gives characterizations of the form $\wt{\cM}(X) = S \cA(X)$ with $\cA$ the corresponding intersections. Example \ref{exe:SAP} concerns $\cM(\cN, \cP)$ and another example of this type is worth recording, while the others are left to the reader.

\begin{example}[Subaffine-plus functions with directionality]\label{exe:SAPD} The fundamental product monotonicity cone $\cM = \cM(\cN, \cD, \cP) = \cM(\cN) \cap \cM(\cD) \cap \cM(\cP) =  \cN \times \cD \times \cP$, with directionality cone $\cD \subsetneq \R^n$
	has dual cone $\wt{\cM} = \cN \times \wt{\cD} \times \wt{\cP}$ and satisfies $\wt{\cM}(X) = S \cA(X)$ where $\cA = \{ \cA(\overline{\Omega})\}_{\Omega \Subset X}$ with
	\begin{equation}\label{SAPD}
	\call A(\bar\Omega) = \Aff_{\call D}^+(\bar\Omega) \defeq \big\{ a|_{\bar\Omega} :\ a \text{ affine},\ a|_{\bar\Omega} \geq 0 \ \text{and} \ Da \in - \intr \call D \ \text{on} \ \Omega \big\}.
	\end{equation}
	We will call ${\rm SA}_{\call D}^+(X)$ the space of \emph{subaffine-plus functions with directionality $\call D$}, or more simply \emph{$\call D$-subaffine-plus functions}.
\end{example}

Finally, as shown in Chapter 6 of \cite{CHLP22}, the remaining five cones
\begin{equation}\label{R_cones}
\cM(\cN, R), \cM(\gamma, R), \cM(\cD, R), \cM(\cN, \cD, R) \  \text{and} \ \cM(\gamma, \cD, R)
\end{equation}
which use  \eqref{gen5} $\cM(R)$ (with $R$ finite), admit quadratic strictly $\cM$-subharmonics on every domain $\Omega \Subset \R^n$ such that 
$$
\mbox{$\Omega$ is contained in a ball of radius $R$}
$$
in the first two cases of \eqref{R_cones} and on every domain $\Omega \Subset \R^n$ such that 
$$
\mbox{$\Omega$ is contained in a translate of the truncated cone $\cD_R:= \cD  \cap B_R(0)$}
$$
in the last three cases of \eqref{R_cones} with a directional cone $\cD \subsetneq \R^n$.

\section{Admissibility constraints and the Correspondence Principle}\label{sec:correspondence}

In this section, we will discuss how the potential theoretic comparison principles in nonlinear potential theory (using monotonicity, duality and fiberegularity) developed in the previous sections can be transported to many fully nonlinear second order PDEs. The equations we treat will be defined by a variable coefficient operator $F \in C(\cG)$ with domain $\cG \subset \cJ^2(X)$ which may or may not be all of $\cJ^2(X)$. Moreover, we will treat operators $F$ with dependence on all jet variables $J = (r,p,A) \in \call J^2$ with sufficient monotonicity with respect to some constant coefficient monotonicity cone subequation $\call M$. Hence, the operators will be {\bf {\em proper elliptic}} with an additional monotonicity in the gradient variables, a concept that we will call {\bf {\em directionality}}.  It is gradient dependence with directionality that distinguishes the present work with respect 
 to the pure second order and gradient free situations treated in \cite{CP17} and \cite{CP21} respectively.

\subsection{Viscosity solutions of PDEs with admissibility constraints}\label{sec:AVS}

We begin by recalling the class of operators with the necessary monotonicity required for the comparison principle. When there is gradient dependence, the additional monotonicity of directionality will also be required (see Definition \ref{defn:MMO}).

\begin{defn}[Proper elliptic operators]\label{defn:PEO} An operator $F \in C(\cG)$ where either
	\begin{equation*}\label{case1}
	\mbox{$\cG = \cJ^2(X)$ } \quad \text{({\bf {\em unconstrained case}})}
	\end{equation*}
	or
	\begin{equation*}\label{case2}
	\mbox{$\cG \subsetneq  \cJ^2(X)$ is a subequation constraint set  \quad \text{({{\bf \em constrained case}})}.}
	\end{equation*}
	is said to be {\em proper elliptic} if for each $x \in X$ and each $(r,p, A) \in \cG_x$ one has
	\begin{equation}\label{PEO}
	F(x,r,p,A) \leq F(x,r + s, p, A + P) \ \quad \forall \, s \leq 0 \ \text{in} \ \R \ \text{and} \  \forall \, P \geq 0 \ \text{in} \ \cS(n).
	\end{equation}
	The pair $(F, \cG)$ will be called a {\em proper elliptic \footnote{Such operators are often refered to as {\em proper} operators (starting from \cite{CIL92}). We prefer to maintain the term  ``elliptic'' to emphasize the importance of the {\em degenerate ellipticity} ($\cP$-monotonicity in $A$) in the theory.}
		(operator-subequation) pair}.
\end{defn}

The minimal monotonicity \eqref{PEO} of the operator $F$ parallels the minimal monotonicity properties (P) and (N) for subequations $\cF$. It is needed for {\em coherence} and eliminates obvious counterexamples for comparison. This is explained for subequations after Definition \ref{defn:Fsub}. A given operator $F$ must often be restricted to a suitable background constraint domain $\cG \subset \cJ^2(X)$ in order to have this minimal monotonicity (the constrained case). The historical example clarifying the need for imposing a constraint is the Monge-Amp\`ere operator
\begin{equation}\label{MAE}
F(D^2u) = {\rm det}(D^2u),
\end{equation}
where one restricts the operator's domain to be the convexity subequation $\cG = \cP := \{ A \in \cS(n): A \geq 0 \}$.

\begin{remark}\label{rem:CC_UC}
The scope of the constrained case is perhaps best illustrated by the more general {\em G\aa rding-Dirichlet operators} as discussed in Section 11.6 of \cite{CHLP22}, of which the Monge-Amp\`ere equation \eqref{MAE} represents the fundamental case. This class of operators are constructed in terms of {\em hyperbolic polynomials} in the sense G\aa rding (see Definition \ref{defn:hyp_poly}).  The unconstrained case, in which $F$ is proper elliptic on all of $\cJ^2(X)$ is the case usually treated in the literature and is perhaps best illustrated by the so-called {\em canonical operators} 
 associated to subequations with sufficient monotonicity, as discussed in Section 11.4 of \cite{CHLP22}. 
\end{remark}

We now recall the precise notion of subsolutions, supersolutions and solutions of a PDE
\begin{equation}\label{FNE} 
	F(J^2u) = 0 \ \ \text{on} \ X \subset \R^n.
\end{equation}
The notions again make use of {\em upper/lower test jets} which we recall are defined by
 \begin{equation}\label{UCJ2}
J^{2,+}_{x} u := \{ J^2_{x} \varphi:  \varphi \ \text{is} \ C^2 \ \text{near} \ x, \  u \leq \varphi \ \text{near} \  x \ \text{with equality at} \ x \},
\end{equation}
and 
 \begin{equation}\label{LCJ}
J^{2,-}_{x} u := \{ J^2_{x} \varphi:  \varphi \ \text{is} \ C^2 \ \text{near} \ x, \  u \geq \varphi \ \text{near} \  x \ \text{with equality at} \ x \},
\end{equation}

\begin{defn}[Admissible viscosity solutions]\label{defn:AVS} Given $F \in C(\cG)$ with either $\cG = \cJ^2(X)$ or $\cG \subsetneq \cJ^2(X)$ a subequation on an open subset $X \subset \R^n$:
	\begin{itemize}
		\item[(a)] a function $u \in \USC(X)$ is said to be a {\em ($\cG$-admissible) viscosity subsolution of $F(J^2u)=0$ on $X$} if for every $x \in X$ one has
		\begin{equation}\label{AVSub}
		\mbox{$J \in J^{2, +}_{x}u \ \ \Rightarrow \ \   J \in \cG_x$ \ \ \text{and} \ \ $F(x, J) \geq 0$;}
		\end{equation}
		\item[(b)] a function $u \in \LSC(\Omega)$ is said to be a {\em ($\cG$-admissible) viscosity supersolution  of $F(J^2u)=0$ on $X$} if for every $x \in X$ one has
		\begin{equation}\label{AVSuper}
		\mbox{$J \in J^{2, -}_{x}u  \ \ \Rightarrow$ \ \ either [ $J \in \cG_x$ and \ $F(x,J) \leq 0$\, ] \ or \ $J \not\in \cG_x$.}
		\end{equation}
	\end{itemize}
	A function $u \in C(X)$ is an {\em ($\cG$-admissible viscosity) solution of $F(J^2u)=0$ on $X$} if both (a) and (b) hold.
\end{defn}
In the unconstrained case where $\cG = \cJ^2(X)$, the definitions are standard. In the constrained case where $\cG \subsetneq  \cJ^2(X)$, the definitions give a systematic way of doing of what is sometimes done in an ad-hoc way (see \cite{IL90} for operators of Monge-Amp\`{e}re type and \cite{Tr90} for prescribed curvature equations.) Note that \eqref{AVSub} says that the subsolution $u$ is also $\cG$-subharmonic and that \eqref{AVSuper} is equivalent to saying that $F(x,J) \leq 0$ for the lower test jets which lie in the constraint $\cG_x$.

If $\cG$ is fiberwise constant, that is, 
\[
\cG_x = \mathcal E \quad \forall x \in X,
\]
for some $\mathcal E \subset \R \times \R^n \times \cS(n)$, then $\cG$-admissible viscosity sub/supersolutions will be for simplicity referred as $\mathcal E$-admissible viscosity sub/supersolutions.

\subsection{The Correspondence Principle} A crucial point in a nonlinear potential theoretic approach to study fully nonlinear PDEs is to establish the {\bf {\em Correspondence Principle}} between a given proper elliptic operator-subequation pair $(F, \cG)$ and a given subequation $\cF$. This correspondence consists of the two equivalences: for every $u \in \USC(X)$
\begin{equation}\label{Corr1}
u \ \text{is $\cF$-subharmonic on $X$} \ \Leftrightarrow u \ \text{is a subsolution of} \ F(J^2u) = 0 \ \text{on $X$} 
\end{equation}
and
\begin{equation}\label{Corr2}
u \ \text{is $\cF$-superharmonic on $X$} \ \text{on $X$} \ \Leftrightarrow  \ u \ \text{is a supersolution of} \ F(J^2u) = 0,
\end{equation}
where the subsolutions/supersolutions are in the $\cG$-admissible viscosity sense of Defintion \ref{defn:AVS}. By the definitions, the equivalence \eqref{Corr1} is the same as the following equivalence: for each $x \in X$ one has
\begin{equation}\label{Corr1'}
J^{2,+}_x u \subset \cF_x \ \Longleftrightarrow \ \text{both} \quad J^{2,+}_x u \subset \cG_x \quad \text{and} \quad  F(x,J) \geq 0 \ \ \text{for each} \ \ J \in J^{2,+}_x u.
\end{equation}
This holds if and only if one has the {\bf{\em correspondence relation}}
\begin{equation}\label{relation}
\cF = \{ (x,J) \in \cG: \ F(x,J) \geq 0 \}.
\end{equation}
In addition, the equivalence \eqref{Corr2} is the same as the following equivalence: for each $x \in X$ one has
\begin{equation}\label{Corr2'}
J^{2,+}_x (-u) \subset \wt{\cF}_x \ \Longleftrightarrow \  \ J \not\in \cG_x \ \text{or} \ [J \in \cG_x \ \text{and} \ F(x,J) \leq 0], \ \forall \,  J \in J^{2,-}_x u.
\end{equation}
Using duality \eqref{dual_fiber} and $J^{2,+}_x (-u) = -J^{2,-}_x u$ one can see that that the equivalence \eqref{Corr2'} holds if and only if one has {\bf {\em compatibility}}
\begin{equation}\label{compatibility1}
\intr \cF = \{ (x,J) \in \cG: \ F(x,J) > 0\},
\end{equation} 
which for subequations $\cF$ defined by \eqref{relation} is equivalent to
\begin{equation}\label{compatibility2}
\partial  \cF = \{ (x,J) \in \cG: \ F(x,J) = 0\}.
\end{equation}

These considerations can be summarized in the following result.

\begin{thm}[Correspondence Principle]\label{thm:corresp_gen}
	Suppose that $F \in C(\cG)$ is proper elliptic and $\cF$, defined by the correspondence relation \eqref{relation}, is a subequation. If compatibility \eqref{compatibility1} is satisfied, then the correspondence principle \eqref{Corr1} and \eqref{Corr2} holds. In particular, $u \in C(X)$ is a $\cG$-admissible viscosity solution of $F(J^2u) = 0$ in $X$ if and only if $u$ is $\cF$-harmonic in $X$.
\end{thm}

It remains to determine structural conditions on a given proper elliptic operator $F \in C(\cG)$ for which the hypotheses of the Correspondence Principle hold. There are the two requirements. First, one needs that the constraint set $\cF$ defined by the correspondence relation \eqref{relation} is, in fact, a subequation. The fiberwise monotoncity properties (P) and (N) for $\cF$ follow easily from the $\cM_0$-monotonicity of the proper elliptic pair $(F, \cG)$. More delicate is the topological property (T) and this will require additional monotonicity and regularity assumptions on the pair $(F, \cG)$. Also, in order to discuss the equation $F(J^2u) = 0$ on $X$ the following {\em non-empty} condition on the zero locus of $F$ is needed
\begin{equation}\label{GammaNE}
\mbox{$\Gamma(x) \defeq \big\{ J \in \call \cG_x:  \ F(x,J) = 0 \big\} \neq \emptyset$ \quad for each $x \in X$.}
\end{equation}
This assumption also insures that $\cF_x \neq \emptyset$ for each $x \in X$. Second, one needs the compatibility \eqref{compatibility1} (or equivalently \eqref{compatibility2} if $\cF$ is a subequation). 
This condition is usually easy to check in practice, where some strict monotonicity of $F$ near the zero locus of $F$ suffices.

We now address the question of sufficient conditions for having the first requirement of the Correspondence Principle for a given proper elliptic operator $F \in C(\cG)$; that is, under what (additional) conditions on the pair $(F, \cG)$ will the constraint set $\cF$ defined by the correspondence relation \eqref{relation} be a subequation? We will, in fact, do more. We will find conditions for which the constraint set $\cF$ is a fiberegular $\cM$-monotone subequation for some monotonicty cone subequation $\cM$ of the pair $(F, \cG)$. This will make the Correspondence Principle useful for proving comparison. To that end, we must impose the appropriate (additional) monotonicity on the operator-subequation pair $(F, \cG)$.

\begin{defn}[$\cM$-monotone operators]\label{defn:MMO} Let $\cM \subset \cJ^2$ be a (constant coefficient) monotonicty cone subequation and let $\cG \subset \cJ^2(X)$ be either $\cG = \cJ^2(X)$ or $\cG \subsetneq  \cJ^2(X)$ an $\cM$-monotone subequation. An operator $F \in C(\cG)$ is said to be {\em $\cM$-monotone} if 
	\begin{equation}\label{MMO}
	F(x,J + J') \geq F(x,J) \ \quad \forall \, x \in X, J \in \cG_x, J' \in \cM. 
	\end{equation}
	The pair $(F, \cG)$ will be called an {\em $\cM$-monotone (operator-subequation) pair}.
\end{defn}
Notice that all $\cM$-monotone operators are proper elliptic since any subequation cone $\cM \subset \cJ^2$ cone contains the minimal monotonicity cone $\cM_0 = \cN \times \{ 0\} \times \cP$; therefore, \eqref{MMO} implies \eqref{PEO}. Also note that in the gradient free case, any proper elliptic operator is $\cM$-monotone for the monotonicity subequation cone $\cM:= \cN \times \R^n \times \cP$. This is the case treated in \cite{CP21}.

Given an $\cM$-monotone operator $F \in C(\cG)$, the fiber map $\Theta$ of the constraint set $\cF$ defined by the compatibility relation \eqref{relation} will be $\cM$-monotone in the following sense.

\begin{defn}[$\cM$-monotone maps]\label{defn:MMM}
	Given a monotonicity cone subequation $\cM \subset \cJ^2$, a map $\Theta: X \to \scK(\cJ^2)$ (taking values in the closed subsets of $\cJ^2$) will be called an {\em $\cM$-monotone map} if
	\begin{equation}\label{MMMap}
	\Theta(x) + \cM \subset \Theta(x), \ \ \forall \, x \in X.
	\end{equation}
\end{defn}
Indeed, $\Theta(x) := \{J \in \cG_x: \ F(x,J) \geq 0 \}$ is closed by the continuity of $F$ and \eqref{MMMap} follows easily from \eqref{MMO} and the $\cM$-monotoncity of $\cG$. 

\begin{remark}\label{rem:MMM} Notice that if $\cF$ is an $\cM$-monotone subequation on $X$, then the fiber map defined $\Theta(x):= \cF_x$ for each $x \in X$ will be an $\cM$-monotone map in the sense of Definition \ref{defn:MMM}. However, this definition does not assume that $\Theta$ is the fiber map of an $\cM$-monotone subequation. Sufficient conditions which ensure that it is will be given in Theorem \ref{thm:MMS} below.
	\end{remark}

Now, under a mild {\bf {\em regularity condition}} on an $\cM$-monotone operator $F \in C(\cG)$ (with  $\cG$ fiberegular in the constrained case), the fiber map of the constraint set $\cF$ defined by the compatibility relation \eqref{relation} will be continuous.

\begin{thm}[Continuous $\call M$-monotone maps]\label{thm:CMM} Let $F \in C(\cG)$ be an $\cM$-monotone operator with either $\cG = \cJ^2(X)$ or $\cG \subsetneq \cJ^2(X)$ a fiberegular ($\cM$-monotone) subequation. Assume that the pair $(F, \cG)$ satisfies the following regularity condition: for some fixed $J_0 \in \intr \call M$, given $\Omega \ssubset X$ and  $\eta > 0$, there exists $\delta= \delta(\eta, \Omega) > 0$ such that 
	\begin{equation}\label{UCF}
	F(y, J + \eta J_0) \geq F(x, J) \quad \forall J \in \cG_x,\ \forall x,y \in \Omega \ \text{with}\ |x - y| < \delta.
	\end{equation}
	Then the $\cM$-monotone map $\Theta \colon X \to \scK(\call J^2)$ defined by
	\begin{equation}\label{Theta_def}
	\Theta(x) := \big\{ J \in \cG_x: \ F(x,J) \geq 0 \big\}
	\end{equation}
	is continuous.
\end{thm}

\begin{proof} 
	We will show that $\Theta$ is locally uniformly continuous. Since $\Theta$ is $\call M$-monotone, by \Cref{unifcontequiv} with fixed $J_0 \in \intr \call M$, it suffices to show that
	for every choice of $\Omega \ssubset X$ and $\eta > 0$ there exists $\delta_{\Theta}= \delta_{\Theta}(\eta, \Omega) > 0$ such that for each $x,y \in \Omega$
	\begin{equation}\label{LUC_Theta}
	\mbox{$|x -y| < \delta_{\Theta} \quad \implies \quad \Theta(x) + \eta J_0 \subset \Theta(y)$.}
	\end{equation}
	In the constrained case, where $\cG \subsetneq \cJ^2$ is a fiberegular $\cM$-monotone subequation, we have the validity of \eqref{LUC_Theta} with the fiber map $\Phi$ of $\cG$ in place of $\Theta$ for some $\delta_{\Phi} = \delta_{\Phi}(\eta, \Omega)$. It suffices to choose $\delta_{\Theta} = \min \{ \delta_{\Phi}, \delta \}$. Indeed, for each pair $x,y \in \Omega$ with $|x-y| < \delta_{\Theta}$, pick an arbitrary $J \in \Theta(x)$ so that $J \in \Phi(x)$ and $F(x,J) \geq 0$, which by the continuity of $\Phi$ and the regularity property \eqref{UCF} yields
	\begin{equation}\label{CT1}
	J + \eta J_0 \in \Phi(y) \quad \text{and} \quad F(y, J + \eta J_0) \geq F(x, J)  \geq 0,
	\end{equation}
	which yields the inclusion in \eqref{LUC_Theta}. 
	
	In the unconstrained case, where $\cG = \cJ^2(X)$, the constant fiber map $\Phi \equiv \call J^2$ is trivially continuous (\eqref{LUC_Theta} for $\Phi$ holds for every $\delta_{\Phi} > 0$) and hence it suffices to choose $\delta_{\Theta} = \delta$ and use the regularity condition \eqref{UCF}.
\end{proof}

\begin{remark}\label{rem:cont0} In Theorem \ref{thm:CMM}, the structural condition \eqref{UCF} on $F$ is merely sufficient to ensure that an $\call M$-monotone map $\Theta$ given by \eqref{Theta_def} is continuous. 
	The (locally uniform) continuity of $\Theta$ is equivalent to the statement that: for any fixed $J_0 \in \intr \call M$, given  $\Omega \ssubset X, $ and $\eta > 0$, there exists $\delta= \delta(\eta, \Omega) > 0$ such that $\forall x,y \in \Omega \ \text{with}\ |x - y| < \delta$ one has
	\begin{equation}\label{UCF_defn}
	F(x,J) \geq 0  \quad \text{and} \quad J \in \cG_x \quad \implies \quad F(y, J + \eta J_0) \geq 0.
	\end{equation}
	This condition is weaker, in general, than the structural condition \eqref{UCF} and hence useful to keep in mind for specific applications (see, for example, the proof of \cite[Theorem~5.11]{CP21} in a pure second order example). 
	On the other hand, the structural condition \eqref{UCF} can be more easily compared to other structural conditions on $F$ present in the literature.	
\end{remark}

\begin{remark}\label{rem:maps}
Notice that Theorem \ref{thm:CMM} is really a result about continuous $\cM$-monotone maps. In particular, we are not making use of the topological property (T) of $\cG$. In fact, one could state a version of the theorem where $\Phi$ is merely a continuous $\cM$-monotone map such that 
the $F \in C(\Phi(X))$ is $\cM$-monotone in the sense that
\begin{equation}\label{MMO2}
F(x,J + J') \geq F(x,J) \ \quad \forall \, x \in X, J \in \Phi(x), J' \in M. 
\end{equation}
The conclusion is that $\Theta: X \to \scK(X)$ defined by 
$$
	\Theta(x):= \{ J \in \Phi(x): \ F(x,J) \geq 0 \}
$$
is continuous. An approach of focusing merely on a background fiber map $\Phi$ (and not a background subequation $\cG$) was followed in the pure second order and gradient free cases in \cite{CP17} and \cite{CP21}. 
\end{remark}

Finally, making use of property (T) for a background subequation $\cG$ and natural non-degeneracy conditions, we have the following result. 

\begin{thm}[Fiberegular $\cM$-monotone subequations from $\cM$-monotone operators]\label{thm:MMS} Let $(F, \cG)$ be an $\call M$-monotone pair with $\cG$ fiberegular and $F$ which satisfies the regularity condition \eqref{UCF}. Then the constraint set $\cF$ defined by the correspondence relation \eqref{relation}; that is,
	\begin{equation}\label{relation1}
	\cF:= \{ (x,J) \in \cG: \ F(x,J) \geq 0 \},
	\end{equation} 
is a fiberegular $\cM$-monotone subequation. Moreover, the fibers of $\cF$ are non-empty if one assumes the non-empty condition \eqref{GammaNE}. Each fiber $\cF_x$ in not all of $\cJ^2$ in the constrained case and also in the unconstrained case if one assumes
	\begin{equation}\label{proper}
\big\{ J \in \call J^2: \ F(x,J) < 0 \big\} \neq \emptyset \quad \text{for each} \ x \in X.
\end{equation}

\end{thm}

\begin{proof} As already noted, $\cF$ defined by \eqref{relation1} will satisfy properties (P) and (N) with fiber map
	\begin{equation}\label{fiber_map}
	 \Theta(x):= \cF_x = \{ J \in \cG_x: \ F(x,J) \geq 0 \}, \ \ \forall \, x \in X
	 \end{equation} 
	 which is $\cM$-monotone and continuous (by Theorem \ref{thm:CMM}). Hence it only remains to show that $\cF$ satisfies property (T), which we recall is the triad 
	 \begin{equation}\tag{T1} 
	 \cF = \overline{\intr \cF};
	 \end{equation}
	 \begin{equation}\tag{T2} 
	 \cF_x = \overline{\intr \left( \cF_x \right)}, \ \ \forall \, x \in X;
	 \end{equation}
	 \begin{equation}\tag{T3} \left( \intr \cF \right)_x = \intr \left( \cF_x \right), \ \ \forall \, x \in X.
	 \end{equation}
The fiberwise property (T2), one can apply Proposition 4.7 of \cite{CHLP22} which says that (T2) holds provided that the fibers $\cF_x$ are closed and $\cM$-monotone. This leaves properties (T1) and (T3). It is not hard to see that if $\cF$ is closed, then properties (T2) plus (T3) imply (T1) (see Proposition \ref{(C)nec}). Hence for a $\cM$-monotone pair $(F, \cG)$, the constraint set $\cF$ defined by \eqref{relation1} will be a subequation if $\cF$ is closed and satisfies (T3). Moreover, since the inclusion $ \left( \intr  \cF \right)_x \subset \intr \left( \cF_x \right)$ is automatic for each $x \in X$, (T3) reduces to the reverse inclusion, which holds provided that $\cF$ is $\cM$-monotone and fiberegular in the sense of Defintion \ref{defn:fibereg}. This fact is proved in Proposition \ref{contt1}. Finally, by Theorem \ref{thm:CMM}, $\cF$ will be fiberegular if $\cG$ is fiberegular provided that $F$ satisfies the regularity condition \eqref{UCF}. 
\end{proof}


\section{Comparison principles for proper elliptic PDEs with directionality}\label{sec:examples}

In this section,  we present comparison principles for $\cM$-monotone operators by potential theoretic methods which combine monotonicity, duality and fiberegularity. A general comparison principle  will be presented which gives sufficient structural conditions on the operator $F$ which ensure that $F$ satisfies the correspondence principle (Theorem \ref{thm:corresp_gen}) with respect to some subequation constraint set $\cF$. The comparison principle for the operator $F$ will follow from the general comparison principle (Theorem \ref{thm:GCT}) satisfied by the subequation $\cF$. Representative examples will be given for the constrained case in Examples \ref{exe:OTE}, \ref{exe:KPO} and 
\ref{exe:hyp_poly}. As discussed in the introduction, we are primarily interested in examples will have gradient dependence in order to distinguish them from known examples the the pure second order and gradient-free cases that one finds in \cite{CP17} and \cite{CP21}, respectively. The needed monotonicity in the gradient variables is called {\em directionality}, which together with {\em proper ellipticity} is incorporated into the notion of $\cM$-monotonicity for some monotonicity cone such as $\cM(\gamma, \cD, R)$ where $\cD \subsetneq \R^n$ is a {\em directional cone}.

Throughout the section $\cM$ will be a constant coefficient monotonicity cone subequation and $X$ an open subset of $\R^n$.

\subsection{A general comparison principle for PDEs with sufficient monotonicity}\label{subsec:GCP}

We begin with the general result. 

\begin{thm}[Comparison principle for $\call M$-monotone PDEs]\label{thm:CP_MME} Let $F \in C(\cG)$ be an operator with domain either $\cG = \cJ^2(X)$ or $\cG \subsetneq \cJ^2(X)$ a fiberegular $\cM$-monotone subequation. Suppose that $F$ satisfies the following structural conditions 
	\begin{itemize}
		\item[(i)] $F$ is $\cM$-monotone 
		\begin{equation}\label{TCP1}
		F(x,J + J') \geq F(x,J) \quad \forall\, x \in X,\ J \in \cG_x,\ J'\in \cM;
		\end{equation}
		\item[(ii)] $F$ satisfies the regularity property \eqref{UCF}: for some fixed $J_0 \in \intr \call M$, for every $\Omega \ssubset X$ and for every $\eta > 0$, there exists $\delta= \delta(\eta, \Omega) > 0$ such that
		\begin{equation}\label{TCP2}
		F(y, J + \eta J_0) \geq F(x, J) \quad \forall J \in \cG_x,\ \forall x,y \in \Omega \ \text{with}\ |x - y| < \delta;
		\end{equation}
		\item[(iii)] $F$ satisfies the non-empty condition \eqref{GammaNE}:
		\begin{equation}\label{TCP3}
		\mbox{$\Gamma(x) \defeq \big\{ J \in \call \cG_x:  \ F(x,J) = 0 \big\} \neq \emptyset$ \quad for each $x \in X$.}
	\end{equation}
	\item[(iv)] $F$ is compatible with the subequation $\cF:= \{ (x,J)\in\cG:\ F(x,J) \geq 0\}$; that is,
		\begin{equation}\label{TCP4}
		\intr \cF = \{ (x,J) \in \cG: \ F(x,J) >0 \},
		\end{equation}
		or, equivalently $\partial \cF = \{ (x,J) \in\cG: \ F(x,J) = 0 \}$.
		\end{itemize}
Then, for every bounded domain $\Omega \ssubset X$ for which $\cM$ admits $\psi \in C^2(\Omega) \cap \USC(\overline{\Omega})$ which is strictly subharmonic on $\Omega$, the comparison principle for the equation $F(J^2u) = 0$ holds on $\overline{\Omega}$; that is, 
	\begin{equation}\label{TCP5}
	\mbox{$u \leq v$ on $\de \Omega \quad \implies \quad u \leq v$ in $\Omega$.}
	\end{equation}
	if $u \in \USC(\Omega)$ is a $\cG$-admissible viscosity subsolution of  $F(J^2u) = 0$ in $\Omega$ and  $v \in \LSC(\Omega)$ is a $\cG$-admissible viscosity supersolution of $F(J^2u) = 0$ in $\Omega$.

If, in addition, $\psi$ satisfies
$$
\psi \equiv-\infty \text { on } \partial \Omega \setminus \partial^{-} \Omega
$$
for some $\partial^{-} \Omega \subset \partial \Omega$, then
\begin{equation}\label{TCP6}
	\mbox{$u \leq v$ on $\de^{-} \Omega \quad \implies \quad u \leq v$ in $\Omega$}
	\end{equation}
	if $u \in \USC(\Omega)$ is a $\cG$-admissible viscosity subsolution of  $F(J^2u) = 0$ in $\Omega$ and  $v \in \LSC(\Omega)$ is a $\cG$-admissible viscosity supersolution of $F(J^2u) = 0$ in $\Omega$.
\end{thm}

\begin{proof} Since $(F, \cG)$ is an $\cM$-monotone operator-subequation pair by hypothesis  (with $\cG$ fiberegular in the constrained case), the regularity condition \eqref{TCP2} ensures that the  constraint set defined by $\cF:= \{ (x,J)\in\cG:\ F(x,J) \geq 0\}$ is a fiberegular $\cM$-monotone subequation by applying Theorem \ref{thm:MMS}. $\cF$ has non empty fibers by \eqref{TCP3}. In particular, $(F, \cG)$ is a proper elliptic ($\cM_0$-monotone) pair with $\cF$ a subequation. The compatibility condition \eqref{TCP4} then ensures that the correspondence principle of Theorem \ref{thm:corresp_gen} holds. Hence, the comparison principle \eqref{TCP5} for $\cG$-admissible sub/supersolutions of the equation $F(J^2u) = 0$ is equivalent to the comparison principle for  
	 $\cF$-sub/superharmonics. The existence of the strict approximator $\psi$ for $\cM$ on $\Omega$ then implies the needed potential theoretic comparison principle by Theorem \ref{thm:GCT}, which completes the proof in both cases.
\end{proof}

\subsection{Examples for both the elliptic and parabolic versions}\label{subsec:ell_par}

 We now give two illustrative examples for which the general comparison principle stated in Theorem \ref{thm:CP_MME} applies. One example for each of the two versions of the theorem. These will be all constrained cases. We will in particular focus our attention on equations of the form
\[
g(Du) F(x, Du, D^2 u) = f(x),
\]
which arise in many areas of mathematical analysis (as we will see below), and are currently the object of intense research activity, see for example \cite{BiDe, ImSil} and the references therein. 

Our first representative example will be treated using the first part (the elliptic version) of Theorem \ref{thm:CP_MME} which ``sees'' the entire boundary.

\begin{example}[Optimal transport]\label{exe:OTE} The equation
\begin{equation}\label{e:ot}
g(Du) \det(D^2 u) = f(x)
\end{equation}
arises in the theory of optimal transport, and describes the optimal transport plan from a source density $f$ to a target density $g$. Here, we assume that $f, g \in C(\overline{\Omega})$ are nonnegative, and we require that $g$ satisfies a (strict) directionality property with respect to some directional cone $\cD \subset \R^n$ (a closed convex cone with vertex at the origin and non-empty interior). More precisely, we assume that
\begin{equation}\label{g_ot1}
g(p + q) \ge g(p), \quad \text{for each} \ p,q \in \cD
\end{equation}
and that there exists $\bar q \in \intr \cD$ and a modulus of continuity $\omega : (0,\infty) \to (0,\infty)$ (satisfying $\omega(0^+) = 0$) such that   
\begin{equation}\label{g_ot2}
g(p + \eta \bar q) \ge g(p) + \omega(\eta), \quad \text{for each} \ p,q \in \cD \ \text{and each} \ \eta > 0.
\end{equation}

In other words, 
 $g$ needs to be increasing on $\cD$ in the directions of $\cD$ and strictly increasing in some direction $\bar q \in \intr \cD \subset \R^n$. Notice also that \eqref{g_ot2} implies that $g \not\equiv 0$ since $g(\bar{q}) \geq g(0) + \omega(1) > 0$.

Then, setting $\cM(\cD,\cP) = \R \times \cD \times \cP$, we have the following comparison principle.

\begin{prop}\label{prop:OTE} Let $f, g \in C(\overline{\Omega})$ be nonnegative, and assume that $g$ satisifes \eqref{g_ot1}-\eqref{g_ot2}. Then, the comparison principle holds for for $\cM(\cD,\cP)$-admissible sub/supersolutions of the equation \eqref{e:ot} on any bounded domain $\Omega$. 
\end{prop}

\begin{proof} It suffices to show that the first part of Theorem \ref{thm:CP_MME} (the elliptic version) can be applied to 
$$
	F(x,r,p,A) := g(p) \det(A) - f(x) \quad \text{on} \quad  \cG := \Omega \times \cM(\cD,\cP).
$$
 Assumptions 
  \textit{(i)}, \textit{(iii)} and \textit{(iv)} on $F$ are straightforward to check, which leaves the continuity property \textit{(ii)}. Pick $J_0 = (0, \bar q, I)$. Then, for every $\eta > 0$ and $J = (r,p,A) \in \cM(\cD, \cP)$, 
\begin{equation*} \begin{split}
F(y, J + \eta J_0 ) &= F(y, r, p + \eta \bar q, A + \eta I) = g(p + \eta \bar q) \det(A + \eta I) - f(y) \\ &\ge [g(p) + \omega(\eta)][ \det(A) + \eta^n ] - f(y) \ge
F(x,J) + f(x) - f(y) + \omega(\eta)\eta^n.
\end{split}\end{equation*}
Notice that $f(x) - f(y) + \omega(\eta)\eta^n \ge 0$ provided that $|x-y| \le \delta = \delta(\eta)$ and hence \textit{(ii)} holds.

Finally, in order to apply the first part of Theorem \ref{thm:CP_MME}, it remains to show that each $\Omega \ssubset X$ admits a strictly $\cM(\cD,\cP)$-subharmonic function, where $\intr \cM(\cD, \cP) = \R \times \intr \cD \times \intr \cP$. The function $\psi(x) := \langle \bar q , x \rangle + \alpha |x|^2$ satisfies
$$
D \psi(x) = \bar{q} + 2 \alpha x \in \intr \cD  
$$
provided that $\alpha = \alpha(\Omega)> 0$ is small enough and $D^2 \psi(x) = 2 \alpha I \in \intr \cP$ for all $\alpha > 0$.
\end{proof}

Note that, in view of Examples \ref{exe:CSE} and \ref{exe:Dmon}, $u$ is a $\cM(\cD,\cP)$-admissible subsolution of \eqref{e:ot} if and only if it is a subsolution of the PDE in the standard viscosity sense, 
it is convex on $\Omega$, and it is nondecreasing in the $\cD^\circ$-directions; that is,
\[
u(x) \ge u(x_0) \quad \text{for every $x, x_0 \in \Omega$ such that $[x_0, x] \subset \Omega,\, x-x_0 \in \cD^\circ$}.
\]
\end{example}

Our next representative example will be treated using the second part (the parabolic version) of Theorem \ref{thm:CP_MME} which ``sees'' only a reduced (parabolic) boundary. 

\begin{example}[Krylov's parabolic Monge-Amp\`ere operator]\label{exe:KPO} In \cite{Krylov76}, the following nonlinear parabolic equation is considered
\begin{equation}\label{e:kry}
-\partial_t u \det (D_x^2 u) = f(x,t), \quad   (x,t) \in \Omega \times (0,T) \subset \R^{n+1}.
\end{equation}
This equation is important in the study of deformation of surfaces by Gauss--Kronecker curvature and in Aleksandrov--Bakel'man-type maximum principles for (linear) parabolic 
equations. We have in this case a comparison principle with respect to the usual parabolic boundary of $\Omega$, for convex functions that are monotone nonincreasing in the $t$-direction. 

In preparation for the result, we introduce some notation as well as the relevant monotonicity cone subequation. As above, $(x,t) \in \R^{n+1} = \R^n \times \R$ will be used as global coordinates on the domain. We will also denote by $p = (p',p_{n+1}) \in \R^n \times \R$ in the first order part of the jet space. For matrices $A \in \cS(n+1)$, $A_n \in \cS(n)$ will denote the upper left $n \times n$ submatrix of $A$.

The relevant constant coefficient monotonicity cone subequation on $\R^{n+1}$ is 
\begin{equation}\label{MDPn}
\cM(\cD_n, \cP_n) := \{ (r,p,A) \in \R \times \R^{n+1} \times \cS(n+1):  \ p_{n+1} \le 0 \  \text{and} \  A_n \ge 0 \};
\end{equation}
that is, the $(n+1)$-th entry of $p$ is nonpositive and the $n\times n$ upper-left submatrix $A_n$ of $A$ is nonnegative. Notice that $\cM(\cD_n, \cP_n)$ is clearly a convex cone with vertex at the origin and nonempty interior, which ensures the topological property (T1). Negativity (N) is trivial as $\cM(\cD_n, \cP_n)$ is independent of $r \in \R$ and positivity (P) also holds. Hence $\cM(\cD_n, \cP_n)$ is a (constant coefficient) monotoncity cone subequation.

\begin{prop}\label{prop:KPO} Let $f \in C(\overline{\Omega} \times [0,T])$ be nonnegative with $\Omega \ssubset \R^n$ a bounded domain. 
Then, the parabolic comparison principle holds 
\[
\mbox{$u \leq v$ on $\left(\overline \Omega \times \{0\}\right) \cup \left(\partial \Omega \times (0, T)\right) \quad \implies \quad u \leq v$ in $\Omega \times (0,T)$.}
\]
for $\cM(\cD_n, \cP_n)$-admissible sub/supersolutions $u,v$ of  \eqref{e:kry}. 
\end{prop}

\begin{proof} As in the previous example, the idea is to apply Theorem \ref{thm:CP_MME}. This time we will make use of the parabolic version applied to
$$
	F((x,t), r,p,A) := -p_{n+1} \det(A_n) - f(x,t) \quad \text{on}  \quad \cG := \left(\Omega \times (0,T)\right) \times \cM(\cD_n, \cP_n).
$$ 
The assumptions \textit{(i)}, \textit{(ii)}, \textit{(iii)}, and \textit{(iv)} are checked in the same fashion as done in the previous example. 

Denoting by $X = \Omega \times (0,T)$, it remains to show that there exists $\psi \in C^2(X) \cap \USC(\overline{X})$ which is strictly $\cM(\cD_n, \cP_n)$-subharmonic on $X$ and satisfies 
$$
\psi \equiv-\infty \text { on } \partial X \setminus \partial^{-} X:= \left(\overline \Omega \times \{0\}\right) \cup \left(\partial \Omega \times (0, T)\right) 
$$
The function
\[
\psi(x,t) := |x|^2 - \frac1{T-t}
\]
is strictly $\cM(\cD_n, \cP_n)$-subharmonic, and $\psi(x,t) \to -\infty$ as $t \to T^-$ (that is, $\psi \equiv -\infty$ on $\overline \Omega \times \{T\}$), so we deduce the comparison principle in the form \eqref{cpm}.
\end{proof}

In view of Examples \ref{exe:CSE} and \ref{exe:Dmon}, $u$ is a $\cM(\cD_n, \cP_N)$-admissible subsolution of \eqref{e:ot} if and only if it is a subsolution of the equation in the standard 
viscosity sense, it is convex in the $x$ variable, and it is nondecreasing in the $ \{p_{n+1} \le 0\}^\circ$-directions, which means that it is nonincreasing in the $t$ variable.

Clearly, more general PDEs of the form
\[
(-\partial_t u) F (x,t,u,D_x u,D_x^2 u) = f(x,t),
\]
could be addressed in a similar way (under suitable monotonicity assumptions), as well as 
 ``standard'' parabolic  equations $\partial_t u = F (x,t,u,D_x u,D_x^2 u)$.
\end{example}

\subsection{Equations modelled on hyperbolic polynomials}\label{subsec:hyp_poly}

Next we present perhaps the simplest meaningful example of a G\aa rding-Dirichlet operator, which are defined via hyperbolic polynomials in the sense of G\aa rding, as mentioned in Remark \ref{rem:CC_UC}. We begin with the definition.

\begin{defn}\label{defn:hyp_poly}
 A homogeneous polynomial $\gf$ of degree $m$ on a finite dimensional real vector space $V$ is called {\em hyperbolic} with respect to a direction $a \in V$ if $\gf(a) > 0$ and if the one-variable polynomial $t \mapsto \gf(t a + x)$ has exactly $m$ real roots for each $x \in V$.
\end{defn}

  There are many examples of nonlinear PDEs that involve hyperbolic polynomials. The most basic example is the Monge-Amp\`ere operator where $\gf(A) = \det A$ for $A \in \cS(n)$ which is hyperbolic in the direction of the identity matrix $I$.  
A systematic study of the relationship between the G\aa rding theory of hyperbolic polynomials and pure second-order equations has been carried out in \cite{HLGarding} (see also \cite{CHLP22}). See also Section 11.6 of \cite{CHLP22}.
  
  In the following example, we observe that the theory of hyperbolic polynomials is flexible enough to cover equations on the whole 2-jet space, providing a natural notion of monotonicity. As before, we focus our attention on the gradient dependence.

\begin{example}\label{exe:hyp_poly}
On a bounded domain $\Omega \subset \R^2$, we consider the equation
\begin{equation}\label{foGarding}
u_x^2 - u_y^2 = 0,
\end{equation}
which builds upon perhaps the simplest hyperbolic polynomial $\gf(p_1, p_2) = p_1^2 - p_2^2$. Since $\gf$ is hyperbolic in the direction $e = (1,0) \in \R^2$, a general construction of G\aa rding yields a monotonicity cone for the operator $F(x,r,p,A) = \gf(p)$. In this example, the negatives of the two real roots $t \mapsto \gf(t (1,0) + p)$ can be ordered
\[
\lambda_1^{\gf}(p) := p_1 - |p_2| \leq p_1 + |p_2| := \lambda_2^{\gf}(p) 
\]
and are called the {\em G\aa rding $e$-eigenvalues} of $\gf$. Since $\gf(e)=1$ (which can always be arranged by normalization since $\gf(e) > 0$) one has
$$
\gf(p) = \lambda_1^{\gf}(p) \lambda_2^{\gf}(p),
$$
so that the first order differential operator defined by the degree two polynomial $\gf$ (which is $e$-hyperbolic) is the product of two {\em G\aa rding $e$-eigenvalues} of $\gf$. G\aa rding's theory also says that the closed {\em G\aa rding cone}
$$
	\overline{\Gamma} := \{ p \in \R^2: \ \lambda_1^{\gf}(p) := p_1 - |p_2| \geq 0\}
$$
must be convex and is characterized by the fact that its interior is the connected component of $\{\gf \neq 0\}$ which contains $e$ (both facts are clearly true here). Finally, since the closed convex cone with vertex at the origin $\overline{\Gamma} \subset \R^n$ has nonempty interior, 
\[
\cM_{\gf} := \R \times \big\{ p \in \R^2: \ \lambda_1^{\gf}(p) := p_1 - |p_2| \geq 0\} \times \cS(2),
\]
is a constant coefficient pure first order monotonicity cone subequation on $\R^2$. Moreover it is easy to check that $F$ is $\cM$-monotone on $\cF:= \R^2 \times \cM$. Finally, $F$ is compatible with $\cF$ by \cite[Proposition 2.6]{HLGarding} (which one can also check directly).

 It is then rather easy to produce strict $\cM_{\gf}$-subharmonics, and therefore to apply Theorem \ref{thm:CP_MME} to deduce a comparison result. We shall further specialize our setting to  $\Omega = (a,b) \times (c,d)$, where comparison on a reduced boundary is possible. To this end, take 
\[
\psi(x,y) = \frac{1}{a - x},
\]
which goes to $-\infty$ as $x \to a^+$, and satisfies 
\[
\lambda_1^{\gf}(D\psi(x,y)) = \psi_x(x,y) - |\psi_y(x,y)| = \frac{1}{(a-x)^2} > 0 \quad \text{on $\Omega$}.
\]
Therefore, we have the following statement.

\begin{prop}\label{prop:DGO} Let $\Omega = (a,b) \times (c,d)$. Then the comparison principle holds for $\cM_{\gf}$-admissible sub/supersolutions $u,v$ of  \eqref{foGarding}, that is, 
\[
\mbox{$u \leq v$ on $\partial \Omega \setminus \{x=a\} \quad \implies \quad u \leq v$ in $\Omega$.}
\]
\end{prop}

As in the previous examples, comparison principles for $u_x^2 - u_y^2 = f(x,y)$ could be deduced similarly. It is worth noting that equations of this form arise in the theory of zero-sum differential games. Though by no means general, we believe that this example well illustrates how the theory of hyperbolic polynomials and nonlinear potential theories may interact through a general notion of \noticina{slight linguistic change} monotonicity to yield comparison principles for a large class of nonlinear PDEs. To further emphasize the flexibility of G\aa rding theory, we notice that one can easily deduce results for inhomogeneous operators with a \textit{product structure} such as $F(x,r,p,A):= \gf_1(r)\gf_2(p)\gf_3(A) -  f(x)$, where $f \ge 0$ and $\gf_1, \gf_2, \gf_3$ are hyperbolic polinomials on $\R, \R^n, \cS(n)$ respectively. Indeed, each $\gf_i$ furnishes its own G\aa rding cone $\overline{\Gamma}_{i}$, and it is easy to check the monotonicity of $F$ with respect to $\cM := \overline{\Gamma}_1 \times \overline{\Gamma}_2 \times \overline{\Gamma}_3$.

\end{example}

\subsection{Examples of equations where standard structural conditions fail}\label{subsec:NSSC}

As a final consideration, we will present a class of proper elliptic operators with directionality for which our Theorem \ref{thm:CP_MME} applies to give the comparison principle, but for which the standard viscosity structural condition \cite[condition~(3.14)]{CIL92} on the operators fails to hold in general (see Proposition \ref{prop:fail} below). Simpler examples of variable coefficient pure second order operators have been discussed in \cite[Remark~5.10]{CP17}. 

The aforementioned condition (rewritten for $F(x,r,p,A)$ which is increasing in $A$, according to our convention) is
\begin{equation}\label{cil:3.14}
F(x,r,\alpha(x-y),A) - F(y,r,\alpha(x-y),B) \leq \omega(\alpha|x-y|^2 + |x-y|),
\end{equation}
for some modulus of continuity $\omega$, whenever $A,B \in \cS(n)$ satisfy
\begin{equation} \label{CABalpha}
-3\alpha \left(\begin{array}{cc}
I & 0 \\
0 & I 
\end{array}\right)
\leq 
\left(\begin{array}{cc}
A & 0 \\
0 & -B 
\end{array}\right)
\leq
3\alpha\left(\begin{array}{cc}
I & -I \\
-I & I 
\end{array}\right).
\end{equation}

\begin{example}[Perturbed Monge-Amp\`ere operators with directionality]\label{exe:PMAD}
On a bounded domain $\Omega \subset \R^n$, consider the operator defined by
\[
F(x,r,p,A) = F(x,p,A) \defeq \det\!\big(A + M(x,p)\big) - f(x), \ \ (x,r,p,A) \in \Omega \times \cJ^2
\]
with $f \in UC(\Omega; [0,+\infty))$ and with $M \in UC(\Omega \times \R^n; \call S(n))$ of the form
\begin{equation} \label{Mex1}
M(x,p) \defeq \pair{b(x)}{p} P(x)
\end{equation}
with $P \in UC(\Omega; \cP)$ and $b \in UC(\Omega; \R^n)$ such that
\begin{gather}
\label{condexcone}
\text{there exists a unit vector $\nu \in \R^n$ such that $\pair{b(x)}{\nu} \geq 0$ for each $x \in \Omega$.}
\end{gather}
Notice that the required uniform continuity for $f$ and $M$ holds if they are continuous on an open set $X$ for which $\Omega \ssubset X$. 

One associates to $F$ the candidate subequation with fibers
\[
\call F_x \defeq \big\{ J=(r,p,A) \in \call J^2 :\ A + M(x,p) \in \call P,\ F(x,J) \geq 0, \big\}, \quad x \in \Omega.
\]
which are clearly ($\R \times \{0\} \times \cP$)-monotone and hence one has properties (N) and (P) for $\cF$. In order to conclude that $\cF$ is indeed a subequation, by Theorem \ref{thm:MMSE}, it suffices to show that $\cF$ is fiberegular and $\cM$-monotone for some (constant coefficient) monotonicity cone subequation. To construct $\cM$ define the half-spaces
\[
H_{b(x)}^+ \defeq \big\{ q \in \R^n :\ \pair{b(x)}{q} \geq 0 \big\},
\]
and define the cone
\[
\call D \defeq \bigcap_{x \in \Omega} H^+_{b(x)} \neq \emptyset.
\]
This $\cD$ is a directional cone for $\call F$. Indeed $M(\cdot,p+q) \geq M(\cdot,p)$ for all $q \in \call D$. Hence $\cF$ is $\call M(\call D, \call P)$-monotone.

 Finally, $\call F$ is fiberegular, since it satisfies the third equivalent condition of \Cref{unifcontequiv}. In fact, for any fixed $J_0 = (r_0, p_0, A_0) \in \intr \call M(\call D, \call P)$ and any fixed $\eta > 0$ small,
\[\begin{split}
F(y,J+\eta J_0) - F(x,J) &= \det\!\big(A + \eta A_0 + M(y,p+\eta p_0)\big) - \det\!\big(A + M(y,p)\big) \\
&\quad + \det\!\big( A + M(y,p) \big) - \det\!\big( A + M(x,p) \big) \\
&\quad + f(x) - f(y) \\
\end{split}\]
is nonnegative if $|x-y|<\delta$, with $\delta = \delta(\eta) > 0$ sufficiently small, and thus $\call F_x + \eta J_0 \subset \call F_y$.

The comparison principle for $\cF$-admissible sub/supersolutions of the equation $F=0$ then follows from Theorem \ref{thm:CP_MME} since every bounded domain $\Omega$ admits a quadratic striclty $\call M(\call D, \call P)$-subharmonic function $\psi$ (as recalled in the proof of Proposition \ref{prop:OTE}) and $F$ satisfies the required conditions of the theorem (which we leave to the reader).
\end{example}

We now arrive to the main point of this subsection. While the comparison principle holds for the equation $F = 0$ of Example \ref{exe:PMAD}, if admissible ``perturbation coefficients'' $b(x)$ and $P(x)$ are chosen suitably, the standard viscosity condition \eqref{cil:3.14} fails to hold. For simplicity we give an example in dimension two, but generalizations are clearly possible.

\begin{prop}\label{prop:fail}
For some $x_0 \in \Omega \ssubset \R^2$, suppose that the perturbation coefficients $b \in UC(\Omega; \R^2)$ and $P \in UC(\Omega; \cP)$ satisfy:
\begin{equation}\label{bad_b}
\mbox{ $b$ has an isolated zero of order $\beta \in (0,3)$ at $x_0$;}
\end{equation}
\begin{equation}\label{bad_P}
\forall \, x \in \Omega, \quad P(x) \defeq \left(\begin{array}{cc}
h(x) & 0 \\
0 & 0
\end{array}\right) \quad \text{with} \quad h(x):= \frac{6g^2(x)}{|\pair{b(x)}{x_0 - x\,}|+g^2(x)},
\end{equation}
where $g \in UC(\Omega)$ is nonnegative and satisfies
\begin{equation}\label{bad_g}
\mbox{ $g$ has an isolated zero of order $\gamma \in \left(\frac{\beta + 1}{2},2\right)$ at $x_0$.}
\end{equation}
Then the condition \eqref{cil:3.14} fails to hold.
	\end{prop}
Before giving the proof, we should note that the function $h$ in \eqref{bad_P} is not actually defined in $x = x_0$, but since $g^2$ has an isolated zero of order $2 \gamma > \beta + 1 > 1$, $h$ extends continuously by defining $h(x_0) = 0$.
\begin{proof} The idea is to exploit the order of the isolated zeros of $b$ and $g$ in $x_0$ to take a suitable sequence $\{y_n\} \subset \Omega$ converging to $x_0$, along which one can find sequences of matrices $\{A_n\}$ and $\{B_n\}$ satisfying \eqref{CABalpha} which contradict the validity of the inequality\eqref{cil:3.14}.

Consider any sequence $\{y_n\}_{n \in \N} \subset \Omega$ such that $y_n \to x_0$ with $b(y_n) \neq 0 \in \R^2$ ($b$ has an isolated zero in $x_0$) and such that
\[
\pair{b(y_n)}{x_0 - y_n} > 0 \quad \forall n \in \N.
\]
Such a choice is possible thanks to condition~\eqref{condexcone}. The desired matrices are defined by
\[
2A_n = B_n \defeq \left(\begin{array}{cc}
0 & 0 \\
0 & g(y_n)^{-1}
\end{array}\right), \quad n \in \N.
\]
Each pair $A_n, B_n$ satisfies (\ref{CABalpha}) with $\alpha = \alpha_n \defeq (3g(y_n))^{-1}$, as one easily verifies. 

By contradiction, assume that \eqref{cil:3.14} holds. Along the sequence $(y_n, A_n, B_n)$ one would have, as $y_n \to x_0$,
\[
F(y_n, \alpha_n(x_0-y_n), A_n) - F(x_0,\alpha_n(x_0-y_n), B_n) \leq \omega\!\left( \frac{|y_n-x_0|^2}{3g(y_n)} + |y_n - x_0| \right) \longrightarrow 0,
\]
but
\begin{equation*}\begin{split}
&F(y_n, \alpha_n(x_0-y_n), A_n) - F(x_0,\alpha_n(x_0-y_n), B_n) \\
& \qquad = \left|\begin{array}{cc}
\frac13 g(y_n)^{-1} \pair{b(y_n)}{x_0-y_n} h(y_n) & 0 \\
0 & \frac12 g(y_n)^{-1}
\end{array}\right|
- f(y_n)
- \left|\begin{array}{cc}
0 & 0 \\
0 & g(y_n)^{-1}
\end{array}\right|
+ f(x_0) \\
&\qquad = \frac{\pair{b(y_n)}{x_0-y_n}}{\pair{b(y_n)}{x_0-y_n} + g(y_n)^2} + o(1) \longrightarrow 1,
\end{split}\end{equation*}
thus leading to a contradiction.
\end{proof}

\begin{appendix}

\section{Monotonicity, fiberegularity and topological stability}\label{AppA}

Given a fiberegular subequation $\cF \subset \cJ^2(X) = X \times \cJ^2$ on an open set $X$ in $\R^n$ which is $\cM$-monotone for some constant coefficient monotonicity cone subequation $\cM$, one knows that the fiber map
$$
\Theta \colon (X, |\cdot|) \to \call (\scK(\call J^2), d_{\scr H})\ \ \text{defined by} \ \  \Theta(x):= \cF_x,\ \ x \in X
$$
is continuous (taking values in the closed subsets $\scK(\cJ^2)$) and $\cM$-monotone in the sense that
\begin{equation}\label{MM_Theta}
\Theta(x) + \cM \subset \Theta(x), \ \ \forall \, x \in X.
\end{equation}
We address here the converse; that is, given a continuous $\cM$-monotone map 
\begin{equation}\label{MMM}
\Theta \colon (X, |\cdot|) \to \call (\scK(\call J^2), d_{\scr H}),
\end{equation}
is it true that
\begin{equation}\label{F_is_SE} 
\cF \subset \cJ^2(X) \ \text{with} \ \cF_x := \Theta(x) \ \text{for all} \ x \in X \ \ \Longrightarrow \ \ \cF \ \text{is a subequation}? 
\end{equation}
The question is 
important in light of the Correspondence Principle of Theorem \ref{thm:corresp_gen}; that is, given a proper elliptic pair $(F, \cG)$ one wants to know whether 
$\cF$ having fibers 
$$
	\cF_x = \{ J \in \cG_x \subset \cJ^2: F(x,J) \geq 0 \}
$$ 
implies that $\cF$ is a subequation (and hence a well developed potential theory). 

If $\Theta$ is a continuous $\cM$-monotone map, then by definition the fibers of $\cF_x := \Theta(x)$ are closed and one has $\cM$-monotonicity
$$
	\cF_x + \cM \subset \cF_x , \ \ \forall \, x \in X,
$$
which implies properties (P) and (N) since $\cM_0 := \cN \times \{0\} \times \cP \subset \cM$. This leaves the topological property (T), which we recall is the triad
\begin{equation}\label{T1} \tag{T1}
\cF = \overline{\intr \cF};
\end{equation}
\begin{equation}\tag{T2} \label{T2}
\cF_x = \overline{\intr \left( \cF_x \right)}, \ \ \forall \, x \in X;
\end{equation}
\begin{equation}\tag{T3} \label{T3}
\left( \intr \cF \right)_x = \intr  \left( \cF_x \right), \ \ \forall \, x \in X.
\end{equation}

In the case of constant coefficient subequations, the triad reduces to \eqref{T1}; that is, to $\call F$ being a regular closed set, and it is known~\cite{HL09, CHLP22} that such condition is equivalent to the reflexivity of the Dirichlet dual
\[
\tildee{\call F} \defeq (-\intr\call F)\compl.
\]
This is essentially a consequence of the fact that, for any open set $\call O$, the set $\call S= \bar{\call O}$ is regular closed. 

When $\call F$ has variable coefficients, as noted in~\cite{HL11}, the two conditions \eqref{T2}-\eqref{T3} involving the fibers are useful in order to be allowed to compute the dual fiberwise; that is, in order to have the following equality:
\[
\tildee{\call F } = \bigsqcup_{x\in X}(-\intr \call F_x)\compl \defeq \bigcup_{x\in X} \{x\} \times (-\intr \call F_x)\compl.
\]
Therefore it is easy to see that the full topological condition (T) (that is, conditions \eqref{T1}-\eqref{T3} together) yields the reflexivity of Dirichlet duals of variable coefficient subequations as well.

Some facts concerning the topological conditions are in order.

\begin{prop} \label{int=int}
Let $\call F \subset X \times \call J^2$; then \eqref{T3} holds if and only if
\begin{equation} \tag{T3$'$}\label{T3'}
\intr \call F = \bigsqcup_{x \in X} \intr \call F_x.
\end{equation}
\end{prop}

\begin{proof}
Equality\eqref{T3'} straightforwardly implies condition \eqref{T3}. For the converse implication, one first notes that the inclusion
\begin{equation} \label{trivial}
(\intr \call F)_x \subset \intr (\call F_x),
\end{equation}
always holds, so that
\begin{equation} \label{2=sup}
(\intr \call F)_x = \intr (\call F_x) \quad\iff\quad (\intr \call F)_x \supset \intr (\call F_x);
\end{equation}
furthermore, since
\begin{equation*} \label{xifx}
\{x\} \times (\intr \call F)_x = \intr \call F \cap \left(\{x\} \times  \call F_x \right)   =  \left( \intr \bigsqcup_{y\in X} \call F_y\right) \cap \big(\{x\} \times  \call F_x \big),
\end{equation*}
we have that
\begin{equation} \label{precond2}
(\intr \call F)_x \supset \intr\call F_x \quad\iff\quad \intr \left( \bigsqcup_{y\in X} \call F_y \right) \supset \{x\} \times \intr (\call F_x).
\end{equation}
Hence, combining (\ref{2=sup}) and (\ref{precond2}) yields
\begin{equation} \label{cond2}
(\intr \call F)_x = \intr  (\call F_x) \quad\iff\quad \intr \left( \bigsqcup_{y\in X} \call F_y \right) \supset \{x\} \times \intr (\call F_x).
\end{equation}
By (\ref{cond2}), one has the inclusion $\supset$ in \eqref{T3'}, while the opposite one is trivial by (\ref{trivial}).
\end{proof}

\begin{prop} \label{(C)nec}
Let $\call F \subset X \times \call J^2$ be closed. Then \emph{
\[
\eqref{T2} \ \text{and} \ \eqref{T3} \implies \eqref{T1}. 
\]
}
\end{prop}

\begin{proof}
On always has the inclusion $\bar{\intr \call F} \subset \call F$. On the other hand, assuming \eqref{T2} and \eqref{T3},
\[
\call F = \bigsqcup_{x \in X} \call F_x = \bigsqcup_{x\in X} \bar{\intr (\call F_x)} \subset \bar{\bigsqcup_{x\in X} {\intr (\call F_x)} } = \bar{\intr \call F},
\]
where we also used \Cref{int=int} for the last equality.
\end{proof}

This shows that an equivalent formulation of property (T) would be to ask that $\cF$ is closed and that \eqref{T2} and \eqref{T3} hold. 

As for \eqref{T2} and \eqref{T3}, it is easy to see that, in general, closed $\call M_0$-monotone subsets of $X \times \call J^2$ do not satisfy it. However, if one has more monotonicity, then that could be enough in order to guarantee 
\begin{equation}\tag{T2}
\call F_x =\bar{\intr (\call F_x)} \qquad \forall x \in X.
\end{equation}
For instance, the following holds.

\begin{prop} \label{mont2}
Suppose that $\call F \subset X \times \call J^2$ has closed fibers and suppose that there exists a subset $\call M \subset X \times \call J^2$ where $\cM$ satisfies property \eqref{T2} and 
\[
\call F + \call M = \call F.
\]
Then $\call F$ satisfies property \eqref{T2}.
\end{prop}

\begin{proof}
One has
\[
\call F_x = \call F_x + \call M_x = \call F_x + \bar{\intr (\call M_x)} \subset \bar{\call F_x + \intr (\call M_x)} \subset \bar{\intr (\call F_x + \call M_x)} = \bar{\intr (\call F_x)} \subset \call F_x,
\] 
hence all the inclusions (and in particular the last one) are in fact equalities.
\end{proof}

\begin{remark} \label{rmkmont2}
A situation in which the hypotheses of \Cref{mont2} are satisfied is that of $\call F$ being $\call M$-monotone for some regular closed subset $\call M \subset \call J^2$ such that $0 \in \call M$. For example, this holds if $\cF$ is $\cM$ monotone for a constant coefficient monotonicity cone subequation. 
\end{remark}

Condition \eqref{T3} requires a little more attention because it is the only one that relates the interior with respect to $X \times \call J^2$ and the interior with respect to $\call J^2$. To stress this fact, let us write
\begin{equation}\tag{T3} \label{tau11}
\big(\intr_{X \times \call J^2} \call F\big)_x = \intr_{\call J^2} \call F_x.
\end{equation}

This condition seems to be related to some sort of continuity of the fiber $\call F_x$ with respect to the point $x$. Here we prove that if $\cF \subset \cJ^2(X)$ has fibers determined by a continuous $\cM$-monotone map $\Theta$, then $\cF$ satisfies  \eqref{T3}.

\begin{prop} \label{contt1}
	Suppose that $\cF \subset  \cJ^2(X)$ has fibers $\cF_x := \Theta(x)$ where the fiber map $\Theta \colon (X, |\cdot|) \to \call (\scK(\call J^2), d_{\scr H})$ is continuous and $\call M$-monotone for some constant coefficient monotonicity cone subequation $\call M$. Then $\cF$ satifies the topological property \eqref{T3}.
\end{prop}

\begin{proof}
By (\ref{cond2}), it suffices to prove that 
\begin{equation} \label{suff1}
\{x\} \times \intr \Theta(x) \subset \intr \bigsqcup_{y\in X} \Theta(y) \qquad \forall x \in X. 
\end{equation} 
Fix $x \in X$ and $J_x \in \intr \Theta(x)$ and let $\rho > 0$ such that $\call B_{2\rho}(J_x) \subset \Theta(x)$, where $\call B$ denotes the ball in $\call J^2$ with respect to the norm $\trinorm\cdot$. Let $J'_x \in \call B_\rho(J_x)$, so that 
\[
J'_x - \rho J_0 \subset \call B_{2\rho}(J_x) \subset \Theta(x),
\]
for some $J_0 \in \intr \call M$ fixed. By \Cref{unifcontequiv}{\it(c)}, there exists $\delta > 0$ such that
\[
J'_x = \left( J'_x - \rho J_0 \right) + \rho J_0 \in \Theta(y) \qquad \forall y \in B_\delta(x).
\]
This proves that
\[
B_\delta(x) \times \{ J'_x \} \subset \bigsqcup_{|y-x|<\delta} \Theta(y) \qquad \forall J'_x \in \call B_\rho(J_x),
\]
hence
\[
B_\delta(x) \times \call B_\rho(J_x) \subset \bigsqcup_{|y-x|<\delta} \Theta(y).
\]
It follows that
\[
(x,J_x) \in B_\delta(x) \times \call B_\rho(J_x) \subset \intr \bigsqcup_{|y-x|<\delta} \Theta(y),
\]
and since $J_x \in \intr\Theta(x)$ is arbitrary, this proves that
\[
\{x \} \times \intr \Theta(x) \subset \intr \bigsqcup_{|y-x|<\delta} \Theta(y) \subset \intr \bigsqcup_{y \in X} \Theta(y), 
\]
and thus, since $x \in X$ is arbitrary, (\ref{suff1}) follows, as desired.
\end{proof}

Finally, the continuity of $\Theta$ also implies that $\call F$ is closed.
\begin{prop} \label{contclos}
Suppose that there exists $\Theta$ as above (not necessarily $\call M$-monotone). Then $\call F$ is closed (in $X \times \call J^2$).
\end{prop}

\begin{proof}
Let $(\bar x, \bar J) \in \bar{\call F}$. Then $\bar x \in X$ and there exist sequences $x_k \to \bar x$ and $J_k \to \bar J$ such that $J_k \in \Theta(x_k)$ for all $k \in \N$. Since by continuity $\Theta(\bar x) = \lim_{x_k \to \bar x} \Theta(x)$, where the limit is computed with respect to the Hausdorff distance $d_{\scr H}$, it is known (cf.~\cite[Exercise~7.4.3.1]{BBI01}) that 
\[
\Theta(\bar x) = \{ J' \in \call J^2 :\ \exists \{J'_k\}_{k\in\N}\ \text{such that}\ J'_k \in \Theta(x_k)\ \forall k \in \N\ \text{and}\ J'_k \to J'\}.
\]
Hence $\bar J \in \Theta(\bar x)$, yielding $(\bar x, \bar J) \in \call F$.
\end{proof}

We now can affirm that the answer to the question \eqref{F_is_SE} is yes.

\begin{thm}\label{thm:MMSE}
Let $\Theta$ be a continuous and $\call M$-monotone map on $X$, for some constant coefficient monotonicity cone subequation $\call M$; define
\begin{equation*} \label{ffromth}
\call F \defeq \bigsqcup_{x\in X} \Theta(x).
\end{equation*}
Then $\call F$ is an $\call M$-monotone subequation on $X$.
\end{thm}

\begin{proof}
By definition, $\call F$ is $\call M$-monotone; that is, $\call F + \call M \subset \call F$. Also, $\call F$ has nontrivial, and closed, fibers. Therefore the proof now amounts to showing that $\call F$ satisfies the triad of topological properties. By \Cref{mont2} and \Cref{rmkmont2}, $\call F$ satisfies \eqref{T2}, by \Cref{contt1}, $\call F$ satisfies \eqref{T3}; this means that $\call F$ satisfies \eqref{T2} and \eqref{T3}.  By \Cref{contclos}, $\call F$ is closed, and thus, by \Cref{(C)nec}, $\call F$ satisfies\eqref{T1} as well.
\end{proof}

\section{Some basic tools in nonlinear potential theory}\label{AppB}

In this appendix, we collect some foundational results which lie at the heart of our methods and which form, along with the Almost Everywhere Theorem~\cite[Theorem 4.1]{HL16}, the ``basic tool kit of viscosity solution techniques'' in \cite{CHLP22}: the Bad Test Jet \Cref{l:btj} and the Definitional Comparison \Cref{defcompa}. Let us highlight that these two tools  will be here stated in a variable coefficient setting, while in \cite{CHLP22} they are proved for constant coefficient subequations. The proofs of these results in the variable coefficient setting will be given in \cite{PR22}, which involve some ``bland adjustments''of the constant coefficient proofs of \cite {HP22}. In particular, they not not require fiberegularity. We will also recall  some elementary properties of the set $\call F(X)$, of all $\call F$-subharmonics on $X$ known from \cite{HL09}, whose proofs are somewhat reformulated in \cite{PR22} making  more explicit use of the definitional comparison of \Cref{defcompa}.

We begin with the first tool which is very useful when one seeks to check the validity of subharmonicity at a point by a contradiction argument. More precisely, if $u$ fails to be subharmonic at a given point, then one must have the existence of a \emph{bad test jet} at that point, as stated in the following lemma. This criterion is essentially the contrapositive of the definition of viscosity subsolution, when one takes \emph{strict} upper contact quadratic functions as upper test functions (see \cite[Lemma 2.8 and Lemma C.1]{CHLP22}). 

\begin{lem}[Bad Test Jet Lemma] \label{l:btj}
Given $u \in \USC(X)$, $x \in X$ and $\call F_x \neq \emptyset$, suppose $u$ is not $\call F$-subharmonic at $x$. Then there exists $\epsilon > 0$ and a $2$-jet $J \notin \call F_x$ such that the (unique) quadratic function $\phi_J$ with $J^2_x\phi_J = J$ is an upper test function for $u$ at $x$ in the following $\epsilon$-strict sense:
\begin{equation} \label{btj:i}
u(y) - \phi_J(y) \leq -\epsilon|y-x|^2 \quad \text{$\forall y$ near $x$ (with equality at $x$)}.
\end{equation}
\end{lem}

The second tool is a \emph{comparison principle} whose validity characterizes the $\call F$-subharmonic functions for a given subequation $\call F$. It states that comparison holds if the function $z$ in (ZMP) is the sum of a $\call F$-subharmonic and a $C^2$-smooth and strictly $\tildee{\call F}$-subharmonic. 
It is called \emph{definitional comparison} because it relies only upon the ``good'' definitions the theory gives for $\cF$-subharmonics and for subequations $\cF$ (which include the negativity condition (N) that is important in the proof). It was stated and proven in a context of constant coefficient subequations in~\cite[Lemma 3.14]{CHLP22}. 

\begin{lem}[Definitional Comparison] \label{defcompa}
Let $\call F$ be a subequation and $u \in \USC(X)$.
\begin{itemize}
	\item[(a)] If $u$ is $\cF$-subharmonic on $X$, then the following form of the comparison principle holds for each bounded domain $\Omega \ssubset  X$:
\begin{equation}\label{DCP}
\left\{ \begin{array}{c}
\mbox{ $u + v \leq 0$ on $\partial \Omega \Longrightarrow u + v \leq 0$ on $\Omega$} \\
\\
\mbox{if $v \in \USC(\overline{\Omega}) \cap C^2(\Omega)$ is strictly $\wt{\cF}$-subharmonic on $X$.} \end{array} \right.
\end{equation}
With $w:= -v$ one has the equivalent statement
\begin{equation}\label{DCP2}
\left\{ \begin{array}{c} \mbox{ $u \leq w$ on $\partial \Omega \Longrightarrow u \leq w$ on $\Omega$} \\
\\
\mbox{if $w \in \LSC(\overline{\Omega}) \cap C^2(\Omega)$ with $J_x^2 w \not\in \cF$ for each $x \in \Omega$.} \end{array} \right.
\end{equation}
(That is, for $w$ which are regular and strictly $\cF$-superharmonic in $X$.)
\item[(b)] Conversely, suppose that for each $x_0 \in X$ there there exist arbitrarily small balls $B$ about $x_0$ where the form of comparison of part (a) holds with $\Omega = B$. Then $u$ is $\cF$-subharmonic on $X$. Moreover, it is enough to consider quadratic $v$ or $w$.
\end{itemize} 
\end{lem}

\begin{remark}[Applying the definitional comparison] \label{appldefcompa}
Sometimes it is useful to prove the contrapositive of the form of comparison in part (a) of \Cref{defcompa} in order to conclude subharmonicity. That is to say, in order to show by (b) that $u$ is subharmonic on $X$ one proves that, for each $x \in X$ there is a neighborhood $\Omega \ssubset X$ of $x$ where
\begin{equation} \label{appdefcompa}
(u+v)(x_0) > 0 \ \text{for some $x_0 \in \Omega$} \ \implies \ (u+v)(y_0) > 0 \ \text{for some $y_0 \in \de\Omega$}
\end{equation}
for every $v \in \USC(\bar \Omega) \cap C^2(\Omega)$ which is strictly $\tildee{\call F}$-subharmonic on $\Omega$. Conversely, one can also infer that the implication (\ref{appdefcompa}) holds whenever one knows that $u$ is subharmonic on $X$. In situations where we are interested in proving the subharmonicity of a function which is somehow related to a given subharmonic, this helps to close the circle (for example, see the proofs of \Cref{thm:UTP}  or \Cref{elemprop}).
\end{remark}

The last tool is a collection of elementary properties shared by functions in $\call F(X)$, the set of $\call F$-subharmonics on $X$. They are to be found in \cite[Section~4]{HL09} for pure second-order subequations, in \cite[Theorem 2.6]{HL11} for subequations on Riemannian manifolds, in \cite[Proposition D.1]{CHLP22} for constant-coefficient subequations. By invoking the Definitional Comparison \Cref{defcompa} one can perform most of the proofs along the lines of those of Harvey--Lawson \cite{HL09}. This is done in \cite{PR22}. 
More precisely, one uses the definitional comparison in order to make up for the lack, for arbitrary subequations, of a result like \cite[Lemma 4.6]{HL09}.

\begin{prop}[Elementary properties of $\call F(X)$] \label{elemprop}
	Let $X \subset \R^n$ be open. For any subequation $\call F$ on $X$, the following hold:
	\begin{itemize}[align=left, leftmargin=*, left=11pt, itemsep=4.5pt]
		\item[(i:] \!\emph{local property}) \	$u \in \USC(X)$ locally $\call F$-subharmonic $\iff$ $u \in \call F(X)$;
		\item[(ii:] \!\emph{maximum property}) \	$u,v \in \call F(X)$ $\implies$ $\max\{u,v\}\in \call F(X)$;
		\item[(iii:] \!\emph{coherence property}) \ if $u \in \USC(X)$ is twice differentiable at $x_0\in X$, then
		\[
		\text{$u$ $\call F$-subharmonic at $x_0$}\ \iff\ \text{$\call J^2_{x_0} u \in \call F_{x_0}$;}
		\]
		\item[(iv:] \!\emph{sliding property}) \ $u\in \call F(X)$ $\implies$ $u-m \in \call F(X)$ for any $m >0$;
		\item[(v:] \!\emph{decreasing sequence property}) \ $\{u_k\}_{k\in \mathbb{N}} \subset \call F(X)$ decreasing $\implies$ $ \lim_{k\to\infty} u_k \in \call F(X)$;
		\item[(vi:] \!\emph{uniform limits property}) \ $\{u_k\}_{k\in \mathbb{N}} \subset \call F(X)$, $u_k \to u$ locally uniformly $\implies$ $u \in \call F(X)$;
		\item[(vii:] \!\emph{families-locally-bounded-above property}) \ if $\scr F \subset \call F(X)$ is a family of functions which are locally uniformly bounded above, then the upper semicontinuous envelope $u^*$ of the Perron function $u(\,\cdot\,) \defeq \sup_{w \in \scr F} w(\,\cdot\,)$ belongs to $\call F(X)$.\footnote[2]{Recall that the \emph{upper semicontinuous envelope} of a function $g$ is defined as the function
			\[
			g^\ast (x) \defeq \lim_{r\dto 0}\sup_{y \in B_r(x)} g(y).
			\]
			It is immediate to see that the \emph{upper semicontinuous envelope operator} ${}^* \colon g \mapsto g^*$ is the identity on the set of all upper semicontinuous functions. Also, we called \emph{Perron function} the upper envelope of the family $\scr F$, since $\scr F$ is a family of subharmonics.}
	\end{itemize}
	Furthermore, if $\call F$ has constant coefficients, the following also holds:
	\begin{itemize}[align=left, leftmargin=*, left=11pt, itemsep=2.5pt]
		\item[(viii:] \!\emph{translation property}) \ $u \in \call F(X)$ $\iff$ $u_y\defeq u(\cdot - y) \in \call F(X + y)$, for any $y \in \R^n$.
	\end{itemize}
\end{prop}

\section{Some facts about the Hausdorff distance}\label{AppC}

We briefly recall a few facts about the Hausdorff distance which we have used in the discussion of fiberegularity in Subsection \ref{sec:FR}. The reader can consult \cite{BBI01} for further information.

\begin{definition}
	Let $(M, d)$ be a metric space.
	\begin{enumerate}[label=(\roman*)]
		\item Given $\emptyset \neq A,B \subset M$, one defines the \emph{excess of $A$ over $B$} by
		\begin{equation} \label{EXdef}
		\mathrm{ex}_B(A) \defeq \sup_{A} \mathrm{dist}(\,\cdot\,, B) = \sup_{a\in A} \inf_{b\in B} d(a,b) \in [0,+\infty];
		\end{equation}
		in addition, one defines
		\begin{equation} \label{EXA0}
		\mathrm{ex}_\emptyset(A) \defeq +\infty, \quad \mathrm{ex}_A(\emptyset) \defeq 0 \qquad \text{for each nonempty}\ A \subset M,
		\end{equation}
		and
		\begin{equation} \label{EX00}
		\mathrm{ex}_\emptyset(\emptyset) \defeq 0.
		\end{equation}
		
		\item The {\em Hausdorff distance} on the power set $\scr P(M)$ is the map $d_{\scr H} \colon \scr P(M)^2 \to [0,+\infty]$ defined by
		\begin{equation} \label{defhaus}
		d_{\scr H}(A,B) \defeq \max\big\{\mathrm{ex}_B(A), \mathrm{ex}_A(B) \big\}.
		\end{equation}
	\end{enumerate}
\end{definition}

\begin{remark}
	It is easy to see that $d_{\scr H}(A,\bar A) = 0$ for any $\emptyset \neq A \subset M$, and that the quotient $\scr P(M)/d_{\scr H}$, determined by the relation $A\sim B$ if and only if $d_{\scr H}(A,B) = 0$, is naturally identified with the space $\frk K(M)$ of all closed subsets of $M$. Furthermore, one can prove $(\frk K(M), d_{\scr H})$ is a metric space, and it is complete (or compact) if $M$ is.
\end{remark}

\begin{remark} \label{hausinf}
	We will make use of the following straightforward properties of the Hausdorff distance: for $A \subsetneqq B$,
	\[
	d_{\scr H}(A,B) = +\infty
	\]
	whenever
	\[
	\text{either} \quad A = \emptyset \quad \text{or} \quad
	\text{$A$ bounded and $B$ is unbounded}.
	\]
\end{remark}

In addition, if we consider $(M,d) = (\call J^2, \trinorm\cdot)$, we know that the Hausdorff distance is infinite in another case as well.

\begin{lem} \label{lem:hausinfmon}
	One has
	\[
	d_{\scr H}(\call E, \call J^2) = +\infty \quad \forall \call E \ \text{$\call M$-monotone}.
	\]
\end{lem}

\begin{proof}
	It suffices to show that $\call E\compl$ contains balls of arbitrarily large radius, so that no finite enlargement of $\call E$ can exhaust $\call J^2$. Note that by the definition of the Dirichlet dual, 
	\[
	\call E\compl = - \intr \tildee{\call E},
	\]
	therefore $\call E\compl$ contains an open ball about some element of $- \tildee{\call E}$; without loss of generality, we may suppose that this is $(0,0,0)$, since translations by a fixed jet preserve proper ellipticity and directionality. Now, one knows that
	\[
	\tildee{\call E} + \call M \subset \tildee{\call E}
	\]
	and since $(0,0,0) \in \tildee{\call E}$ we have
	\[
	\call M \subset \tildee{\call E},
	\]
	yielding
	\[
	-\intr \call M \subset \call E\compl.
	\]
	At this point, it suffices to show that $\intr \call M$ contains balls of arbitrarily large radius. To see this, fix $J_0 \in \intr{\call M}$ and without loss of generality suppose that $\call B_1(J_0) \subset \call M$;\footnote{Since $J_0 \in \intr \call M$, there exists $\rho > 0$ such that $\call B_\rho(J_0) \subset \call M$; since $\call M$ is a cone with vertex at the origin, $\call B_1(\rho^{-1}J_0) = \rho^{-1} \call B_\rho(J_0) \subset \rho^{-1}\call M = \call M$, therefore it suffices to replace $J_0$ by $\rho^{-1}J_0$.} note that one has $t J_0 \in \intr \call M$ for any $t>0$ and
	\[
	\call B_t(t J_0) \subset \intr \call M \quad \forall t>0. \qedhere
	\]
\end{proof}

\end{appendix}

\newtheoremstyle{Claim}{}{}{\itshape}{}{\itshape\bfseries}{:}{ }{#1}
\theoremstyle{Claim}
\newtheorem{ack}{Acknowledgment}  

\begin{ack}
	The authors express their gratitude to Reese Harvey for numerous useful conversations which contributed to the development of this work.
\end{ack}

\end{document}